\documentclass[10pt,twoside, a4paper, english, reqno]{amsart}
\usepackage{amsmath,amsthm,amssymb,enumerate}
\usepackage{fullpage}
\usepackage{color}
\usepackage[colorlinks,backref]{hyperref}
\usepackage{a4wide}
\usepackage{comment}
\allowdisplaybreaks
\usepackage[margin=1in]{geometry}
\usepackage{amsmath}
\usepackage{amsfonts}
\usepackage{amssymb}
\usepackage{amsthm,mathtools,mathabx,accents,}
\usepackage{newlfont}
\usepackage{graphicx}
\usepackage{xcolor,enumerate,mathrsfs, scalerel}

\newtheorem{thm}{Theorem}[section]
\newtheorem{cor}[thm]{Corollary}
\newtheorem{lem}[thm]{Lemma}
\newtheorem{prop}[thm]{Proposition}
\newtheorem{proposition}[thm]{Proposition}

 \newtheorem*{maintheorem*}{Main Theorem}

\theoremstyle{definition}
\newtheorem{defn}[thm]{Definition}
\newtheorem{rem}[thm]{Remark}
\numberwithin{equation}{section}

\newcommand{\norm}[1]{\Vert#1\Vert}

\newcommand{\na}{\nabla}
\newcommand{\pa}{\partial}
\newcommand{\les}{\lesssim}
\newcommand{\lec}{\lesssim}
\newcommand{\td}{\tilde}

\renewcommand{\div}{\operatorname{div}}
\newcommand{\Tau}{\mathrm{T}}

\newcommand*{\ad}{\ensuremath{\mathrm{ad\,}}}
\newcommand{\abs}[1]{\left|#1\right|}

\newcommand\al{\alpha}
\newcommand\be{\beta}
\newcommand\de{\delta}
\newcommand\De{\Delta}
\newcommand{\ga}{\gamma}
\newcommand\Ga{\Gamma}

\newcommand\e {\varepsilon}

\newcommand{\la}{\lambda}
\newcommand{\La}{\Lambda}
\newcommand\ph{\varphi}
\newcommand{\ta}{\tau}
\renewcommand{\th}{\theta}

\newcommand{\Grad}{\nabla}

\newcommand{\T}{\mathbb{T}}
\newcommand{\R}{\mathbb{R}}
\newcommand{\Z}{\mathbb{Z}}
\newcommand{\N}{\mathbb{N}}

\newcommand{\cH}{\mathcal{H}}

\newcommand{\cR}{\mathcal{R}}
\newcommand{\cRl}{\mathcal{R}^{\mathrm{loc}}}

\newcommand{\crF}{\mathscr{F}}
\newcommand{\vr}{\varrho}

\newcommand{\supp}{\operatorname{supp}}
\newcommand{\tr} {\mathop{\mathrm{tr}}}
\newcommand{\I}{\mathrm{Id}}
\newcommand{\tri}{\triangle}

\newcommand{\idv}[1]{\mathcal{R}\left( #1\right)}
\newcommand{\idvl}[1]{\mathcal{R}^{\mathrm{loc}}\left( #1\right)}

\newcommand{\as}[1]{\accentset{\Diamond}#1}
\renewcommand{\dot}[1]{\accentset{\circ}#1}

\DeclarePairedDelimiter{\ceil}{\lceil}{\rceil}

\usepackage[usestackEOL]{stackengine}
\def\dint{\,\ThisStyle{\ensurestackMath{%
  \stackinset{c}{.2\LMpt}{c}{.5\LMpt}{\SavedStyle-}{\SavedStyle\phantom{\int}}}%
  \setbox0=\hbox{$\SavedStyle\int\,$}\kern-\wd0}\int}

\renewcommand{\dot}[1]{\accentset{\circ}#1}

\newcommand{\MM}[1]{\ensuremath{\mathcal{M}}\left(#1\right)}
\newcommand{\const}{\mathcal{C}}
\newcommand{\dpot}{{\mathsf d}}
\newcommand{\Ndec}{{\mathsf{N}_{ dec}}}
\newcommand{\divH}{\mathcal{H}}
\newcommand{\divR}{\mathcal{R}^*}
\newcommand{\RR}{\overline R}

\begin{document}

\title{Non-uniqueness of H\"{o}lder continuous solutions for Inhomogeneous Incompressible Euler flows}

\author{Vikram Giri}
\author{Ujjwal Koley}
\date{\today}
\address{\parbox{\linewidth}{
Vikram Giri \\
Department of Mathematics, ETH Z\"urich\\
R\"amistrasse 101, Z\"urich 8092\\
E-mail address: vikramaditya.giri@math.ethz.ch
\\ \\
Ujjwal Koley \\
Centre for Applicable Mathematics, Tata Institute of Fundamental Research\\
P.O. Box 6503, GKVK Post Office, Bangalore 560065, India\\
E-mail address: ujjwal@math.tifrbng.res.in
}
} 

\begin{abstract} 
We consider the inhomogeneous (or density dependent) incompressible Euler equations in a three-dimensional periodic domain. We construct density $\vr$ and velocity $u$ such that, for any $\al<1/7$, both of them are $\al $-H\"older continuous and $(\vr, u)$ is a weak solution to the underlying equations. The proof is based on typical convex integration techniques using Mikado flows as building blocks. As a main novelty with respect to the related literature, our result produces a H\"older continuous density.
\end{abstract}

\maketitle

\tableofcontents

\section{Introduction} 
We are interested in the following system of equations for the inhomogeneous (or density dependent) incompressible Euler equations on the spatially periodic domain $[0,T] \times \T^3$,
\begin{equation}\label{eqn.E}
\begin{cases}
\partial_t \vr + \na \cdot (\vr u) = 0,\\
\partial_t (\vr u) +  \na \cdot (\vr u \otimes u)+  \Grad p  =  0,\\
\na \cdot  u  =0.
\end{cases}
\end{equation}
where $\T^3=[-\pi,\pi]^3$ and $0<T<\infty$. Here  $\vr:[0,T] \times \T^3\to \R^+$ denotes the density, $u:[0,T] \times \T^3\to \R^3$ represents the velocity of the fluid and $p:[0,T] \times \T^3\to \R$ denotes the pressure. This kind of system typically describes the motion of fluid that is obtained by mixing two miscible fluids (or pollutants) that are incompressible and have different densities. In fact, density fluctuations are widely present in turbulent flows. They may arise from non-uniform species concentrations, temperature variation or pressure.
In this article, we are interested in studying weak (distributional) solutions $(\vr, u)$ to \eqref{eqn.E} which are H\"older continuous in space, for instance
$$
|u(x,t) - u(y,t)| \le C |x-y|^{\gamma}, \qquad \forall x, y \in \T^3, \quad \forall t \in [0,T],
$$
where $C$ is a constant, independent of time $t$, and $\gamma \in (0,1)$ is the H\"older exponent.

Note that when $\vr$ is constant, the above system \eqref{eqn.E} is called the homogeneous incompressible Euler equations, often simply called
the incompressible Euler equations. Indeed, by setting $\vr \equiv 1$, the equations \eqref{eqn.E} can be formulated as
\begin{equation}\label{eqn.E1}
\begin{cases}
\pa_t u + \na \cdot (u\otimes u) + \na p = 0\\
\na \cdot u =0.
\end{cases}
\end{equation}
Many authors (\cite{DLSz2013,DlSzJEMS,Is2013, Bu2014,BuDLeSz2013,Bu2015,BuDLeIsSz2015,BuDLeSz2016,DaSz2016,IsOh2016,Is2016,BDLSV2020}) have made significant contributions to the research on the non-uniqueness of solutions to incompressible Euler equations.

A {\it dissipative weak solution} to \eqref{eqn.E1} is a distributional solution which belongs $L^3([0,T]\times \T^3)$ and satisfies additionally the {\it local energy inequality} 
\begin{equation}\label{eqn.EI}
\pa_t\left( \frac {|u|^2}2\right) + \na \cdot \left(\left( \frac {|u|^2}2 + p\right) v \right) \le 0
\end{equation}
in the sense of distributions. In the case of dissipative weak solutions, the pioneering work of Isett~\cite{Is2016} developed a Nash iteration in this context and produced the first examples of non-uniqueness amongst H\"older continuous solutions. In this paper, we are heavily indebted to the results and methods by De Lellis and Kwon \cite{Kwon1}, where they were able to produce infinitely many global-in-time disspative weak solutions $u$ to the incompressible Euler equations \eqref{eqn.E1} in the H\"older space $C^\gamma([0, T]\times \T^3)$, for any $0\le \gamma< \frac 17$. The methods in~\cite{Is2016, Kwon1} even produced H\"older continuous weak solutions to incompressible Euler equations \eqref{eqn.E1} which satisfy the local energy equality (i.e., \eqref{eqn.EI} with $\le$ is replaced by $=$). These techniques and results were extended by the first author and Kwon in \cite{Kwon2} to the compressible Euler equations: they have provided infinite many weak solutions, without the presence of discontinuities, to compressible Euler equations which satisfy the local entropy equation. 

In this paper, we would mainly focus on the density dependent incompressible Euler system \eqref{eqn.E}. Mathematically, having a variable density changes the regularity structure substantially, and poses additional challenges compared to the much more extensively studied homogeneous flows. We prove nonuniqueness in this context by noticing the formal similarity of the inhomogeneous eqns.~\eqref{eqn.E} with the system~\eqref{eqn.E1}+\eqref{eqn.EI}. Indeed the first equation of \eqref{eqn.E} has a quadratic non-linearity $\varrho u$ in the divergence which is analogous to the quadratic non-linearity $u\otimes u$ in the divergence in \eqref{eqn.E1}. Also, the cubic non-linearity $\varrho u\otimes u$ in the divergence in \eqref{eqn.E} is analogous to the cubic term in \eqref{eqn.E1}. Noting these formal similarities allows us to follow the work~\cite{Kwon1} and develop an analogous Nash iteration for the inhomogenous incompressible Euler equations.

We briefly summarize some existing results in the literature concerning inhomogeneous incompressible Euler flows. For local well-posedness, we refer to the works of Danchin, Fanelli, and Marsden in \cite{05,06,22}. For global-in-time existence results, we refer to the work of Gebhard et al. in \cite{gebhard}. On a different note, there are a few results concerning energy conservation for weak solutions to \eqref{eqn.E}. We highlight the work by Chen and Yu \cite{chen1}, where, employing well-known commutator estimates, the authors established energy conservation for the density-dependent Euler equations. Specifically, they identified two distinct sets of sufficient regularity conditions for energy conservation: one involving integrability of the spatial gradient of the density, and the other requiring additional time Besov regularity for the velocity field, allowing for lower regularity in the density profile.
We also mention a related result by Fei \cite{Fei}, who proved the existence of weak solutions to the three-dimensional inhomogeneous incompressible Navier–Stokes equations that satisfy a corresponding local energy inequality.

Here, we investigate the issue of non-uniqueness for Hölder continuous weak solutions to \eqref{eqn.E}. Specifically, we establish the following main theorem, which demonstrates the existence of non-unique weak solutions to \eqref{eqn.E}. Compared to the homogeneous case, the literature on the inhomogeneous incompressible Euler equations is relatively sparse. To the best of our knowledge, this is the first result concerning global-in-time Hölder continuous weak solutions to \eqref{eqn.E} (for global-in-time $L^{\infty}$ solutions, consult \cite{gebhard}). We emphasize that our analysis applies to the genuinely inhomogeneous case, where the density is non-constant. When 
$\rho \equiv 1$, the system reduces to the classical incompressible Euler equations, for which more extensive and stronger results are known.

\begin{thm}\label{thm}
For any $0\leq \al < \frac17$ there are infinitely many global-in-time weak solutions $(\vr, u)$ to the inhomogeneous incompressible Euler equation \eqref{eqn.E} in $C^\al([0, T]\times \T^3) \times C^\al([0, T]\times \T^3)$ with $\vr > 0$ and not identically constant. Moreover, these solutions share the same initial data.
\end{thm}

Broadly speaking, the proof of the above theorem relies on the convex integration scheme originally developed by De Lellis and Székelyhidi for the incompressible Euler equations (see \cite{DLSz2013, DlSzJEMS}). More specifically, we employ a variant of this scheme as proposed by Isett \cite{Is2022} for the incompressible Euler equations and further extended by the first author and Kwon \cite{Kwon2} to the compressible setting. However, our approach differs significantly from that of \cite{Kwon2}, where the density is prescribed in advance and remains fixed throughout the iteration. In contrast, our construction involves a genuinely coupled convex integration scheme, in which both the density and the velocity evolve at each stage of the iteration.

To implement this, we utilize two distinct sets of Mikado building blocks \cite{DaSz2016} - one tailored for the density and the other for the velocity. The first step is to ensure that three distinct types of Reynolds errors are controlled by imposing appropriate moment conditions on the Mikado flows. The principal difficulty in our analysis arises from a wide array of error terms generated by the interactions between the evolving density and velocity profiles (see Section~\ref{errors} for details). These interaction-driven errors are absent in previous works such as \cite{Kwon2}, where the density is fixed by construction.

Among the various challenges, we highlight the treatment of the so-called transport current error term (cf. Subsection~\ref{transport error}). A natural approach would be to apply ``inverse divergence'' (Bogovskii) operator twice; however, it is well known that the Bogovskii operator does not commute with the material derivative. This introduces intricate commutator estimates that complicate the analysis significantly. To address this issue effectively, we adopt a localized inverse divergence operator introduced in recent works by the first author, Kwon, and Novack \cite{Giri02, Giri03} (see also Buckmaster et al.\ \cite{Giri01}). The key advantage of this operator lies in its explicit structure, which facilitates direct computation of material derivatives involving inverse divergence terms. However, its application requires inductive control over higher-order (material) derivatives of the velocity field throughout the iteration process.

The rest of the paper is organized as follows: In Section~\ref{pre}, we present the mathematical framework  employed throughout the paper. Moreover, we have introduced the inductive scheme for constructing inhomogeneous Euler-Reynolds flows and two main propositions which play pivotal role in the proof of the Theorem~\ref{thm}. The proof of Theorem~\ref{thm}, assuming Proposition~\ref{ind.hyp} and Proposition~\ref{p:ind_technical}, is given in Section~\ref{proof}. The construction of the density and velocity perturbations, estimates on the perturbations, the mollification process are described in Section~\ref{sec:correction.const}, while the definition of three different types of new errors are given in Section~\ref{errors}. The corresponding estimates of the errors are carried out over Section~\ref{est}. The proofs of two key inductive propositions are given Section~\ref{induc}.  Finally, four appendices collect several analytical tools and auxiliary results used throughout the paper.

\section{Preliminaries and mathematical framework}
\label{pre}
Here we briefly recall some relevant mathematical tools which are used in the subsequent analysis. To begin, we use the letter $C$ to denote various generic constants that may change from line to line along the proofs. Explicit tracking of the constants is possible but it is highly cumbersome and avoided for the sake of readability. However, if situation demands, we will explicitly mention the dependence of $C$ on the various parameters involved in our arguments. Moreover, throughout the manuscript, given any two parameter ($q$) dependent quantities $A_q$ and $B_q$ we use the following convention: $A\lec B $ implies that $A\leq CB$ for some constant $C$ which is independent of the parameter $q$. Furthermore, we shall use the following convention:
\begin{itemize}
	\item For any given function $g$ on $[0, T]\times \T^3$, we will denote by $\supp_t (g)$ its temporal support, namely
	\[
	\supp_t (g) := \{t: \exists \ x \;\mbox{with}\; g (t,x) \neq 0\}\, .
	\] 
    Also, for a set of such functions $\{g_i\}$, we denote by $\supp_t(\{g_i\})$ their joint temporal support, namely
    \[ \supp_t(\{g_i\}) := \bigcup_i \supp_t g_i\,. \]
	\item Given any interval $\cal{I} = [c,d]$, we will denote by $|\cal{I}|$ its length $(d-c)$ and $\cal{I} + \zeta$ will denote the concentric enlarged open interval $(c-\zeta, d+\zeta)$. 
\end{itemize}

We set $\N_0 = \N \cup \{0\}$. For $\alpha \in (0,1)$, $N \in \N_0$, we denote supremum norm, and H\"{o}lder semi-norm and norm by
\begin{align*}
& \|f \|_{C_t^0 C_x^0}:= \sup_{t,x} |f(t,x)|, \qquad [f]_{C_t^0 C_x^N} := \max_{|\gamma|=N} \| D^{\gamma} f\|_{C_t^0 C_x^0} \\
& [f]_{C_t^0 C_x^{N+\alpha}}:= \max_{|\gamma|=N}  \sup_{x\neq y} \frac{|D^{\gamma} f(t,x) - D^{\gamma} f(t,y)|}{|x-y|^{\alpha}} \\
& \| f \|_{C_t^0 C_x^N} :=  \sum_{j=0}^{N} [f]_{C_t^0 C_x^j}, \qquad 
\|f\|_{C_t^0 C_x^{N+\alpha}}:= \| f \|_{C_t^0 C_x^N} + [f]_{C_t^0 C_x^{N+\alpha}}.
\end{align*}

We now introduce some notations in Fourier analysis described in \cite{Gra}.
For a function $h$ in the Schwartz space $\cal{S}(\R^3)$, we can define the Fourier transform of $h$ and its inverse on $\R^3$ as 
\[
\widehat{h} (\xi) = \frac 1{(2\pi)^3}\int_{\R^3} h(x) e^{-ix \cdot \xi} dx,\quad
\widecheck{h} (x) = \int_{\R^3} h(\xi) e^{ix \cdot \xi} d\xi. 
\]   
Moreover, we can extend the Fourier transform to more general linear functionals in $\mathcal{S}' (\R^3)$ by duality. Since all the relevant functions, vectors and tensors in this paper are defined on $J\times \mathbb T^3$ for some time domain $J\subset \mathbb R$ which we also regard with a slight abuse of notation as spatially periodic functions, we always consider their Fourier transform as time-dependent elements of $\mathcal{S}' (\R^3)$. Following standard convention for Littlewood-Paley operators, we multiply the Fourier transform of a distribution $h$ by a smooth cut-off function and then apply inverse Fourier transform on the product to obtain a smooth function.
To that end, we let $m(\xi)$ be a radial smooth function such that $ \text{supp}\, m (\xi) \subset B(0,2)$ and $m\equiv 1$ on $\overline{B(0,1)}$. Then for any number $i\in \Z$ and distribution $h$ in $\R^3$, we set
\begin{align*}
\widehat{P_{\leq 2^i} h}(\xi) &:=m\left(\frac{\xi}{2^i}\right) \hat{h}(\xi),\quad
\widehat{P_{> 2^i} h}(\xi) := \left(1-m\left(\frac{\xi}{2^i}\right)\right) \hat{h}(\xi), 
\end{align*}
and for $i\in \Z$
\begin{align*}
\widehat{P_{2^i}h}(\xi) :=\left( m\left(\frac{\xi}{2^i}\right)
-m\left(\frac{\xi}{2^{i-1}}\right)\right) \hat{h}(\xi).
\end{align*}
In the subsequent analysis, for any given number $K$, we donote by $P_{\leq K} h= P_{\leq 2^I} h$ where $I= \lfloor  \log_2 K \rfloor$ is the largest integer satisfying $2^{I} \le K$. Similarly, we define $P_{> K} h= h - P_{\leq K} h$. For any spatially periodic function $\phi$ on $J\times \mathbb T^3$,  $P_{>\ell^{-1}} \phi = P_{>2^J} \phi$ can be written as the space convolution of $\phi$ with the kernel $\widecheck{m}_{\ell}(x) =2^{3J} \widecheck{m} (2^J x)$, which belongs to $\mathcal{S} (\R^3)$. An immediate application of above formula is that $\norm{|y|^k\widecheck{m} _\ell}_{L^1(\R^3)}\lec \ell^k$, for any $k\geq 0$. We also recall the celebrated Bernstein's inequality. Indeed, we have $\norm{\phi-P_{\le \ell^{-1}}\phi}_0 = \norm{P_{> \ell^{-1}}\phi}_0 \lec \ell^{i}\norm{\na^i \phi}_0$ for any $\phi\in C^i(\T^3)$. 

Finally we mention, as alluded to before, we need to use two different types of inverse divergence operators to deal with the ``transport current error'' term in Subsection~\ref{transport error}. Indeed, throughout the paper, we reserve the notation $\mathcal{R}$ to denote the (non-local) inverse divergence operator given by Definition~\ref{idv.defn}, and we use the notation $\mathcal{R}^{\mathrm{loc}}$ to denote the (local) inverse divergence operator given by the Proposition~\ref{prop:intermittent:inverse:div}.

\subsection{Iteration Scheme}

Note that, our main aim is to construct global inhomogeneous Euler flows as a limit of sequences of inhomogeneous Euler-Reynolds flows. In what follows, we begin with the definition of inhomogeneous Euler-Reynolds flows.

\begin{defn}[Inhomogeneous Euler-Reynolds flows] A tuple of smooth tensors $(\vr, u,p,R,\Phi, S)$, with $\vr \ge \varepsilon_0$ for some positive constant $\varepsilon_0$, is called an {\it inhomogeneous Euler-Reynolds flow} if the tuple solves the following inhomogeneous Euler-Reynolds system:
\begin{equation}\label{app.eq}
\begin{cases}
\partial_t {\vr} + \na \cdot ({\vr} {u})  = - \na \cdot R,\\ 
\partial_t ({\vr} {u} ) +  \na \cdot ({\vr} u \otimes u)+  \Grad p  =  -(\pa_t+ (u\cdot\na)) R - \na \cdot (R \otimes {u}) + \na \cdot \Phi + \na \cdot (\vr S),\\
\na \cdot \,{u}  =0.
\end{cases}
\end{equation}
\end{defn}

As usual in such convex integration schemes, our aim is to build a corrected inhomogeneous Euler-Reynolds flow at the level $q+1$, i.e., $(\vr_{q+1}, u_{q+1},p_{q+1},R_{q+1},\Phi_{q+1}, S_{q+1})$, assuming that we have already contructed inhomogeneous Euler-Reynolds flow
at the level $q$, i.e., $(\vr_{q}, u_{q},p_{q},R_{q},\Phi_{q}, S_{q})$. Moreover, we also want the new error $(R_{q+1},\Phi_{q+1},S_{q+1})$ at the level $q+1$ to be substantially smaller than the error $(R_{q},\Phi_{q},S_{q})$ at the level $q$. Indeed, this is typically achieved by adding a suitable correction $(\theta_{q+1}, w_{q+1},q_{q+1})$ to the density, velocity and pressure $(\vr_q, u_q, p_q)$, namely defining $\vr_{q+1} = \vr_q+\theta_{q+1}$, $u_{q+1} = u_q + w_{q+1}$, and $p_{q+1} = p_q +q_{q+1}$ (for a precise statement, consult Proposition \ref{ind.hyp}). In this way we obtain a sequence of solutions to \eqref{app.eq} which we will show converges to a solution of the inhomogeneous incompressible Euler system \eqref{eqn.E}.

The relaxation of \eqref{eqn.E} in \eqref{app.eq} is inspired from the analogous relaxation for the incompressible Euler equations with the local energy inquality as pioneered by Isett in~\cite{Is2022}. Indeed in both our and Isett's relaxations, the various new errors that appear can be written as commutators of mollifications at certain length scales of a weak solution. A difference between our relaxation and Isett's relaxation is the presence of two error terms, $R_q$ and $S_q$, which will be cancelled by quadratic non-linear interactions. Roughly speaking, $R_q$ will be cancelled by the high-high to low interactions of $\theta_{q+1}$ with $w_{q+1}$ while $S_q$ will be cancelled with the high-high to low self interactions of $w_{q+1}$.

Next, we introduce some parameters which will help us to measure the size of various approximate solutions. Let $q\in \N_0$, then we introduce the frequency $\la_q$ and the amplitude $\de_q$ of our approximate solutions as
\begin{align*}
\la_q = \ceil{\la_0 ^{(b^q)}}, \quad \de_q = \la_q^{-2\beta}\, ,
\end{align*}
where $\ceil{a}$ denotes the smallest integer $n \ge a$, $\beta$ is a positive parameter smaller than $1$, $\lambda_0$ is a very large parameter, and $b$ is  typically chosen close to $1$. Due to these choices, we can make sure that $\delta_q^{\frac{1}{2}} \lambda_q$ is a monotone increasing sequence.
 
We will assume several inductive estimates on the tuple $(\vr_q, u_q, p_q, R_q, \Phi_q, S_q)$ satisfying \eqref{app.eq}. First note that, due to truncation and smoothing in space-time variables, the domains of definition of approximates solutions going to change at each step and it is chosen at step $q$ as $[-\tau_{q-1}, T+\tau_{q-1}]\times \mathbb T^3$, where $\tau_{-1} = \infty$ and for $q\geq 0$ the parameter $\tau_q$ is defined by 
\[
\tau_q = \left(40\pi {M}\eta^{-1} \la_q^\frac12\la_{q+1}^\frac12 \de_q^\frac14\de_{q+1}^\frac14 \right)^{-1}
\] 
for some geometric constants $\eta$ (given by Proposition~\ref{p:supports}) and $M$ (given by Proposition \ref{ind.hyp}). It is clear from the definition that $\tau_q$ is decreasing in $q$. 
We are now ready to give details of inductive estimates:
\begin{align}
&\norm{\vr_q}_0 \leq 5-\de_q^\frac12, \qquad \qquad \quad \norm{u_q}_0 \leq 5-\de_q^\frac12, \label{est.vp22} \\ & 
\norm{\vr_q}_N \leq M\la_q^N \de_q^\frac 12, \qquad \qquad \norm{u_q}_N \leq M\la_q^N \de_q^\frac 12, \qquad N=1,2. \label{est.vp} 
\end{align}
Moreover,  for $ N = 0,1,2,$
\begin{align}
& \norm{D_{t} \vr_q}_{N-1} \leq M \tau_q^{-1} \la_q^{N-1} \de^{1/2}_q, \quad \norm{D_{t} u_q}_{N-1} \le M \tau_q^{-1} \la_q^{N-1} \de^{1/2}_q,  \label{e:ad-pressure} \\
&\norm{R_q}_N \leq \la_q^{N-2\ga}\de_{q+1}, \qquad \qquad \,\,\,\,\,\,\norm{D_{t} R_q}_{N-1} \leq \la_q^{N-2\ga} \de_q^\frac 12 \de_{q+1}, \label{est.R}\\
&\norm{\Phi_q}_N \leq \la_q^{N-4\ga} \de_{q+1}^\frac 32, \qquad \qquad \,\,\,\,\,\,
\norm{D_{t} \Phi_q}_{N-1} \leq \la_q^{N-4\ga}\de_q^\frac12 \de_{q+1}^\frac 32,  \label{est.ph} \\
&\norm{S_q}_N \leq \la_q^{N-2\ga}\de_{q+1}, \qquad \qquad  \,\,\,\,\,\, \, \norm{D_{t}  S_q}_{N-1} \leq \la_q^{N-2\ga} \de_q^\frac 12 \de_{q+1}. \label{est.S}
\end{align}
Furthermore, for all $0 \le s \le M_*$ and $N \le N_* $, with $N \neq 0$
\begin{align}
\norm{D^s_{t} \varrho_q}_{N} \le M \la_q^N \de^{\frac 12}_q \tau_q^{-s}, \qquad
\norm{D^s_{t} u_q}_{N} \le M \la_q^N \de^{\frac 12}_q \tau_q^{-s}, \qquad 
\| \partial_t^s u_{q}\|_N \le M \la^N_q \de_q^{1/2} \tau_q^{-s}, \label{est.new}
\end{align}
where the material derivative $D^s_{t} = (\pa_t + u_q\cdot \na)^s $, the constants $M_*$ and $N_*$ depend on $b$ and will be chosen later (see Remark~\ref{rmk01}), and the numerical constant $\ga = (b-1)^2$. Given above inductive assumptions, we are now ready to state the main inductive proposition.

\begin{rem}
Note that in the above inductive estimates, we have used the following convention: For any function $G_q$ defined on $[-\tau_{q-1}, T+\tau_{q-1}]\times \mathbb T^3$, $\norm{G_q}_N$ is the $C^0_t C^N_x$ norm of $G_q$ on its domain of definition, namely 
\[
\norm{G_q}_N := \norm{G_q}_{C^0([0, T]+ \tau_{q-1};C^N(\T^3))}\, .
\]
Moreover, for $N=0$ the advective derivative estimate $\norm{D_{t}  G}_{N-1} $ is an empty statement - we are not claiming any negative Sobolev estimate.
\end{rem}

\begin{prop}[Inductive proposition]\label{ind.hyp} 
For any $\beta \in (0, \frac{1}{7})$, there exists constants $M{>1}$ and functions $\bar{b} (\beta) >1$ and $\Lambda_0 (\beta, b,{M}) >0$ such that the following property holds. For any $b \in (1, \bar b (\beta))$ and $\la_0\geq \La_0 (\beta , b,{M})$, assume that $(\vr_q, u_q, p_q, R_q, \Phi_q, S_q)$ is a inhomogeneous Euler-Reynolds flow defined on the time interval $[0,T]+ \tau_{q-1}$ satisfying \eqref{est.vp22}-\eqref{est.new}. Then, there exists a corrected inhomogeneous Euler-Reynolds flow $( \vr_{q+1}, {u}_{q+1},  p_{q+1},  R_{q+1}, \Phi_{q+1}, S_{q+1})$ which is defined on the time interval $[0,T]+\tau_q$ satisfies \eqref{est.vp22}-\eqref{est.new} for $q+1$ and additionally
\begin{align}\label{cauchy_01}
\inf_{\T^d}{(\vr_{q+1}-\vr_q)} \ge - \frac M2 \de_{q+1}^\frac 12, 
\end{align}
\begin{align}\label{cauchy}
\norm{\vr_{q+1}-\vr_q}_0 + \frac 1{\la_{q+1}}\norm{\vr_{q+1}-\vr_q}_1 + \norm{u_{q+1}-u_q}_0 + \frac 1{\la_{q+1}}\norm{u_{q+1}-u_q}_1 \leq  M \de_{q+1}^\frac 12,
\end{align}
\begin{align}\label{cauchy_p}
p_{q+1} = p_q - \de_{q+1} \varrho_q.
\end{align}
\end{prop}

We also need a technical refinement of the above proposition following~\cite{Kwon1} to prove the Theorem \ref{thm}. In an effort to improve the readability of the paper, we state the following proposition separately.

\begin{prop}[Bifurcating inductive proposition]\label{p:ind_technical}
	Let the geometric constant $M {>1}$, the functions  $\bar b$, $\Lambda_0$, the parameters $\beta$, $b$, $\lambda_0$ and the tuple $(\vr_q, u_q,p_q,R_q,\Phi_q,S_q)$ be as in the statement of Proposition \ref{ind.hyp}. For any time interval $\cal{I}\subset (0, T)$ with $|\cal{I}|\geq 3\tau_q$
	we can produce a first tuple $(\vr_{q+1}, {u}_{q+1},  p_{q+1},  R_{q+1}, \Phi_{q+1}, S_{q+1})$ and a second one $(\td \vr_{q+1}, \td u_{q+1},  \td p_{q+1}, \td R_{q+1}, \td \Phi_{q+1}, \td S_{q+1})$ which share the same initial data, satisfy the same conclusions of Proposition \ref{ind.hyp} and additionally
	\begin{align}\label{e:distance_and_support}
	\norm{u_{q+1}-\td u_{q+1}}_{C^0([0, T];L^2(\T^3)) }\geq \de_{q+1}^\frac 12, \quad
	\supp_t(u_{q+1}-\td u_{q+1}) \subset \cal{I}. 
	\end{align}
	Furthermore, if we are given two tuples $(\vr_q, u_q,p_q,R_q,\Phi_q,S_q)$ and $(\td \vr_q, \td u_q, \td p_q, \td R_q, \td \Phi_q, \td S_q)$ satisfying \eqref{est.vp}-\eqref{est.new} and
	\[
	\supp_t(\{\vr_q - \td \vr_q, u_q-\td u_q, p_q-\td p_q, R_q-\td R_q, \Phi_q-\td \Phi_q, S_q-\td S_q\})\subset \cal{J} 
	\]
	for some interval $\cal{J}\subset (0, T)$, we can exhibit corrected counterparts $(\vr_{q+1}, {u}_{q+1},  p_{q+1},  R_{q+1}, \Phi_{q+1}, S_{q+1})$ and $(\td \vr_{q+1}, \td u_{q+1},  \td p_{q+1}, \td R_{q+1}, \td \Phi_{q+1}, \td S_{q+1})$ again satisfying the same conclusions of Proposition \ref{ind.hyp} together with the following control on the support of their difference:
	\begin{equation}\label{e:second_support}
	\supp_t(\{\vr_{q+1} - \td \vr_{q+1}, u_{q+1}-\td u_{q+1}, p_{q+1}-\td p_{q+1}, R_{q+1}-\td R_{q+1}, \Phi_{q+1}-\td \Phi_{q+1}, S_{q+1}-\td S_{q+1}\})\subset \cal{J} + (\la_q\de_q^\frac12)^{-1}.
	\end{equation}
\end{prop}


\section{Proof of Theorem \ref{thm}} 
\label{proof}
Here we shall prove the result stated in Theorem~\ref{thm}, assuming Proposition~\ref{ind.hyp}, and Proposition~\ref{p:ind_technical}. In what follows, let us first fix $T\ge 20$, $\alpha< \frac 17$, and $\beta\in (\al, \frac  17)$. Then choose $b$ and $\la_0$ in the range suggested in Proposition \ref{ind.hyp}. We first need to choose the starting tuple $(\vr_0, u_0,p_0,R_0,\Phi_0,S_0)$, so that they satisfy the estimates given by \eqref{app.eq}-\eqref{est.S}. For this purpose, we fix an integer parameter $\bar\lambda>0$ (which will be chosen appropriately later) and define 
\begin{align*}
\vr_0 &= \chi_0(t) + ( 1 + \frac 14 \cos(\bar\lambda x_3))(1- \chi_0(t)), 
\qquad u_0 = 0, \qquad p_0 = \Phi_0 =0,\\
R_0 & =  \partial_t \chi_0(t)  (4 \bar\lambda)^{-1}\sin(\bar\lambda x_3)e_3, \qquad 
\div (\vr_0 S_0) = -\partial_t R_0,
\end{align*}
where we choose $\chi_0(t):= (1-2 \delta_0^\frac{1}{2} + e(t))^{\frac{1}{2}}$, and $e(t) := -\de_1(1-\exp(-\de_0^\frac12 t))$. 
Note that, by assuming $\lambda_0$ large enough, $0 \leq \delta_0 - {\de_1} \leq 2\delta_0^\frac{1}{2} - {e(t)} \leq 3\delta_0^\frac{1}{2} \leq \frac{1}{2}$ 
and thus $\vr_0$ and $R_0$ are well defined and smooth, since
\begin{equation}
1 - 2\delta_0^{\frac{1}{2}} + e(t) \geq \frac{1}{2}\, .
\end{equation}
Next, notice that it is easy to check that the tuple satisfies \eqref{app.eq}, and thanks to the choice of $p_0, u_0$, and $\Phi_0$, the estimates \eqref{est.ph}, and the estimates related to $p_0$ and $u_0$ in \eqref{est.vp22}, \eqref{est.vp},\eqref{e:ad-pressure}, and \eqref{est.new} are trivially satisfied. Therefore our aim is to choose $\bar\lambda$ so that all other estimates in  \eqref{est.vp22}-\eqref{est.S} hold. 

We first make use of the fact that $e(t)\leq 0$ to estimate
$$
\|\vr_0\|_0 \leq |\chi_0| + \frac 54 |1- \chi_0| \le (1-2\delta_0^\frac{1}{2})^{1/2} + \frac 52 \le 5 - \delta_0^{1/2}, \qquad \|\vr_0\|_N \leq 2 \bar\lambda^N,
$$
Therefore above estimate and the choice of $u_0$ guarantees that estimates \eqref{est.vp22} and \eqref{est.vp} is satisfied as soon as
\begin{equation}\label{e:barlambda_1}
\bar\lambda \leq  \lambda_0\delta_0^{\frac{1}{2}}\, .
\end{equation}
As for \eqref{e:ad-pressure}, \eqref{est.R}, and \eqref{est.S} we estimate
\begin{align*}
\|D_t \vr_0\|_{N-1} & \leq C \bar \lambda^{N-1} \|e'\|_0 \\
\|R_0\|_N  &\leq C \bar \lambda^{N-1} \|e'\|_0 , \qquad 
\|D_t R_0\|_N =\|\partial_t R_0\|_N \leq C \bar \lambda^{N-1} (\|e''\|_0 + \|e'\|_0^2), \\
\|\vr_0 S_0\|_N  = \| \cal{R} \partial_t R_0 \|_N &\leq  \|\partial_t R_0\|_N, \qquad 
\|D_t (\vr_0 S_0)\|_N =\|\partial_{tt} R_0\|_N \leq C \bar \lambda^{N-1} (\|e'''\|_0 + \|e' e''\|_0 + \|e'\|_0^3),
\end{align*}
where $C$ is a geometric constant.
Notice that thanks to \eqref{e:barlambda_1}, we already have $\bar \lambda \leq \lambda_0$; by the choice of $\vr$, we have $\inf{\varrho_0} \ge 3/4$; and moreover $\|e'\|_0 \leq {\de_0^\frac12}\delta_1$, $\|e''\|_0 \leq {\de_0}\delta_1$, $\|e'''\|_0 \leq {\de_0}^{3/2}\delta_1$, by the specific choice of the function $e(t)$ (and hence
$\|e'\|_0^2 \leq {\de_0} \delta_1^2 \leq {\delta_0\delta_1}$ and so on). Therefore, to verify the estimates \eqref{e:ad-pressure}, \eqref{est.R}, and \eqref{est.S}, it suffices to impose 
\begin{equation}\label{e:barlambda_2}
\bar \lambda\geq C \lambda_0^{2\gamma} {\de_0^\frac12}\, .
\end{equation} 
Finally, we need to check that the requirements \eqref{e:barlambda_1} and \eqref{e:barlambda_2} are mutually compatible for some choice of the parameters satisfying the assumptions of Proposition \ref{ind.hyp}: taking into account that $\bar\lambda$ must be an integer, the requirement amounts to the inequality ${\delta_0^{\frac{1}{2}}( \lambda_0 - C \lambda_0^{2\gamma})} \geq 1$. 
Recall that $2\gamma = 2(b-1)^2$ and 
the restrictions on $b$ given by Proposition \ref{ind.hyp} certainly allows us to choose $b$ so that ${2}(b-1)^2<\frac{1}{2}$. The compatibility is then satisfied if {$\lambda_0 >C \lambda_0^{\frac{1}{2}} +\la_0^{\beta}$.} Since $C$ is a geometric constant, we just need a sufficiently large $\lambda_0$, which is again a restriction compatible with the requirements of Proposition \ref{ind.hyp}.

Since the sequence $\{\vr_q, u_q, p_q\}$ also satisfies \eqref{cauchy}, it is Cauchy in $C^0([0, T];C^{\al}(\T^3))$. Indeed, for any $q<q'$, we have
the following estimates
\begin{align*}
\norm{\vr_{q'} - \vr_q}_{C^0([0, T];C^\al(\T^3))}
&\leq \sum_{l=1}^{q'-q}\norm{\vr_{q+l} - \vr_{q+l-1}}_{C^0([0, T];C^\al(\T^3))}\\
&\lec \sum_{l=1}^{q'-q}
\norm{\vr_{q+l} - \vr_{q+l-1}}_{0}^{1-\al}
\norm{\vr_{q+l} - \vr_{q+l-1}}_{1}^{\al}
\lec \sum_{l=1}^{q'-q} 
\la_{q+l}^{\al} \de_{q+l}^{\frac12} = \sum_{l=1}^{q'-q} 
\la_{q+l}^{\al-\be} \to 0, 
\end{align*}
as $q$ goes to infinity because of $\al-\be<0$. Therefore, we obtain its limit $\vr$ in $ C^0([0, T];C^\al( \T^3))$. An exact similar calculation reveals that $u$ is in $ C^0([0, T];C^\al( \T^3))$. Moreover, an analogous interpolation argument can be used for $R_q$ as well to conclude the convergence of $R_q$ to $0$ in $C^0([0, T]; C^\alpha ( \T^3))$. Also, the convergence of $p_q$ can be obtained from \eqref{cauchy_p} and the bounds on $\varrho_q$. Since $(R_q, \Phi_q, \vr_q S_q)$ converges to $0$ in $C^0([0, T]\times \T^3)$, the limit $(\vr, u,p)$ solves the inhomogeneous incompressible Euler equation \eqref{eqn.E}. Estimating 
\begin{align*}
\|\partial_t \vr_q\|_0 & \leq \|u_q\|_0 \| \nabla \vr_q\|_0 + \| \nabla R_q\|_0 , \\
\|\partial_t (\vr_q u_q)\|_0 & \leq \|u_q\|_0 \|\nabla(\vr_q  u_q)\|_0 + \| \nabla p_q\|_0 + \|D_t R_q\|_0 + \|\div(R_q \otimes u_q)\|_0 \\
&\qquad\qquad  + \|\div (\vr_q S_q)\|_0
 +\|\div \Phi_q\|_0,
\end{align*}
the time regularity of $\vr$ and $\vr u$ can be concluded with an analogous interpolation argument. Hence $\vr$ and $\vr u$ both are in the space $C^\alpha ([0, T]\times \T^3)$. This also implies, as $\vr$ is bounded below by a positive constant, that $u$ belongs to the space $C^\alpha ([0, T]\times \T^3)$. Indeed, to verify this lower bound of the density, note that thanks to \eqref{cauchy_01}, we have
\begin{align*}
\inf_{\T^d} \vr \ge \inf_{\T^d} \vr_0 + \sum_{q=0}^{\infty} \inf_{\T^d} (\vr_{q+1}-\vr_q) 
\ge \frac 34 - \sum_{q=0}^{\infty} \frac M2 \de_{q+1}^{1/2} \ge 1/4,
\end{align*}
where we have chosen $\lambda_0$ large enough so that $\sum_{q=0}^{\infty} \de_{q+1}^{1/2}  \le \frac{1}{4M}$.

Finally, to generate infinitely many solutions, fix $\bar q\in \N$ satisfying $b^{\bar  q}\geq \bar q$. At the $\bar q$th step using Proposition \ref{p:ind_technical} we can produce two distinct tuples, one which we keep denoting as above and the other which we denote by $(\td \vr_q, \td u_q, \td p_q, \td R_q, \td\Phi_q, \td S_q)$ and satisfies \eqref{e:distance_and_support}, namely
\begin{align*}
\norm{\td u_{\bar q} - u_{\bar q}}_{C^0([0, T];L^2(\T^3))} \geq  \de_q^\frac 12, \quad \supp_t (u_q-\td u_q) \subset \cal I\, ,  
\end{align*}
with $\cal I  = (10, 10+3\tau_{\bar q-1})$.
Applying now the Proposition \ref{ind.hyp} iteratively, we can build a new sequence $(\td \vr_q, \td u_q, \td p_q, \td R_q, \td\Phi_q, \td S_q)$ of approximate solutions which satisfy \eqref{est.vp22}-\eqref{cauchy} and \eqref{e:second_support}, inductively. Arguing as above, this second sequence converges to a solution $(\td \vr, \td u, \td p)$ to the Euler equation. Indeed, $\td \vr$ and $ \td u$ are in $ \in C^\beta ([0, T]\times \T^3)$. We remark that for any $q\geq \bar q$,
\[
\supp_t(\vr_{q} -\td \vr_{q}) \subset \cal I + \sum_{q=\bar q}^\infty (\la_{q}\de_{q}^\frac12)^{-1} \subset[9,T], \quad \supp_t(u_{q} -\td u_{q}) \subset \cal I + \sum_{q=\bar q}^\infty (\la_{q}\de_{q}^\frac12)^{-1} \subset[9,T],
\]
(by adjusting $\la_0$ to be even larger than chosen above, if necessary), and hence $\td \vr_q, \td u_q$ share initial data with $\vr_q, u_q$, respectively, for all $q$. In particular, two solutions $\td u_q$ and $u_q$ have the same initial data.  However, the new solution $\td u$ differs from $u$ because
\begin{align*}
\norm{u - \td u}_{C^0([0, T];L^2(\T^3))}
&\geq \norm{u_{\bar q} -\td u_{\bar q}}_{C^0([0,T];L^2(\T^3))}
-\sum_{q=\bar q}^\infty \norm{u_{q+1}-u_q - (\td u_{q+1} - \td u_q)}_{C^0([0,T];L^2(\T^3))} \\
&\geq \norm{u_{\bar q} - \td u_{\bar q}}_{C^0([0,T];L^2(\T^3))}
-(2\pi)^\frac32 \sum_{q=\bar q+1}^\infty (\norm{u_{q+1}-u_q}_0 + \norm{\td u_{q+1}-\td u_q}_0)\\
&\geq   \de_{\bar q}^\frac 12  - 2(2\pi)^\frac32M \sum_{q=\bar q}^\infty\de_{q+1}^\frac12>0.
\end{align*} 
The last inequality follows from adjusting $\la_0$ to a larger one if necessary. By changing the choice of time interval $\cal I$ and the choice of $\bar{q}$, we can easily generate infinitely many solutions.

\section{Construction of density and velocity perturbations}\label{sec:correction.const}

In this section, we will introduce the building blocks used in the subsequent analysis. In fact, the most efficient one so far found in the literature are the so-called Mikado building blocks, first introduced in \cite{DaSz2016}. They are stationary solutions to incompressible Euler equations with zero pressure. 

\begin{defn}[Mikado Density and Field] 
For any given vector $h \in \Z^3$ and $s_h \in \mathbb R^3$, we refer to $U_{h, s}: \T^3 \rightarrow \R^3$ as a Mikado field and $V_{h, s}: \T^3 \rightarrow \R$ as a Mikado density, provided they satisfy
$$
U_{h, s} = \psi_h (x -s_h) h, \quad V_{h, s} = \varphi_h (x -s_h),
$$ 
where $\psi_h, \varphi_h \in C_c^{\infty}(\T^3)$ and $h \cdot \nabla \psi_h =0 = h \cdot \nabla \varphi_h $.
\end{defn}

Here the vector $h$ is called the direction of the Mikado field (they will vary in a finite fixed set of directions $\mathcal{H}$), and $s=s_h$ is called its shift. It is also easy to check that
\begin{align*}
\div U_{h,s} &= h \cdot \nabla \psi_h =0, \quad
\div(U_{h,s} \otimes U_{h,s}) = (U_{h,s} \cdot \nabla) U_{h,s} = h \psi_h (h \cdot \nabla) \psi_h =0.
\end{align*}
One of the most fundamental observation proposed in the pioneering work \cite{DaSz2016} is the following important lemma:
\begin{lem}\label{l:Mikado}
Given any $\lambda \in \N_0$, $h\in \mathcal{H}$, let $s_h \in \mathbb R^3$ and $\chi_h \in \mathbb R$. Assume that the supports of the maps $U_h (\cdot - s_h)$ are pairwise disjoint. Then
\[
\sum_{h\in \mathcal{H}} \chi_h U_h (\lambda (x-s_h))
\]
is a stationary solution of the incompressible Euler equations on $\mathbb T^3$. 
\end{lem}

Once we choose a finite set of fixed Mikado directions, we can construct corresponding Mikado fields and Mikado densities to construct velocity and density perturbations that reduces the errors $R_q, \Phi_q$, and $S_q$ at each stage. To successfully implement this idea, we need to recall couple of well-known geometric lemmas, proved in \cite{Kwon1}. To state them precisely, let us first denote by $\mathbb{S}$ the subset of $\mathbb R^{3\times 3}$ of all symmetric matrices, and set $|K|_\infty := \max_{l,m} |k_{lm}|$, for any $K= (k_{lm})_{l,m=1}^{3}\in \mathbb R^{3\times 3}$.

\begin{lem}[Geometric Lemma I]\label{lem:geo1} 
	Let $\cH= \{h_i\}_{i=1}^6$ be a set of vectors in $\Z^{3}$ such that
	\begin{equation}\label{con.FIR}
	\sum_{i=1}^6 h_i \otimes h_i = C\I, \quad\mbox{and}\quad
	\{h_i\otimes h_i\}_{i=1}^6 \,\,\text{ forms a basis of }\mathbb{S}\, 
	\end{equation}
	for some constant $C>0$. Then, it is possible to find a positive constant $\mathcal{M}_0 = \mathcal{M}_0(\cH)$ such that for any $N\leq \mathcal{M}_0$, we can find smooth functions $\{\Ga_{h_i}\}_{i=1}^6\subset C^\infty(S_{N}; (0,\infty))$, with domain $S_{N}:= \{\I-K: K \text{ is symmetric, } |K|_\infty\leq N\}$,
	satisfying 
	\begin{align*}
	\I -K = \sum_{i=1}^6 \Ga_{h_i}^2(\I - K) (h_i\otimes h_i), 
	\quad\forall (\I-K) \in S_{N}\, .
	\end{align*}
\end{lem}

\begin{lem}[Geometric Lemma II] \label{lem:geo2}
	Suppose that 
	\begin{equation}\label{con.FIph}
	\{h_1, h_2, h_3\}\subset \Z^3\setminus\{0\} \text { is an orthogonal frame and } 
	h_4=-(h_1+ h_2+h_3).
	\end{equation}
	Then, for any $N_0>0$, there are affine functions $\{\Ga_{h_i}\}_{1\leq i \leq 4} \subset C^\infty(\cal{W}_{N_0}; [N_0, \infty))$ with domain $\cal{W}_{N_0} := \{w\in \R^3: |w| \leq N_0\}$ such that
	\[
	w= \sum_{i=1}^4 \Ga_{h_i} (w) h_i,  \quad\forall w \in \cal{W}_{N_0}\, .
	\] 
\end{lem}

Based on these above two lemmas, we can choose 27 pairwise disjoint families $\cH^i$ indexed by $i\in \mathbb{Z}_3^3$ such that each $\cH^i$ consists of further three (disjoint) subfamilies $\cH^{i,R}\cup\cH^{i,\Phi}\cup\cH^{i,S} $ with cardinalities $|\cH^{i,R}|=4$, $|\cH^{i,\Phi}|=6$, and $|\cH^{i,S}|=6$. Moreover, $\cH^{i,R}$ and $\cH^{i,\Phi}, \cH^{i,S}$ satisfy \eqref{con.FIph} and \eqref{con.FIR}, respectively. For example, for $i= (0,0,0)$ we can choose
\begin{align*}
\cH^{i,\Phi} = \{(1,\pm 1, 0), (1, 0, \pm 1), (0, 1, \pm1)\}, \quad
\cH^{i,R} = \{(1,2, 0), (-2, 1, 0), (0, 0, 1), (1,-3,-1)\},
\end{align*}
and similar for $\cH^{i,S}$, and then we can apply $26$ suitable rotations (and rescalings). 
Then the density and velocity perturbations will be chosen for each $h$ in three different ways, depending on whether $h\in \cH^{i,R}$, $h\in \cH^{i,\Phi}$ or $h\in \cH^{i,S}$. By making use of the shorthand notation
$\langle w \rangle = \dint_{\T^3} w (x)\, dx$, we impose the moment conditions for future purpose
\begin{equation}\label{con.psi2}
\begin{split}
\langle \psi_h \rangle = \langle \varphi_h \rangle = 0, \quad \langle \psi_h \varphi_h \rangle =1, \quad \langle \psi^2_h \varphi_h \rangle =0,  \qquad \forall h\in \cH^{i,R},\\
\langle \psi_h\rangle = 0, \quad \langle \psi_h^2 \rangle=1, \quad \varphi_h \equiv 0, 
\qquad \forall h\in \cH^{i,S}, \\
\langle \psi_h \rangle = \langle \varphi_h \rangle = 0, \quad \langle \psi_h \varphi_h \rangle =0, \quad \langle \psi^2_h \varphi_h \rangle =1,  \qquad \forall h\in \cH^{i,\Phi}.
\end{split}
\end{equation} 
Our main strategy is to use Mikado flows directed along $h\in \cH^{i,R}$ to ``cancel the error $R_q$'', while Mikado flows directed along $h\in \cH^{i, \Phi}$ will be used to ``cancel the error $\Phi_q$'', and finally Mikado flows directed along $h\in \cH^{i, S}$ will be used to ``cancel the error $S_q$''. Note that above mentioned different moment conditions \eqref{con.psi2} will play a pivotal role to employ the above strategy. 

Let us also mention the following fact which would be useful to deal with the ``transport current error'' term in Subsection~\ref{transport error}. Indeed, if $\zeta \colon \T^n \to \R$ is a zero mean scalar function, then there exists a large positive even integer $\dpot \gg 1$ and a smooth, mean-zero, adjacent-pairwise symmetric tensor potential\footnote{We use $i_j$ for $1\leq j \leq \dpot$ to denote any number in the set $\{1,\dots,n\}$.} $\vartheta^{(i_1,\dots, i_\dpot)}:\T^n \rightarrow \R^{\left(n^\dpot\right)}$ such that $\zeta(x) = \partial_{i_1}\dots\partial_{i_\dpot} \vartheta^{(i_1 \dots i_\dpot)}(x)$. For our purpose, we shall use $\zeta=\varphi_h$ or $\zeta=\psi_h$ or $\zeta=\psi_h \varphi_h $ or $\zeta=(\psi_h \varphi_h -1)$.

\subsection{Cutoffs and Mollifications}
\label{ss:regularization} 

Following the work by DeLellis-Kwon \cite{Kwon1}, we now describe the partitions of unity in both space and time variables. To do so, let us first introduce non-negative smooth functions $\{\chi_k\}_{k\in\Z^3}$ and $\{\th_p\}_{p\in \Z}$ in space $\R^3$ and in time $\R$, respectively, such that they obey the following relations:
\[
\sum_{k\in \Z^3} \chi_k^6(x) =1, \quad
\sum_{p\in \Z} \th_p^6(t) =1.   
\]
Here, $\chi_k(x) = \chi_0(x-{2}\pi k)$ where $\chi_0$ is a non-negative smooth function supported in $Q(0, {9/8}\pi)$ satisfying $\chi_0 = 1$ on $\overline{Q(0, {7/8}\pi)}$, where $Q (x, r)$ denotes the cube $\{z: |z-x|_\infty < r\}$. Similarly, $\th_p(t) = \th_0(t-p)$ where $\th_0\in C_c^\infty(\R)$ satisfies $\th_0 = 1$ on $[1/8, 7/8]$ and $\th_0 =0$ on $(-1/8,9/8)^c$. Next, we divide the integer lattice $\Z^3$ into 27 equivalent families $[k]$ with $k \in \mathbb{Z}_3^3$ via the usual equivalence relation
\[
k=(k_1,k_2,k_3) \sim \td k = (\td k_1,\td k_2,\td k_3) 
\iff k_i \equiv \td k_i\; \mod 3\quad \text{for all }i=1,2,3.
\]
We now introduce the following notation which deals with set of indices as follows:
\[
\mathscr{I} := \{(p,k, h): (p,k)\in \mathbb Z\times \mathbb Z^3, \, \mbox{and}\, h\in \mathcal{H}^{[k]}\}\, .
\]
For each $I:=(p,k, h) \in \mathscr{I}$, let $h_I$ denotes the third component of the index. Then we can further subdivide $\mathscr{I}$ into $\mathscr{I}_R\cup \mathscr{I}_\Phi \cup \mathscr{I}_S$ depending on whether $h_I\in \mathcal{H}^{[k],R}$, $h_I\in \mathcal{H}^{[k], \Phi}$ or $h_I\in \mathcal{H}^{[k], S}$. 
Moreover, we introduce a cut-off parameter $\mu = \mu_q$ with $\mu_q^{-1}\in \Z_{+}$ and recall the definition of $\tau=\tau_q$:
\begin{align}\label{mu.tau}
\mu_q^{-1} = 3 \ceil{\la_q^{\frac12}\la_{q+1}^{\frac12} \de_q^\frac14\de_{q+1}^{-\frac14} /3}, 
\quad 
\ta_q^{-1} = 40\pi {M}\eta^{-1}\cdot {\la_q^\frac12\la_{q+1}^\frac12 \de_q^\frac14\de_{q+1}^\frac14 },
\end{align} 
With this choice, we define two cut-off functions
\begin{align*}
\th_I(t) = 
\begin{cases}
\th_p^3(\tau^{-1}t), \quad I\in \mathscr{I}_R \cup \mathscr{I}_S,\\
\th_p^2(\tau^{-1}t), \quad I\in \mathscr{I}_\Phi,
\end{cases}
\quad
\chi_I(x)=
\begin{cases}
\chi_k^3(\mu^{-1}x), \quad I\in \mathscr{I}_R \cup \mathscr{I}_S,\\
\chi_k^2(\mu^{-1}x), \quad I\in \mathscr{I}_\Phi\, .
\end{cases}
\end{align*}

Next we move on to the description of the mollification process. Indeed, the loss of (spatial) derivatives is a well known problem in any convex integration scheme. To overcome this issue efficiently, one needs to mollify the inhomogeneous Euler-Reynolds equation. For this purpose, following \cite{Kwon1}, we first introduce the parameters $\ell$ and $\ell_t$, defined by
\[
\ell = \frac 1{\la_q^\frac34\la_{q+1}^\frac14} \left(\frac{\de_{q+1}}{\de_q}\right)^\frac38, \quad
\ell_t = \frac 1{\la_q^{\frac12-3\ga} \la_{q+1}^\frac12 \de_q^\frac14\de_{q+1}^\frac14 }\, .
\]
Note that $\tau_q \ll \ell_t$. Inspired by Littlewood-Paley theory, for space regularizations of $\vr_q$, $u_q$ and $p_q$, we let them go through a ``low-pass filter'' which will essentially eliminate all the waves larger than a certain scale. More precisely, we define the coarse scale density $\vr_\ell$, velocity $u_\ell$ and pressure $p_\ell$ as follows:
\begin{align}\label{def.vl}
\vr_\ell = P_{\le \ell^{-1}}\vr_{{q}}, \quad u_\ell = P_{\le \ell^{-1}}u_{{q}}, \quad p_\ell = P_{\le \ell^{-1}} p_{{q}}\, .
\end{align}
Note that $\vr_\ell, u_\ell$ and $p_\ell$ are spatially periodic functions on $J\times \mathbb T^3$, for some time domain $J \subset \R$. 

For the regularizations of the errors $R_q, \Phi_q$ and $S_q$, we need to introduce another process of regularization (namely mollify them along the flow trajectory), as depicted in \cite{Is2013}.  Indeed, such regularizations are absolutely essential, in view of sharp estimates on their advective derivatives along $u_\ell$. We recall the forward flow map $\Xi(\tau,x;t)$ with the drift velocity $u_\ell$. The flow is defined on some time interval $[a,b]$ starting at the initial time $t\in [a,b)$ and given by
\begin{align}\label{e:fflow_first_instance}
\begin{cases}
\pa_\tau \Xi(\tau, x;t)= u_\ell (\tau, \Xi(\tau, x;t))\\
\Xi(t, x;t) = x\, .
\end{cases}
\end{align}
With the help of the forward flow map, we can define the mollification along the trajectory
\begin{align*}
(\omega_\delta \ast_{\Xi} F) (t,x)
= \int_{\R} F(t+s,\Xi(t+s,x;t))  \omega_\delta (s) ds,
\end{align*}
where $\omega$ is a standard moolifier in $\R$ satisfying usual conditions $\norm{\omega}_{L^1(\R)}=1$, $\supp\omega\subset (-1,1)$, and we set $\omega_\delta (s) = \delta^{-1} \omega( \delta^{-1} s)$ for any $\de>0$. The main advantage of this type of mollification is that it satisfies 
\begin{align*}
D_{t,\ell} (\omega_\delta \ast_{\Xi} F)(t,x) & = \int_{\R} (D_{t,\ell}F)(t+s, \Xi(t+s, x;t)) \omega_\delta (s) ds \\ 
& = - \int_{\R} F(t+s, \Xi(t+s, x;t)) \omega'_\delta (s) ds\, .
\end{align*}
Here we have used the notation $D_{t, \ell} := \pa_t + u_\ell \cdot \na$.
Note that if $F$ and $u_\ell$ are well defined on some time interval $[c,d]$, then $\omega_\delta \ast_{\Xi} F$ is well defined on $[c,d]-\delta$.

Finally, we can write the regularized errors as 
\begin{align}\label{def.me}
R_\ell = \omega_{\ell_t} \ast_{\Xi} P_{\leq \ell^{-1}} R_{{q}}, \quad
\Phi_\ell = \omega_{\ell_t} \ast_{\Xi} P_{\leq \ell^{-1}}\Phi_{{q}}, \quad S_\ell = \omega_{\ell_t} \ast_{\Xi} P_{\leq \ell^{-1}} S_{{q}}
\end{align}
It is easy to check that these errors can be defined on $[0,T]+2\tau_q$, thanks to a sufficiently large choice of the parameter $\la_0$. 

We now focus on deriving estimates on the difference between the original and the regularized tuple. Moreover, we also need precise estimates on higher order derivatives of the regularized tuple. In what follows, we first set the notation: $\norm{\cdot}_N = \norm{\cdot}_{C^0([0,T]+\tau_q; C^N(\T^3))}$, and then collect several a-priori estimates on the regularized tuple.
\begin{lem}
\label{lemimp}
There exists $\bar{b}(\al)>1$ such that
for any $b\in (1,\bar b(\al))$ it is possible to find $\La(\al,b,M)$ satisfying the following property: If $\la_0\geq \La$, then the regularized tuple given in \eqref{def.vl} and \eqref{def.me} satisfies the following estimates:
\begin{align*}
&\norm{\vr_\ell}_N \lec_N \ell^{1-N} \lambda_q \de_{q}^\frac12 \, \qquad \quad \,\,\,\,\norm{u_\ell}_N \lec_N \ell^{1-N} \lambda_q \de_{q}^\frac12, \qquad \qquad \,\,\,\,\forall N\geq 1,\\
&\norm{D_{t,\ell}^s R_\ell}_0 \lec_{s} \ell_t^{-s}\la_q^{-2\ga} \de_{q+1}, \quad
\norm{D_{t,\ell}^s \Phi_\ell}_0 \lec_{s} \ell_t^{-s}  \la_q^{-4\ga} \de_{q+1}^\frac32,
\quad
\norm{D_{t,\ell}^s S_\ell}_0 \lec_{s} \ell_t^{-s}  \la_q^{-2\ga} \de_{q+1}, \quad \forall s\geq 0. 
\end{align*}
Moreover, using the fact that if $\la_0\geq \La_0$, then $|\na^{N+1} \Xi(t+s,x;t)|\lec_M  \ell^{-N}$ holds for $N\geq 0$ and $s\le M_*$, the following estimates hold:
\begin{align}
&\ell_t^s\norm{D_{t,\ell}^s R_\ell}_N {\lec_{s,N,M}} \ \ell^{-N} \la_q^{-2\ga} \de_{q+1}, \quad 
\ell_t^s\norm{D_{t,\ell}^s \Phi_\ell}_N 
{\lec_{s,N,M}} \ \ell^{-N} \la_q^{-4\ga} \de_{q+1}^\frac32
\label{est.mph}, \\
& \ell_t^s\norm{D_{t,\ell}^s S_\ell}_N {\lec_{s,N,M}} \ \ell^{-N} \la_q^{-2\ga} \de_{q+1} \label{est.mS}
\end{align}
\end{lem}
\begin{proof}
For the detailed computation, see \cite[Section18]{Is2013}. 
\end{proof}

Next, we collect estimates on the differences between the regularized objects $(\vr_\ell, u_\ell, p_\ell, R_\ell, \Phi_\ell, S_\ell)$ and their original counterparts. 

\begin{lem}
There exists $\bar{b}(\al)>1$ such that for any $b\in (1,\bar b(\al))$, it is possible to find $\La(\al,b,M)$ with the following property: If $\la_0\geq \La$ and $N\in \{0,1,2\}$ then following estimates on the differences hold:
	\begin{align}
	&\norm{\vr_q-\vr_\ell}_N + \norm{u_q-u_\ell}_N  
	\lec_M \ell^{2-N} \la_q^2 \de_q^\frac12, \label{est.r.dif} \\
	&\norm{D_{t,\ell}(\vr_q-\vr_\ell)}_{N-1}  
	+ \norm{D_{t,\ell}(u_q-u_\ell)}_{N-1}
	\lec_M \ell^{2-N} \la_q\de_q^{\frac12} \tau_q^{-1}, \label{est.v.dif} \\
	&\norm{R_q-R_\ell}_N+ 
	\de_{q+1}^{-\frac12}\norm{D_{t,\ell} (R_q-R_\ell)}_{N-1} 
	\lec_M
	\la_{q+1}^N {\la_q^\frac12}{\la_{q+1}^{-\frac12}}
	\de_q^\frac14 \de_{q+1}^\frac34
	\label{est.R.dif} \\
	&\norm{S_q-S_\ell}_N+ 
	\de_{q+1}^{-\frac12}\norm{D_{t,\ell} (S_q-S_\ell)}_{N-1} 
	\lec_M
	\la_{q+1}^N {\la_q^\frac12}{\la_{q+1}^{-\frac12}}
	\de_q^\frac14 \de_{q+1}^\frac34
	\label{est.S.dif} \\
	&\norm{\Phi_q-\Phi_\ell}_N+\de_{q+1}^{-\frac12}\norm{D_{t,\ell} (\Phi_q-\Phi_\ell)}_{N-1}
	\lec_M  
	\la_{q+1}^N{\la_q^\frac12}{\la_{q+1}^{-\frac12}}\de_q^\frac14 \de_{q+1}^\frac54.\label{est.ph.dif} 
	\end{align}
	Additionally, the following estimate hold for all $N \le N_*$ and $s \le M_*$:
	\begin{align}
	\norm{D^s_{t,\ell}(\varrho_q-\varrho_\ell)}_{N}  \lec_M \ell^{1-N} \la_q \de_q^\frac12 (\tau_q^{-1})^s, \quad
	\norm{D^s_{t,\ell}(u_q-u_\ell)}_{N}  \lec_M \ell^{1-N} \la_q \de_q^\frac12 (\tau_q^{-1})^s. \label{est.u.new}
	\end{align}
\end{lem}

\begin{proof} 
The proof of estimates \eqref{est.r.dif}-\eqref{est.ph.dif} follow from \cite{Kwon1}, modulo cosmetic changes. For the proof of the estimate \eqref{est.u.new} we can apply the mollification estimate \eqref{est.final01} in Appendix~\ref{apenC}, with $f=\varrho_q\, \text{and}\, u_q$, $v=u_\ell$.
\end{proof}

\subsection{Quadratic commutator} We provide a quadratic commutator estimate, which is a version of the estimate in \cite[Lemma 5.2]{Kwon1}. The commutator appears due to the convloution the density equation. Indeed, after mollifying the density equation in \eqref{app.eq}, we have
\begin{align*}
\partial_t {\vr}_\ell + \div ({\vr_\ell} {u_\ell})  &= - \div P_{\le \ell^{-1}} R_q + \, Q(\vr_q,u_q),
\end{align*}
where $Q(\vr_q,u_q) = \div ({\vr_\ell} {u_\ell} - (\vr_q u_q)_\ell)$.

\begin{lem}
\label{lem:est.Qvv} 
Assume that the parameters $\bar{b}(\al)>1$ and $\La(\al,b,M)$ be as in the previous Lemma. Then for any $b\in (1,\bar b(\al))$, and
any $N\geq 0$, $Q(\vr_q,u_q)$ satisfies 
	\begin{align}\label{est.Qvv}
	\norm{Q(\vr_q,u_q)}_N\lec_M \ell^{1-N} (\la_q\de_q^\frac12)^2, \quad
	\norm{D_{t,\ell} Q(\vr_q,u_q)}_{N} \lec_M \ell^{1-N} (\la_q\de_q^\frac12)^2 \tau_q^{-1}. 
	\end{align}
\end{lem}

\begin{proof} In order to simplify our notation we drop the subscript $q$ fom $\vr_q$ and $u_q$. The estimate for $\norm{Q(\vr,u)}_N$ easily follows from \eqref{est.com}. To estimate $D_{t,\ell}Q(\vr,u)$, we first decompose $Q(\vr,u)$ into
\begin{align*}
Q(\vr,u)
= (u_\ell - u) \cdot \na \vr_\ell + [u \cdot \na, P_{\le \ell^{-1}} ] \vr.
\end{align*}
Recall that $P_{\le \ell^{-1}} f(x) = \int_{\R^3} f(x-y) \widecheck{m} _\ell(y)dy$, where $\widecheck{m} _\ell(x)= 2^{3J} \widecheck{m}(2^J x)$. By making use of this formula, the advective derivative of the commutator term can be written as follows,
\begin{align*}
    & D_{t,\ell}[u\cdot \na, P_{\le \ell^{-1}}] \vr
	=  (\pa_t +u _\ell(x)\cdot \na) \int (u(x)-u(x-y))\cdot \na \vr (x-y) \widecheck{m} _\ell(y) dy\\
	&=  \int  (D_{t,\ell}u(x)-D_{t,\ell}u(x-y)) \cdot  \na \vr(x-y) \widecheck{m} _\ell(y) dy\\
	&\quad-\int (u _\ell(x)-u _\ell(x-y))_a\na_a u_b(x-y)  \na_b \vr(x-y) \widecheck{m} _\ell(y) dy
	+ \int (u(x)-u(x-y))\cdot D_{t,\ell}\na  \vr(x-y) \widecheck{m} _\ell(y) dy\\
	&\qquad + \int (u(x)-u(x-y))_a (u _\ell(x)-u _\ell(x-y))_b (\pa_{ab} \vr)(x-y) \widecheck{m} _\ell(y) dy. 
	\end{align*} 
	Based on the decompositions, we use \eqref{est.vp}, \eqref{est.v.dif}, and $\norm{|y|^n\widecheck{m} _\ell}_{L^1(\R^3)}\lec \ell^n$, $n\geq 0$, to get
	\begin{align*}
	& \norm{D_{t,\ell}Q(\vr,u)}_0
	\lec \norm{D_{t,\ell}(u-u _\ell)}_0  \norm{\na \vr_\ell}_0 + \norm{u-u _\ell}_0\norm{D_{t,\ell}\na \vr _\ell}_0 \\
	&\quad+ \ell\norm{\na D_{t,\ell} u}_0\norm{\na \vr}_0
	+\ell\norm{\na u}_0^2 \norm{\na \vr}_0+ \ell \norm{\na u}_0 \norm{D_{t,\ell} \na \vr}_0 + \ell^2 \norm{\na u}_0^2 \norm{\na^2 \vr}_0 
	\lec \ell (\la_q\de_q^\frac12)^2 \tau_q^{-1}
	\end{align*}
	Here, we use $\norm{\na D_{t,\ell} u _\ell}_0 \leq \norm{\na D_{t,\ell} (u-u _\ell)}_0+\norm{\na D_{t} u}_0 +\norm{\na (((u-u_\ell)\cdot \na) u _\ell)}_0\lec \la_q\de_q^\frac12 \tau_q^{-1} $, and
	\begin{align*}
	\norm{D_{t,\ell} \na u_\ell }_0 \leq \norm{\na D_{t,\ell} u _\ell}_0 + \norm{\na u}_0^2 \lec \la_q\de_q^\frac12 \tau_q^{-1}.
	\end{align*}
	In the case of $N\geq 1$, we remark that $D_{t,\ell}Q(\vr,u)$ has frequency localized to $\lec \ell^{-1}$, so that the remaining estimates follows from the Bernstein inequality. Similarly, we also have
	\begin{align}\label{est.Dtvl}
	\norm{D_{t,\ell} \na u_\ell }_N \lec \ell^{-N} \la_q\de_q^\frac12 \tau_q^{-1}, 
	\quad \text{and}\quad \norm{D_{t,\ell} \na \vr_\ell }_N \lec \ell^{-N} \la_q\de_q^\frac12 \tau_q^{-1}.
	\end{align} 
	This finishes the proof of the Lemma.
\end{proof}

\subsection{Higher material derivative estimates}
Since our analysis fundamentally depends on the local inverse divergence result (Proposition~\ref{prop:intermittent:inverse:div}), which in turn necessitates iterating higher-order material derivatives, we must first establish the following estimates for these higher-order material derivatives.

\begin{lem}
\label{lem:est.high} 
There exists $\bar{b}(\al)>1$ such that for any $b\in (1,\bar b(\al))$, it is possible to find $\La(\al,b,M)$ with the following property: If $\la_0\geq \La$, then the following estimate hold for all $N \le N_*$ and $s \le M_*$:
	\begin{align}
	\norm{D^s_{t,\ell} D \varrho_\ell}_{N}  \lesssim \ell^{-N} \la_q \de_q^\frac12 (\tau_q^{-1})^s, \quad
	\norm{D^s_{t,\ell} D u_\ell}_{N}  \lesssim \ell^{-N} \la_q \de_q^\frac12 (\tau_q^{-1})^s.\label{est.dtu.new}
	\end{align}
\end{lem}

\begin{proof}
We prove only the velocity estimate, as the density estimate can be treated analogously. To prove the second estimate in \eqref{est.dtu.new}, we apply mathematical induction principle on $s$. To that context, note that, we already have the estimate for $s=0$ and $s=1$, thanks to the Lemma~\ref{lemimp} and the first estimate of \eqref{est.Dtvl}. Next, we assume that the estimate is true for the exponent $s-1$ and for all $N \le N_*$. Notice that it is enough to establish the result for $N=0$. In fact, in the case of $N\geq 1$, we remark that $D^s_{t, \ell} D u_{\ell} $ has frequency localized to $\lec \ell^{-1}$, so that the remaining estimates follows from the Bernstein inequality. In what follows, we write 
$$
D^s_{t, \ell} D u_{\ell} =  D D^s_{t, \ell} u_{\ell} + [D^s_{t, \ell}, D] u_\ell.
$$
We only need to verify estimate for $N=0$, and we have
$$
\|D^s_{t, \ell} D u_{\ell}\|_0 \le  \|D^s_{t, \ell} u_{\ell}\|_1 + \|[D^s_{t, \ell}, D] u_\ell \|_0
$$
To deal with first term on the right hand side of the above expression, we make use of Remark~\ref{rem.new} to conclude
$ \|D^s_{t, \ell} u_{\ell}\|_1 \le \la_q \de_q^{1/2} (\tau_q^{-1})^{s}$. 
For the commutator estimate, we use the commutator Lemma~\ref{lem:Komatsu} in Appendix~\ref{comm01} with $n=1$, $m=s$, and 
$D_t = D_{t,\ell}$ to conclude
$$
\| [D^s_{t, \ell}, D] u_\ell\|_0 \le \| \nabla u_\ell\|_0 (\tau_q^{-1})^{s} + \| \nabla u_\ell\|_0 \| D^{s-1}_{t, \ell} D u_{\ell} \|_0
\le \la_q \de_q^{1/2} (\tau_q^{-1})^{s}.
$$
\end{proof}

\subsection{Backward flow}
Here we describe the backward flow map. For a given $I= (p,k,h) \in \mathscr{I}$, we denote $\xi_I = \xi_p$ which solves the following equation
\begin{align}\label{def.bflow}
\begin{cases}
\pa_t \xi_p + (u_\ell \cdot \na)\xi_p =0, \\
\xi_p(t_p,x) =x,
\end{cases}
\end{align}
where $t_p = \tau p$. Observe that 
$$
\Xi(t, \xi_p(t,x); t_p) =x,
$$
where $\Xi(t,x; t_p)$ is the forward flow map defined in \eqref{e:fflow_first_instance}.

In the rest of the paper $\nabla \xi_I$ will denote the Jacobi Matrix of the partial derivatives of the components of the vector map $\xi_I$ and we will use the shorthand notations $\nabla \xi_I^\top$, $\nabla \xi_I^{-1}$ and $\nabla \xi_I^{-\top}$ for, respectively, its transpose, inverse and transpose of the inverse. Moreover, for any vector $f\in \mathbb R^3$ and any matrix $A\in \mathbb R^{3\times 3}$ the notation $\nabla \xi_I f$ and $\nabla \xi_I A$ (resp. $\nabla \xi_I^{-1} f$, etc.) will be used for the usual matrix product, regarding $f$ as a column vector (i.e. a $\mathbb R^{3\times 1}$-matrix). Next, we collect some useful estimates on the backward flow map in the following lemma.

\begin{lem}
\label{lem:est.flow} 
Given any $b>1$, there exists $\La=\La(b)$ such that for all $N \le N_*$, $s \le M_*$ and $\la_0\geq \La$ the backward flow map $\xi_I$ satisfies the following estimates on $\cal{I}_p = [t_p - \frac 12 \tau_q , t_p + \frac 32\tau_q]\cap [0,T]+2\tau_q$
\begin{align}
&\norm{\I - \na \xi_I}_{0} \leq \frac 15, \,\, \norm{\I - \na \xi_I^{-1}}_{0} \leq 2M \tau_q \la_q \de_q^\frac12, \,\,
\norm{\na \xi_I}_{N} + \norm{(\na \xi_I)^{-1}}_{N} \lec_N \ell^{-N}, \label{est.flow2} \\
&\norm{D_{t,\ell}^s (\na \xi_I)^{-1}}_{N} 
\lec_{N,M} \ell^{-N} (\tau_q^{-1})^s, \qquad 
\norm{D_{t,\ell}^s \na \xi_I}_{N}\lec_{N,M} \ell^{-N} (\tau_q^{-1})^s. \label{est.flow.indM}
\end{align}
\end{lem}
\begin{proof} 
For a proof of the estimates in \eqref{est.flow2}, consult \cite{Kwon1}. To show other estimates in \eqref{est.flow.indM}, we make use of \eqref{est.dtu.new} and the commutator Lemma~\ref{lem:Komatsu} in Appendix~\ref{comm01} with $n=1$, $m=s$, and $D_t = D_{t,\ell}$. Indeed, we first write, thanks to \eqref{def.bflow}
$$
D_{t,\ell}^s \na \xi_I = \na D_{t,\ell}^s  \xi_I + [D_{t,\ell}^s, \na] \xi_I = [D_{t,\ell}^s, \na] \xi_I,
$$
and 
$$
\| [D^s_{t, \ell}, \na] \xi_I\|_N \le \| \nabla \xi_I\|_N (\tau_q^{-1})^s + \| \nabla \xi_I\|_N \| D^{s-1}_{t, \ell} D u_{\ell} \|_0
\le \ell^{-N}(\tau_q^{-1})^s.
$$
We can use same argument to establish the other part of estimate \eqref{est.flow.indM}. This finishes the proof.
\end{proof}

\subsection{Building Blocks}
As we mentioned before, we will make use of Mikado densities and Mikado fields to construct our building blocks for the solutions to inhomogenous incompressible Euler equations. More precisely, in order to reduce the error at the level of $q+1$, we need to add suitable perturbation to the density, velocity and pressure at the level $q$, i.e, $(\vr_q, u_q, p_q)$. Indeed, we wish to define our perturbation $u_{q+1} - u_q$, and  $\vr_{q+1} - \vr_q$, respectively as
\[
\sum_{h\in \mathcal{H}} \Gamma_h (R_q(t,x), \Phi_q (t,x), S_q (t,x)) U_h (\lambda (x-s_h)), \,\, \text{and}\,\, \sum_{h\in \mathcal{H}} \Gamma_h (R_q(t,x), \Phi_q (t,x), S_q(t,x)) V_h (\lambda (x-s_h)),
\]
where the coefficients (weights) $\Gamma_h$ are chosen smooth functions, $\lambda$ is a large parameter and the shifts $s_h$ are chosen to maintain the disjoint support condition of Lemma~\ref{l:Mikado}.

Given any $I = (p,k,h) \in \mathscr{I}$, we shall choose appropriate shift $z_I$ such that
\[
z_I= z_{p,k} + s_h \in \mathbb R^3, \quad z_{p,k} = z_{p,k'}, \quad \mbox{if ${\mu}(k-k')\in 2\pi \mathbb Z^3$.}
\] 
Note that $U_h (\cdot - z_I)$, $V_h (\cdot - z_I)$, $\varphi_h (\cdot - z_I)$ and $\psi_h (\cdot - z_I)$ are all periodic functions. To simplify our notations, we shall denote 
$\varphi_I$ for $\varphi_{h_I} (\cdot - \la z_I)$, $\psi_I$ for $\psi_{h_I} (\cdot - \la z_I)$ and $\widetilde{h_I}$ for $\nabla \xi_I^{-1} h_I$. We also assume supports of all the above mentioned periodic functions contained in a small neighborhood of 
\begin{equation}
l_h+s := \left\{x\in \mathbb T^3: \left(x -\Sigma h - s\right) \in 2\pi \mathbb Z^3, \quad \mbox{for some $\Sigma \in \mathbb R$}\right\}\, .
\end{equation} 
More precisely, we assume that
\begin{align}
\supp (\psi_I) \subset B \left(l_h, \frac{\eta}{10}\right) &
{ := \{z\in \R^3: |z-y|<\textstyle{\frac {\eta}{10}}, \text{ for some } y\in l_h\} }
\, , \label{e:eta} \\
\supp (\varphi_I) \subset B \left(l_h, \frac{\eta}{10}\right) &
{ := \{z\in \R^3: |z-y|<\textstyle{\frac {\eta}{10}}, \text{ for some } y\in l_h\} }
\, ,\label{f:eta}
\end{align}
where $\eta$ is a geometric constant to be specified later, cf. Proposition \ref{p:supports}. Following \cite{Kwon1}, we now give the details of the perturbations. We first introduce the scalar maps
\[
w_I (t,x) := \theta_I (t) \chi_I (\xi_I (t,x)) \psi_I (\lambda (\xi_I (t,x))), \quad 
\vartheta_I (t,x) := \theta_I (t) \chi_I (\xi_I (t,x)) \varphi_I (\lambda (\xi_I (t,x))). 
\]
With the help of these scalar maps, we can specify the principle parts of our perturbations $u_{q+1} - u_q$ and $\vr_{q+1} - \vr_q$, respectively, as
\begin{align}
\label{eqnew11}
w_0 :=  \sum_{I\in \mathscr{I}} a_I  \widetilde{h_I} w_I, \quad 
\theta_0 :=  \sum_{I\in \mathscr{I}} b_I    \vartheta_I,
\end{align}
where the weights $a_I, b_I$'s are appropriately chosen smooth scalar functions, to be specified in the next subsection. We set $\lambda= \lambda_{q+1}$, and introduce the main part of our perturbation
\begin{align}
W (R, \Phi, S, t,x) & := \sum_{I\in \mathscr{I}} \theta_I (t) \chi_I (\xi_I (t,x)) {a}_I (R, \Phi, S, t,x) \nabla \xi_I^{-1} (t,x) U_{h_I} (\lambda (\xi_I (t,x) - z_I)), \label{e:master}\\
\Theta (R, \Phi, S, t,x) & := \sum_{I\in \mathscr{I}} \theta_I (t) \chi_I (\xi_I (t,x)) {b}_I (R, \Phi, S, t,x)  V_{h_I} (\lambda (\xi_I (t,x) - z_I)). \label{f:master}
\end{align}
We can recast them as
\begin{equation}\label{e:master2}
W := \sum_{I\in \mathscr{I}} \theta_I \chi_I (\xi_I) {a}_I \widetilde{h_I} \psi_I (\lambda \xi_I), \quad \Theta := \sum_{I\in \mathscr{I}} \theta_I \chi_I (\xi_I) {b}_I  \varphi_I (\lambda \xi_I)
\end{equation}
Observe that $W, \Theta$ are periodic functions of $x$, thanks to the fact that $z_{p,k} = z_{p,k'}$, if $\mu(k-k')\in 2\pi \mathbb Z^3$.
	
Finally, the main part of the corrections $u_{q+1}-u_q$, and $\vr_{q+1} -\vr_q$ will take the form, respectively,
\begin{align}
w_0 (t,x) &:= W (R_\ell (t,x), \Phi_\ell (t,x), S_\ell (t,x), t,x), \,\, \text{and} \label{e:wo} \\
\theta_0 (t,x) &:= \Theta (R_\ell (t,x), \Phi_\ell (t,x), S_\ell (t,x), t,x), \label{f:wo}
\end{align}
which are well-defined on $[0,T]+ 2\tau_q$. Rest of the analysis in this section \ref{sec:correction.const} are done in the time interval $[0,T]+ 2\tau_q$, without specifically mentioning about it.
For future purpose, we also need to rewrite the functions $w_o$ and $\theta_0$ in more elementary pieces. Recall that we subdivided $\mathscr{I}$ into $\mathscr{I}_R\cup \mathscr{I}_\Phi \cup \mathscr{I}_S$ depending on whether $h_I\in \mathcal{H}^{[k],R}$, $h_I\in \mathcal{H}^{[k], \Phi}$ or $h_I\in \mathcal{H}^{[k], S}$. In a similar fashion, 
we write $\theta_0 = \Theta^R + \Theta^{\Phi}$ and $w_0 = W^R + W^S + W^{\Phi}$, where $\Theta^R$ represents the sum \eqref{f:master} over $\mathscr{I} \in \mathscr{I}_R$. Similarly, we define other elementary pieces present in $w_0$ and $\theta_0$.

Finally, we state the crucial lemma which guarantees that different Mikado densities and fields do not interact with each other. For a proof of this lemma, modulo cosmetic changes, we refer to the work by DeLellis-Kwon \cite[Proposition 3.5]{Kwon1}.
\begin{prop}\label{p:supports}
There is a constant $\eta = \eta (\mathcal{H})$ in \eqref{e:eta} such that it allows a choice of the shifts $z_I= z_{p,k} + \bar s_h$ which ensure that {for each $(\mu_q, \tau_q, \la_{q+1})$,} the conditions $\supp (w_I) \cap \supp (w_J) =\emptyset$, $ \supp (w_I) \cap \supp (\vartheta_J)=\emptyset$ for every $I\neq J$ and $z_{p,k} = z_{p,k'}$, if $\mu(k-k')\in 2\pi \mathbb Z^3$ for every $p,k$ and $k'$. 
\end{prop}

\subsection{Choice of the weights}\label{ss:weights} 
Here we give details about the construction of weights. Let us first briefly describe the idea for the construction. In fact, we do it in several steps.
\medskip

\noindent {\bf Step-I}. Define $\Theta^{\Phi}$ and $ W^{\Phi}$ in such a way that
$$
\Big((\theta_0 w_0 \otimes w_0) + \Phi_\ell -  \la_q^{-3\gamma}\delta_{q+1}^{3/2}   \mathrm{Id}\Big)_L \equiv 0,
$$
where for any function $h$, we denote by $(h)_L$ the low frequency part of the function $h$.

\noindent {\bf Step-II}. Next, we define $\Theta^R$ and $W^R$ so that we have the following cancellation:
$$
\Big(\theta_0 w_0 - R_\ell \Big)_L \equiv 0.
$$

\noindent {\bf Step-III}. Finally, we calculate the low-frequency parts of $(W^R \otimes W^R)$ and $(W^{\Phi} \otimes W^{\Phi})$ . Then we define $W^S$ in such a way that the following cancellation holds:
$$
\Big( w_0 \otimes w_0 + S_\ell - \delta_{q+1} \mathrm{Id} \Big)_L \equiv 0.
$$

Keeping these steps in mind, we are now ready to calculate various weights associated to the density and velocity perturbations. 

\subsubsection{Cubic non-linearity weights} 

Our aim is to find weights $a_I$ and $ b_I$ such that the low frequency part of $\theta_0 w_0 \otimes w_0$ makes a cancellation with the current $\Phi_\ell -\delta^{3/2}_{q+1} \mathrm{Id}$. Note that
\begin{align*}
\theta_0 w_0 \otimes w_0 &= \sum_{I\in \mathscr{I}} \th_I^3 \chi_I^3(\xi_I)   a^2_I b_I \langle \varphi_I \psi_I^2 \rangle (\widetilde{h_I} \otimes \widetilde{h_I}) \\
& \qquad + \sum_{I\in \mathscr{I}} \th_I^3 \chi_I^3(\xi_I)  a^2_I b_I (\varphi_I \psi_I^2(\la_{q+1}\xi_I) - \langle \varphi_I \psi_I^2 \rangle) (\widetilde{h_I} \otimes \widetilde{h_I}) := (\theta_0 w_0 \otimes w_0)_L + (\theta_0 w_0 \otimes w_0)_H.
\end{align*}
Observe that $\widetilde{h_I} \otimes \widetilde{h_I} = \nabla \xi_I^{-1} ({h_I} \otimes {h_I}) \nabla \xi_I^{-T}$. In order to find the desired $a_I, b_I$, thanks to \eqref{con.psi2}, it suffices to achieve
\begin{align}\label{W3L}
(\theta_0 w_0 \otimes w_0)_L
=\sum_{p,k} \th_p^6\left(\frac t{\tau_q}\right) \chi_k^6\left(\frac{\xi_p}{\mu_q}\right)\sum_{I \in \mathscr{I}_{\Phi}}  a_I^2 b_I \widetilde{h_I} \otimes \widetilde{h_I} \, .
\end{align}
In other words, we demand
\begin{align*}
\sum_{I \in \mathscr{I}_{\Phi}} a^2_I b_I ({h_I} \otimes {h_I}) = \nabla \xi_I ( \la_q^{-3\gamma}\delta_{q+1}^{3/2} \mathrm{Id} - \Phi_\ell  ) \nabla \xi_I^{T}
\end{align*}
To achieve this, we define $\mathcal{M}_I:=  \la_q^{-3\gamma} \delta_{q+1}^{3/2} [\nabla \xi_I  \nabla \xi_I^{T} - \mathrm{Id}] - \nabla \xi_I  \Phi_\ell \nabla \xi_I^{T}$, choose $a^2_I b_I =  \la_q^{-3\gamma} \delta_{q+1}^{3/2} \Gamma^2_I$ and impose
\begin{align*}
\sum_{I \in \mathscr{I}_{\Phi}} \Gamma^2_I ({h_I} \otimes {h_I})  = \mathrm{Id} + \la_q^{3\gamma}\delta_{q+1}^{-3/2} \mathcal{M}_I 
\end{align*}
To make sure that such a choice is possible, observe that we can make $\norm{\la_q^{3\gamma}\de_{q+1}^{-3/2}\mathcal{M}_I}_{C^0(\supp(\th_I)\times \R^3)}$ sufficiently small, provided that $\la_0$ is sufficiently large, because of Lemma~\ref{lem:est.flow}, and $\norm{\la_q^{3\gamma}\de_{q+1}^{-3/2}\Phi_\ell}_0\lec \la_q^{-\ga}$. We can thus apply Lemma \ref{lem:geo1} to $\mathcal{H}^{[k],\Phi}$ and, denoting by $\Gamma_{h_I}$ the corresponding functions, we just need to set 
\[
\Ga_I=\Ga_{h_I} ( \I + \la_q^{3\gamma}\de_{q+1}^{-3/2} \mathcal{M}_I)\, .
\] 
Therefore, for $I\in \mathscr{I}_\Phi$, we may choose $a_I$ and $b_I$ in terms of $\Gamma_I$. Indeed we choose
\begin{align}
\label{one}
a_I =  \la_q^{-\gamma}\delta_{q+1}^{1/2} \Gamma_I^{2/3}, \quad b_I =  \la_q^{-\gamma}\delta_{q+1}^{1/2} \Gamma_I^{2/3}
\end{align}
Above choice of weights give us the desired ``cancellation property'':
\begin{align*}
(\theta_0 w_0 \otimes w_0)_L
=\sum_{p,k} \th_p^6\left(\frac t{\tau_q}\right) \chi_k^6\left(\frac{\xi_p}{\mu_q}\right)(\la_q^{-3\gamma}\delta_{q+1}^{3/2} \mathrm{Id} - \Phi_\ell) = \la_q^{-3\gamma} \delta_{q+1}^{3/2} \mathrm{Id} - \Phi_\ell 
\end{align*}

\subsubsection{Transport non-linearity weights}\label{subsec:R}

Here idea is to find weights such that the low frequency part of the product $\theta_0 w_0$ makes a cancellation with the transport error term $R_\ell $. Note that
\begin{align*}
\theta_0 w_0 &= \sum_{I} \th_I^2 \chi_I^2(\xi_I)  a_I b_I  \langle \varphi_I \psi_I \rangle \widetilde{h_I} + \sum_{I} \th_I^2 \chi_I^2(\xi_I) a_I b_I (\varphi_I \psi_I - \langle \varphi_I \psi_I \rangle) \widetilde{h_I}  := (\theta_0 w_0)_L + (\theta_0 w_0)_H.
\end{align*}
Again, in view of \eqref{con.psi2}, we demand that
\begin{align*}
\sum_{I\in \mathscr{I}_{R}} a_I b_I {h_I} = \nabla \xi_I R_\ell. 
\end{align*}
Next, we choose $a_I b_I = \lambda_q^{-2\gamma}\delta_{q+1} \Gamma_I$ and impose
\begin{align*}
\sum_{I\in \mathscr{I}_{R}} \Gamma_I {h_I}  = \lambda_q^{2\gamma}\delta_{q+1}^{-1} \nabla \xi_I R_\ell  
\end{align*}
Again such a choice is possible since $\| \lambda_q^{2\gamma}\delta_{q+1}^{-1} \nabla \xi_I R_\ell \|_0 \le C$. We can thus apply Lemma~\ref{lem:geo2} to $ \mathcal{H}^{[k],R}$ and, denoting by $\Gamma_{h_I}$ the corresponding functions, we just need to set 
\[
\Ga_I(t,x)=\Ga_{h_I} ( \lambda_q^{2\gamma}\delta_{q+1}^{-1} \nabla \xi_I R_\ell  )\, .
\] 
Therefore, for $I\in \mathscr{I}_R$, we may choose $a_I$ and $b_I$ in terms of $\Gamma_I$. Indeed we choose
\begin{align}
\label{two}
a_I = \lambda_q^{-\gamma} \delta_{q+1}^{1/2} \Gamma_I^{1/2}, \quad b_I =\lambda_q^{-\gamma}  \delta_{q+1}^{1/2} \Gamma_I^{1/2}
\end{align}
Finally notice that above choice of weights give us the desired ``cancellation property'':
\begin{align*}
(\theta_0 w_0)_L =  \sum_{p,k} \th_p^6\left(\frac t{\tau_q}\right) \chi_k^6\left(\frac{\xi_p}{\mu_q}\right) R_\ell = R_\ell.
\end{align*}

\subsubsection{Quadratic non-linearity weights}

Here idea is to find weights such that the low frequency part of $w_0 \otimes w_0$ makes a cancellation with the quadratic term $S_\ell  -\delta_{q+1} \mathrm{Id}$. First we calculate
\begin{align*}
w_0 \otimes w_0 &= \sum_{I\in \mathscr{I}} \th_I^2 \chi_I^2(\xi_I) a^2_I \langle \psi_I^2 \rangle (\widetilde{h_I} \otimes \widetilde{h_I}) + \sum_{I\in \mathscr{I}} \th_I^2 \chi_I^2(\xi_I) a^2_I ( \psi_I^2 - \langle  \psi_I^2 \rangle) (\widetilde{h_I} \otimes \widetilde{h_I}) \\
&:= ( w_0 \otimes w_0)_L + ( w_0 \otimes w_0)_H.
\end{align*}
Making use of \eqref{con.psi2}, and noting  that $\widetilde{h_I} \otimes \widetilde{h_I} = \nabla \xi_I^{-1} ({h_I} \otimes {h_I}) \nabla \xi_I^{-T}$, we want that
\begin{align}
\label{eq.Ga}
\sum_{I\in \mathscr{I}_{S}} & a^2_I  ({h_I} \otimes {h_I}) = \nabla \xi_I \Big[\delta_{q+1} \mathrm{Id} - S_\ell \Big] \nabla \xi_I^{T} \\
&\qquad - \nabla \xi_I \Big[ \sum_{J\in \mathscr{I}_{R}} \th_J^2 \chi_J^2(\xi_J) a^2_J \langle \psi_J^2 \rangle (\widetilde{h_J} \otimes \widetilde{h_J}) - \sum_{J\in \mathscr{I}_{\Phi}} \th_J^2 \chi_J^2(\xi_J) a^2_J \langle \psi_J^2 \rangle (\widetilde{h_J} \otimes \widetilde{h_J})\Big] \nabla \xi_I^{T}. \nonumber
\end{align}
Next, we define 
\begin{align*}
\mathcal{N}_I:&= \delta_{q+1} [\nabla \xi_I  \nabla \xi_I^{T} - \mathrm{Id}] - \nabla \xi_I  S_\ell \nabla \xi_I^{T} \\
&\qquad - \nabla \xi_I \Big[ \sum_{J\in \mathscr{I}_{R}} \th_J^2 \chi_J^2(\xi_J) a^2_J \langle \psi_J^2 \rangle (\widetilde{h_J} \otimes \widetilde{h_J}) - \sum_{J\in \mathscr{I}_{\Phi}} \th_J^2 \chi_J^2(\xi_J) a^2_J \langle \psi_J^2 \rangle (\widetilde{h_J} \otimes \widetilde{h_J})\Big] \nabla \xi_I^{T},
\end{align*}
choose $a^2_I = \delta_{q+1} \Gamma^2_I $ and impose
\begin{align}
\label{eq.Ga1}
\sum_{I\in \mathscr{I}_{S}} \Gamma^2_I ({h_I} \otimes {h_I})  = \mathrm{Id} + \delta_{q+1}^{-1} \mathcal{N}_I 
\end{align}
Again, to make sure that such a choice is possible, notice that we can make $\norm{\de_{q+1}^{-1}\mathcal{N}_I}_{C^0(\supp(\th_I)\times \R^3)}$ sufficiently small, provided that $\la_0$ is sufficiently large since $\norm{\de_{q+1}^{-1}S_\ell}_0\lec \la_q^{-2\ga}$, and $\norm{\de_{q+1}^{-1}a_J^2}_0\le \la_q^{-2\gamma}$ when $J \in \mathscr{I}_{\Phi}$ or $\mathscr{I}_{R}$. We can thus apply Lemma \ref{lem:geo1} to $\mathcal{F}^{[k],S}$ and, denoting by $\Gamma_{h_I}$ the corresponding functions, we just need to set 
\[
\Ga_I=\Ga_{h_I} ( \I + \de_{q+1}^{-1} \mathcal{N}_I)\, .
\] 
Therefore, for $I\in \mathscr{I}_S$, we may choose $a_I$ in terms of $\Gamma_I$. Indeed we choose $ a_I = \delta_{q+1}^{1/2} \Gamma_I$. This choice of weight give us the desired ``cancellation property'':
\begin{align*}
(w_0 \otimes w_0)_L
=\sum_{p,k} \th_p^6\left(\frac t{\tau_q}\right) \chi_k^6\left(\frac{\xi_p}{\mu_q}\right)(\delta_{q+1} \mathrm{Id} - S_\ell) = \delta_{q+1} \mathrm{Id} - S_\ell 
\end{align*}

\subsection{Divergence corrector for the velocity}
\label{ss:correction} 

Note that any smooth function on $\T^3$ with mean zero can be represented as a Fourier series. Taking this advantage, we use representations of $\vr_0, w_0$, $w_0 \otimes w_0$, $ \theta_0w_0$, and $\theta_0 w_0\otimes w_0$ based on the Fourier series of $\varphi_I$, $\psi_I$, $\psi_I^2$, $\varphi_I \psi_I$ and $\varphi_I \psi_I^2$.
Indeed, since $\varphi_I, \psi_I$ are smooth functions on $\T^3$ with zero-mean, we have
\begin{equation}\label{rep.psi}
\varphi_I (x) = \sum_{m\in \Z^3\setminus \{0\}} \dot a_{I,m}   e^{ i m\cdot x}, \quad \psi_I (x) = \sum_{m\in \Z^3\setminus \{0\}} \dot{b}_{I,m}  e^{ i m\cdot x},\quad \psi_I^2 (x) = \dot c_{I,0} 
 +\sum_{m\in \Z^3\setminus \{0\}} \dot c_{I,m}   e^{ i m\cdot x},
\end{equation}
\begin{equation}\label{rep.psi1}
\varphi_I  \psi_I (x) = \dot d_{I,0}  + \sum_{m\in \Z^3\setminus \{0\}} \dot{d}_{I,m}  e^{ i m\cdot x},\quad
\varphi_I  \psi_I^2 (x) = \dot e_{I,0} 
+\sum_{m\in \Z^3\setminus \{0\}} \dot e_{I,m}   e^{ i m\cdot x},
\end{equation}
where we used the following notations:
\[
\dot c_{I,0} =\langle \psi_I^2\rangle, \quad
\dot d_{I,0} = \langle \varphi_I   \psi_I\rangle, \quad 
\dot e_{I,0} = \langle \varphi_I   \psi_I^2 \rangle.
\]
Moreover, note that as $\varphi_I$ and $\psi_I$ are in $C^\infty(\T^3)$, we have  
\begin{equation}\label{est.k}
\sum_{m\in \Z^3} |m|^{n_0+2}|\dot{a_{I,m}}|
+\sum_{m\in \Z^3} |m|^{n_0+2}|\dot{b_{I,m}}|
+\sum_{m\in \Z^3} |m|^{n_0+2}|\dot{c_{I,m}}|
+\sum_{m\in \Z^3} |m|^{n_0+2}|\dot{d_{I,m}}|
+\sum_{m\in \Z^3} |m|^{n_0+2}|\dot{e_{I,m}}|
 \lec 1,
 \end{equation} 
 \begin{equation}
\sum_{m\in\Z^3} |\dot{c}_{I,m}|^2\lec 1. 
\end{equation}
for $n_0 = \ceil{\frac{2b(2+\al)}{(b-1)(1-\al)}}$. Also,  it follows from $h_I \cdot \na \psi_I = h_I \cdot \na \psi_I^2=h_I \cdot \na \varphi_I= h_I \cdot \na (\varphi_I \psi_I^2)=0 $ that
\begin{align}\label{coe.van}
\dot b_{I,m} (h_I\cdot m)=\dot c_{I,m} (h_I\cdot m )=\dot a_{I,m} (h_I\cdot m )=\dot e_{I,m} (h_I\cdot m)=0.
\end{align}
Next, in view of the above representation, we have 
\begin{align}
&\theta_0 = \sum_{p} \sum_{m\in \Z^3\setminus \{0\}} \de_{q+1}^{1/2}a_{p,m} e^{i\la_{q+1} m\cdot \xi_I} \label{rep.theta}\\
&w_0 = \sum_{p} \sum_{m\in \Z^3\setminus \{0\}} \de_{q+1}^{1/2}b_{p,m} e^{i\la_{q+1} m\cdot \xi_I} \label{rep.W}\\
&w_0\otimes w_0
=\de_{q+1}\I -S_\ell
+\sum_{p} \sum_{m\in \Z^3\setminus \{0\}}  \de_{q+1} c_{p,m} e^{ i\la_{q+1} m\cdot \xi_I} \label{alg.eq}\\
&\theta_0 w_0 =R_\ell
+\sum_{p} \sum_{m\in \Z^3\setminus \{0\}}  \de_{q+1} d_{p,m} e^{ i\la_{q+1} m\cdot \xi_I}\label{e:rep3}\\
&\theta_0 w_0 \otimes w_0
= \la_q^{-3\gamma} \de_{q+1}^{3/2}  \I -\Phi_\ell 
+\sum_{p} \sum_{m\in \Z^3\setminus \{0\}}  \de_{q+1}^{3/2} e_{p,m} e^{ i\la_{q+1} m\cdot \xi_I}.\label{e:rep4}
\end{align}
where the relevant coefficients are defined as follows:
\begin{equation}\begin{split}\label{coef.def}
&a_{p,m} = \sum_{I: p_I=p} \th_I\chi_I(\xi_I) \de_{q+1}^{-\frac 12}b_I \dot a_{I,m},\\
&b_{p,m} = \sum_{I: p_I=p} \th_I\chi_I(\xi_I) \de_{q+1}^{-\frac 12}a_I \dot b_{I,m} \widetilde{h_I} =: \sum_{I:p_I =p} B_{I,m} \widetilde{h_I},\\
&c_{p,m}
=
\sum_{I: p_I=p} \th_I^2\chi_I^2(\xi_I)
\de_{q+1}^{-1}a_I^2 \dot c_{I,m}
\widetilde{h_I}\otimes \widetilde{h_I},\\
&d_{p,m}
= \sum_{I: p_I=p} \th_I^2 \chi_I^2(\xi_I)
\de_{q+1}^{-1}a_I b_I \dot d_{I,m} \widetilde{h_I}, \\
&e_{p,m}
= \sum_{I: p_I=p} \th_I^3\chi_I^3(\xi_I)
\de_{q+1}^{-3/2}a_I^2 b_I \dot e_{I,m} \widetilde{h_I} \otimes \widetilde{h_I}.
\end{split}\end{equation}
Notice that, if $|p-p'|> 1$, then
\begin{align*}
\supp_{t,x} (a_{p,m}) & \cap \supp_{t,x} (a_{p',m'})
=\supp_{t,x} (b_{p,m})\cap \supp_{t,x} (b_{p',m'})
=\supp_{t,x} (c_{p,m})\cap \supp_{t,x} (c_{p',m'})\\
&=\supp_{t,x} (d_{p,m})\cap \supp_{t,x} (d_{p',m'})
= \supp_{t,x} (e_{p,m})\cap \supp_{t,x} (e_{p',m'})
 = \emptyset
\end{align*}
for any $m, m'\in \Z^3\setminus\{0\}$.  

We now describe how to make an additional correctors $\theta_c$ and $w_c$. The goal of these correctors is to ensure that $\theta=\theta_0+\theta_c$ is mean-zero and that $w=w_0+w_c$ is divergence-free. Indeed, in view of \eqref{coe.van} and the identity $\na \times (\na \xi_I^{\top}U(\xi_I)) = \na \xi_I^{-1}(\na \times U)(\xi_I)$ for any smooth function $U$ (see for example \cite{DaSz2016}), we have 
\[
\frac 1{\la_{q+1}}\na \times \left( \dot{b}_{I,m}\na\xi_I^{\top} \frac{im \times h_I}{|m|^2}  e^{i\la_{q+1} m\cdot {\xi_I}}\right) =\dot{b}_{I,m} \na \xi_I^{-1}h_I e^{i\la_{q+1} m\cdot \xi_I}.
\]
Keeping above identity in mind, we can now rewrite $w_0$ as
\begin{align*}
w_0 
&= \sum_{\substack{p\in \Z\\m\in \Z^3\setminus \{0\}}}  \de_{q+1}^\frac 12 \sum_{I:p_I=p} B_{I,m} \na \xi_I^{-1}h_I e^{i\la_{q+1} m\cdot \xi_I}\\
& \qquad =  \sum_{\substack{p\in \Z\\m\in \Z^3\setminus \{0\}}}  \frac{ \de_{q+1}^\frac 12}{\la_{q+1}}\sum_{I:p_I=p}  
B_{I,m} \na \times\left( \na \xi_I^{\top}\frac{im\times h_I}{|m|^2} e^{i\la_{q+1} m\cdot \xi_I}\right).
\end{align*}
Therefore, we can now define
\begin{align}\label{def.Wc}
w_c = \frac {\de_{q+1}^\frac 12 } {\la_{q+1}} \sum_{\substack{p\in \Z\\m\in \Z^3\setminus\{0\}}} \sum_{p_I =p} \na B_{I,m} \times \left(\na\xi_I^{\top} \frac{im \times h_I}{|m|^2} \right) e^{i\la_{q+1} m\cdot \xi_I}
=: \frac {\de_{q+1}^\frac 12}{\la_{q+1}\mu_q}\sum_{p,m} g_{p,m} e^{i\la_{q+1} m\cdot \xi_I},
\end{align}
where
\begin{align}\label{def.e}
g_{p,m} 
= {\mu_q}\sum_{I:p_I=p} \na(\th_I\chi_I(\xi_I) \de_{q+1}^{-\frac 12} a_I \dot b_{I,m}) \times  \left((\na\xi_I)^{\top} \frac{im \times h_I}{|m|^2} \right).
\end{align}
Thus, the final velocity correction $u_{q+1}-u_q =: w= w_0+w_c$ can be written as
\begin{align*}
w = \na\times \left( \frac {\de_{q+1}^\frac12}{\la_{q+1}}
\sum_{I,m}  
B_{I,m} \na\xi_I^{\top} \frac{im \times h_I}{|m|^2}e^{i\la_{q+1} m\cdot \xi_I} \right),
\end{align*}
and hence it is divergence-free. For future purpose, note that if $|p-p'| {>} 1$, $\supp_{t,x} (g_{p,m})\cap \supp_{t,x} (g_{p',m'})
 = \emptyset$ holds
for any $m, m'\in \Z^3\setminus\{0\}$. Moreover, the correction $w$ has the representation
\begin{align}\label{def.w}
w= \sum_{p\in \Z} \sum_{m\in \Z^3\setminus \{0\}} \de_{q+1}^\frac 12 (b_{p,m} + (\la_{q+1}\mu_q)^{-1} g_{p,m} )e^{i\la_{q+1} m\cdot \xi_I}.
\end{align}
For mean-zero corrector, we set $\theta_c(t):= - \dint_{\T^3} \theta_0(t,x)\,dx$. With this choice, it is clear that $\theta=\theta_0+\theta_c$ has mean-zero.

\subsection{Estimates for density and velocity perturbations}

Here we give details about the estimates on the density and velocity perturbations $\theta$ and $w$ respectively. Recall that $\norm{\cdot}_N = \norm{\cdot}_{C^0([0,T]+\tau_q; C^N(\T^3))}$.

\begin{lem}
\label{p:velocity_correction_estimates} 
The following estimates hold for $\theta$, $w_o$, $w_c$, $w=w_0+w_c$, and $\theta=\theta_0+\theta_c$, for any $N\le N_*$ and $s \le M_*$:
\begin{align}
\tau_q^s \norm{D_{t,\ell}^s w_0}_{N}
&{\lec_M} \la_{q+1}^{N} \de_{q+1}^\frac12, \qquad  \qquad  \qquad 
\tau_q^s\norm{D_{t,\ell}^s w_c}_{N}
{\lec_M} \la_{q+1}^{N} \de_{q+1}^\frac12 (\la_{q+1}\mu_q)^{-1} 		\label{est.Wc} \\
\tau_q^s\norm{D_{t,\ell}^s w}_{N}
&{\lec_M} \la_{q+1}^{N}\de_{q+1}^\frac12, \qquad  \qquad  \qquad 
\tau_q^s\norm{D_{t,\ell}^s \theta_0}_{N}
{\lec_M} \la_{q+1}^{N}\de_{q+1}^\frac12, \label{est.r}\\
\tau_q^s |D_{t,\ell}^s \theta_c(t)|
&{\lec_M} \de_{q+1}^\frac12 (\la_{q+1}\mu_q)^{-1} , \qquad \qquad
\norm{w}_N \lec \la_{q+1}^N \de_{q+1}^\frac12, \qquad \quad  \qquad 
\norm{\theta}_N \lec \la_{q+1}^N \de_{q+1}^\frac12. \label{est.w.indM}
\end{align}
\end{lem}

In fact, the proof of the above lemma follows from the lemma below which provides estimates on the functions $a,b,c,d,e$ and $g$ defined in \eqref{coef.def} and \eqref{def.e}. Note that the estimate for $\theta_c$ follows directly from Lemma~\ref{phase}.

\begin{lem}
\label{lem:est.coe} 	
The coefficients $a_{p,m}, b_{p,m}$, $c_{p,m}$, $d_{p,m}$, $e_{p,m}$ and $g_{p,m}$ defined by \eqref{coef.def} and \eqref{def.e} satisfy, for any $N\geq 0$ and $s \le M_*$, the following:
\begin{align}
\ta_q^s\norm{D_{t,\ell}^s a_{p,m}}_N
&{\lec_{N,M}} \ \mu_q^{-N}\max_I |\dot{a}_{I,m}|, \qquad 
\ta_q^s\norm{D_{t,\ell}^s b_{p,m}}_N
{\lec_{N,M}} \ \mu_q^{-N}\max_I |\dot{b}_{I,m}|, \label{est.b}\\
\ta_q^s\norm{D_{t,\ell}^s c_{p,m}}_N
&{\lec_{N,M}} \ \mu_q^{-N}\max_I |\dot{c}_{I,m}|, \qquad 
\ta_q^s\norm{D_{t,\ell}^s d_{p,m}}_N
{\lec_{N,M}}\  \mu_q^{-N}\max_I |\dot{d}_{I,m}| \label{est.d}\\
\ta_q^s\norm{D_{t,\ell}^s e_{p,m}}_N
&{\lec_{N,M}}\  \mu_q^{-N}\max_I |\dot{e}_{I,m}|, \qquad 
\ta_q^s\norm{D_{t,\ell}^s g_{p,m}}_N
{\lec_{N,M}} \ \mu_q^{-N}\max_I |\dot{b}_{I,m}| \label{est.g}
\end{align}
Moreover, for $N=0,1,2$,
\begin{align}
\label{est.be.indM}
\norm{b_{p,m}}_N + \norm{g_{p,m}}_N
\lec \mu_q^{-N} \max_I |\dot{b}_{I,m}|.
\end{align}
\end{lem}


\begin{proof}
Observe that, for any $s\geq 0$ and $N\geq 0$, 
	\begin{equation}\begin{split}
	\label{est.th.chi}
	\norm{D_{t,\ell}^s\th_I}_{C^0(\R)} 
	&= \norm{\pa_t^s\th_I}_{C^0(\R )} \lec_s \tau_q^{-s}, \\ 
	\norm{\chi_I(\xi_I)}_{C^0(\cal{I}_p; C^N(\R^3))} 
	&\lec_N \mu_q^{-N}, \  D_{t,\ell}^s [\chi_I(\xi_I)] = 0,
	\end{split}\end{equation}
	where $\cal{I}_p = [t_p - \frac 12\tau_q, t_p + \frac 32 \tau_q]$. In fact, the estimate of $\chi_I(\xi_I)$ follows from \eqref{est.flow2}, Lemma \ref{lem:est.com}, and $\ell^{-1}\leq \mu_q^{-1}$. Note also that the implicit constants are independent of $I$.

Recall that for $I\in \mathscr{I}_R$, we have $\de_{q+1}^{-1} \la_q^{2\gamma} a_I b_I = \Ga_I = \Ga_{h_I} (\lambda_q^{2 \gamma} \de_{q+1}^{-1} \na \xi_I R_\ell)$ for a finite collection of smooth functions $h_I$ chosen through Lemma \ref{lem:geo2}. First we obtain, using \eqref{est.flow2} and \eqref{est.mph}, the estimate for $\lambda_q^{2 \gamma}\na \xi_I R_\ell$, for $N=0,1,2$:
	\begin{align*}
	&\norm{\lambda_q^{2 \gamma} \na \xi_I R_\ell}_{C^0(\cal{I}_m; C^N(\R^3))} 
	\le \sum_{N_1+N_2=N} \lambda_q^{2 \gamma} \norm{\na \xi_I}_{C^0(\cal{I}_m; C^{N_1}_x)} \norm{R_\ell}_{N_2}  \lec \de_{q+1} \mu_q^{-N} 
	\end{align*}
	Similarly, using \eqref{est.flow2} and \eqref{est.mph}, and the fact that $\lambda_q \delta_q^{1/2} \le \tau_q^{-1}$, we have 
	\begin{align*}
	\norm{D_{t,\ell}^s \lambda_q^{2 \gamma} \na \xi_I R_\ell}_N \lec_{N,M} \de_{q+1}\ta_q^{-s} \mu_q^{-N}.
	\end{align*}
Next, we make use of Faa di Bruno formula in Appendix~\ref{Faadi}. In particular, thanks to \eqref{Faadi_01}, for any smooth functions $\Ga=\Ga(x)$ and $g=g(t,x)$, we have
\begin{align*}
\norm{D_{t,\ell}^s \Ga(g)}_{C^N_x} 
\lec \sum_{N_1+N_2=N} \norm{D_{t,\ell}^s g}_{C^{N_1}_x} \norm{(\na^{s-1}\Ga)(g)}_{C^{N_2}_x}
+ \norm{D_{t,\ell} g\otimes \cdots \otimes D_{t,\ell} g}_{C^{N_1}_x} \norm{(\na^s\Ga)(g)}_{C^{N_2}_x},
\end{align*}
Therefore, when $h_I \in \cH_{I, R}$, for any $s\le M_*$ and $N\geq 0$, 
	\begin{align}\label{est.Ga.R}
	\norm{D_{t,\ell}^s \de_{q+1}^{-1} \la_q^{2 \gamma} a_I b_I}_N = \norm{D_{t,\ell}^s(\Ga_{h_I} (\la_q^{2 \gamma}\de_{q+1}^{-1}\na \xi_I R_\ell))}_N
	\lec_{N,M} \ta_q^{-s} \mu_q^{-N}.
	\end{align}
	{In particular, for $N=0,1,2$, the implicit constant can be chosen to be independent of $M$ and $N$;
		\begin{align*}
		\norm{\de_{q+1}^{-1} \la_q^{2 \gamma} a_I b_I}_N 
		\lec  \mu_q^{-N}.
		\end{align*}}
On the other hand, when $I\in \mathscr{I}_{\Phi}$, recall that $\la_q^{3\gamma}\de_{q+1}^{-\frac 32} a^2_I b_I = \Ga^2_I = \Ga^2_{h_I} (\I+\la_q^{3\gamma}\de_{q+1}^{-3/2}\mathcal{M}_I)$ for a finite collection of smooth functions $h_I$ chosen through Lemma \ref{lem:geo1}. First we obtain the estimate for $\mathcal{M}_I$, 
	\begin{align*}
	&\norm{\mathcal{M}_I}_{C^0(\cal{I}_m; C^N(\R^3))} 
	\lec_N\ \la_q^{-3\gamma} \de_{q+1}^{3/2}  \norm{(\na \xi_I)(\na \xi_I)^{\top} - \I}_{C^0(\cal{I}_m; C^N(\R^3))}  \\
	&\quad + \sum_{N_1+N_2+N_3=N}\norm{\na \xi_I}_{C^0(\cal{I}_m; C^{N_1}_x)} \norm{\Phi_\ell}_{N_2} \norm{\na \xi_I}_{C^0(\cal{I}_m; C^{N_3}_x)}
	\lec \la_q^{-3\gamma} \de_{q+1}^{3/2} \mu_q^{-N}. 
	\end{align*}
	Similarly, we have 
	\begin{align*}
	\norm{D_{t,\ell}^s \mathcal{M}_I}_N \lec_{N,M} \la_q^{-3\gamma} \de_{q+1}^{3/2}\ta_q^{-s} \mu_q^{-N},
	\end{align*}
	{but $\norm{ \mathcal{M}_I}_N \lec \la_q^{-3\gamma} \de_{q+1}^{3/2} \mu_q^{-N}$ for $N=0,1,2$.}
	Then, an application of the Faa di Bruno formula \eqref{Faadi_01} implies that when $h_I \in \cH_{I, \Phi}$, for $s\le M_*$ and $N\geq 0$, 
	\begin{align}\label{est.Ga.phi}
	\norm{D_{t,\ell}^s \la_q^{3\gamma}\de_{q+1}^{-3/2} a^2_I b_I}_N = \norm{D_{t,\ell}^s(\Ga^2_{h_I} (\I+ \la_q^{3\gamma}\de_{q+1}^{-3/2}\mathcal{M}_I))}_N
	\lec_{N,M} \ta_q^{-s} \mu_q^{-N}.
	\end{align}
	{In particular, for $N=0,1,2$, the implicit constant can be chosen to be independent of $M$ and $N$;
		\begin{align*}		\norm{\la_q^{3\gamma}\de_{q+1}^{-3/2} a^2_I b_I}_N 
		\lec  \mu_q^{-N}.
		\end{align*}}
Finally, when $I\in \mathscr{I}_S$, recall that $\de_{q+1}^{-1/2} a_I = \Ga_I = \Ga_{h_I} (\I+\de_{q+1}^{-1}\mathcal{N}_I)$ for a finite collection of smooth functions $h_I$ chosen through Lemma \ref{lem:geo1}. First we obtain the estimate for $\mathcal{N}_I$, 
	\begin{align*}
	&\norm{\mathcal{N}_I}_{C^0(\cal{I}_m; C^N(\R^3))} \\
	&\lec_N\ \de_{q+1}  \norm{(\na \xi_I)(\na \xi_I)^{\top} - \I}_{C^0(\cal{I}_m; C^N(\R^3))}  \\
	&\ + \sum_{N_1+N_2+N_3=N}\norm{\na \xi_I}_{C^0(\cal{I}_m; C^{N_1}_x)} \norm{S_\ell}_{N_2} \norm{\na \xi_I}_{C^0(\cal{I}_m; C^{N_3}_x)}\\
	&\ + \sum_{N_1+N_2+N_3+N_4=N} \sum_{{J}: f\in \cH_{{J}, R}} \norm{\na \xi_I}_{C^0(\cal{I}_m; C^{N_1}_x)} \norm{\chi_{J}^2(\xi_{J})}_{C^0(\cal{I}_m; C^{N_2}_x)} \norm{a_{J}^2}_{C^0(\cal{I}_m; C^{N_3}_x)}  \norm{\na \xi_I}_{C^0(\cal{I}_m; C^{N_4}_x)}\\
	&\ + \sum_{N_1+N_2+N_3+N_4=N} \sum_{{J}: f\in \cH_{{J}, \Phi}} \norm{\na \xi_I}_{C^0(\cal{I}_m; C^{N_1}_x)} \norm{\chi_{J}^2(\xi_{J})}_{C^0(\cal{I}_m; C^{N_2}_x)} \norm{a_{J}^2}_{C^0(\cal{I}_m; C^{N_3}_x)}  \norm{\na \xi_I}_{C^0(\cal{I}_m; C^{N_4}_x)}\\
	&\lec \de_{q+1} \mu_q^{-N}.
	\end{align*}
	Similarly, we also have 
	\begin{align*}
	\norm{D_{t,\ell}^s \mathcal{N}_I}_N \lec_{N,M} \de_{q+1}\ta_q^{-s} \mu_q^{-N},
	\end{align*}
	{but $\norm{ \mathcal{N}_I}_N \lec \de_{q+1} \mu_q^{-N}$ for $N=0,1,2$.}
	Then, an application of the Faa di Bruno formula \eqref{Faadi_01} implies that when $h_I \in \cH_{I, R}$, for $s\le M_*$ and $N\geq 0$, 
	\begin{align}\label{est.Ga.S}
	\norm{D_{t,\ell}^s \de_{q+1}^{-1/2} a_I}_N = \norm{D_{t,\ell}^s(\Ga_{h_I} (\I+\de_{q+1}^{-1}\mathcal{N}_I))}_N
	\lec_{N,M} \ta_q^{-s} \mu_q^{-N}.
	\end{align}
	{In particular, for $N=0,1,2$, the implicit constant can be chosen to be independent of $M$ and $N$;
		\begin{align*}
		\norm{\de_{q+1}^{-\frac 12} a_I}_N 
		\lec  \mu_q^{-N}.
		\end{align*}}
Finally, recalling the definition of $a_{m,k}$, $b_{m,k}$, $c_{m,k}$, $d_{m,k}$, $e_{m,k}$ and $g_{m,k}$, we see that the estimates \eqref{est.b}-\eqref{est.be.indM} follows from previous estimates \eqref{est.Ga.R}-\eqref{est.Ga.S}.
\end{proof}

%
%

\section{Definition of the new errors}
\label{errors}
Let us start with the following relaxed density and momentum equations at the $q$-th level and derive the equations for the $(q+1)$-th level. In what follows, we start with
\begin{align*} 
\partial_t {\vr}_q + \div ({\vr_q} {u_q})  &= - \div R_q,\\ 
\partial_t ({\vr_q} {u_q} ) +  \div ({\vr_q} u_q \otimes u_q)+  \Grad p_q  &=  -(\partial_t + {u_q} \cdot \nabla) R_q - \div (R_q \otimes {u_q}) + \div \Phi_q  + \div (\vr_q S_q ),\\
\div \,{u_q} & =0.
\end{align*}
Next, we write ${\vr}_{q+1} = {\vr}_q + \theta$, and $u_{q+1}= u_q + w$, and $p_{q+1} = p_q - \de_{q+1} \vr_q$, where $\theta$ and $w$ are appropriate perturbations for the density and the velocity respectively. Recall that $\div \, w =0$.

Moreover, mollified density equation reads
\begin{align*}
\partial_t {\vr}_\ell + \div ({\vr_\ell} {u_\ell})  &= - \div P_{\le \ell^{-1}} R_q + \, Q(\vr_q,u_q),
\end{align*}
where $Q(\vr_q,u_q) = \div ({\vr_\ell} {u_\ell} - (\vr_q u_q)_\ell)$ with $(\vr_q u_q)_\ell = P_{\le \ell^{-1}} (\vr_q u_q)$. Furthermore,
\begin{align*}
Q(\vr_q,u_q) = u_\ell \cdot \na \vr_\ell - (u_q \cdot \na \vr_q)_\ell
= (u_\ell - u_q) \cdot \na \vr_\ell + [u_q \cdot \na, P_{\le \ell^{-1}} ] \vr_q.
\end{align*}

We also define the notation $a \otimes_s b$ for two vectors $a, b$ to denote the symmetric $2$-tensor
\[ a \otimes_s b := a \otimes b + b \otimes a\,. \]

\subsection{New Reynolds stress}
We calculate 
\begin{align*}
&\partial_t {\vr}_{q+1} + \div ({\vr_{q+1}} {u_{q+1}})  =\underbrace{\partial_t {\vr}_q + \div ({\vr_q} {u_q})}_{= -\div R_q} + \partial_t \theta + \div(\vr_q w + u_q \theta) + \div(\theta w)\\
&= \underbrace{\div(\theta w - R_q)}_{\text{quadratic error}}+ \underbrace{\partial_t \theta + \div(\vr_q w + u_q \theta)}_{\text{linear errors}} 
=\underbrace{\na \cdot(\theta w - R_\ell)}_{=: \na \cdot \as R_O} + \underbrace{(\partial_t \theta  + (u_\ell \cdot \na) \theta)}_{=: \na \cdot \as R_T } + \underbrace{w \cdot \na \vr_\ell}_{=: \na \cdot \as R_N} \\[2pt]
& \qquad \qquad + \underbrace{\na \cdot ((u_q -u_\ell) \theta + (\vr_q -\vr_\ell) w - (R_q -R_\ell))}_{=:\na \cdot \as R_M} := - \div R_{q+1}^{\text{pre}}
\end{align*}
Note that we have used the fact that $\theta$ has zero mean in order to say that it can be written as the divergence of a vector.

\subsection{New quadratic stress and new current}
Keeping in mind the definitions of $\vr_{q+1}$, $u_{q+1}$, and $p_{q+1}$, we begin with the momentum equations
\begin{align*}
&\partial_t ({\vr_{q+1}} {u_{q+1}} ) +  \div ({\vr_{q+1}} u_{q+1} \otimes u_{q+1})+  \Grad p_{q+1} \\[2pt]
&= \underbrace{\partial_t ({\vr_q} {u_q} ) +  \div ({\vr_q} u_q \otimes u_q)+  \Grad p_q  }_{=-(\partial_t + {u_q} \cdot \nabla) R_q - \div (R_q \otimes {u_q}) + \div \Phi_q + \div (\vr_q S_q)} 
+ \partial_t (\vr_q w + u_q \theta + \theta w) - \de_{q+1} \na \vr_q\\[2pt]
&\qquad\qquad+ \div(\vr_q u_q \otimes w + \vr_q w \otimes u_q + \vr_q w \otimes w + u_q \theta \otimes u_q + u_q \theta \otimes w + \theta w \otimes u_q + \theta w \otimes w)\\[2pt]
& = (\partial_t + {u_q} \cdot \nabla) (\theta w -R_q )  +\div ((\theta w -R_q) \otimes {u_q}) + \div(\varrho_q(w\otimes w - \de_{q+1}\I + S_q))
\\[2pt]
&\qquad +\partial_t (w\varrho_{\ell}) +\pa_t( w(\varrho_q-\varrho_\ell)) + \div(u_\ell \otimes w \varrho_\ell) + \div(u_\ell \otimes w (\varrho_q-\varrho_\ell)) + \div((u_q-u_\ell) \otimes w \varrho_q)\\
&\qquad + \partial_t( \theta u_\ell) + \pa_t(\theta(u_q-u_\ell)) + \div(u_\ell \otimes u_\ell \theta) + \div(u_\ell \otimes \theta (u_q-u_\ell)) + \div((u_q-u_\ell) \otimes \theta u_q)\\
&\qquad + \div (\varrho_\ell w\otimes u_\ell) + \div(\varrho_\ell w \otimes (u_q-u_\ell)) + \div((\varrho_q-\varrho_\ell) w \otimes u_q)\\
&\qquad + \div(\theta w \otimes w + \Phi_q) \\
& = (\partial_t + {u_q} \cdot \nabla) (\theta w -R_q )  +\div ((\theta w -R_q) \otimes {u_q}) + \div(\varrho_q(w\otimes w - \de_{q+1}\I + S_q))
\\[2pt]
&\qquad +w(\partial_t \varrho_{\ell} + \na\varrho_\ell \cdot u_\ell) + u_\ell(\partial_t\theta + u_\ell\cdot\na\theta) +u_\ell(w\cdot\na\varrho_\ell)\\
&\qquad +\varrho_\ell(\pa_tw + u_\ell\cdot\na w)+\theta(\pa_t u_\ell+u_\ell\cdot\na u_\ell) +(\varrho_\ell w)\cdot\na u_\ell\\
&\qquad + \div(\theta w \otimes w + \Phi_q) \\
&\qquad + \pa_t( w(\varrho_q-\varrho_\ell))+ \div(u_\ell \otimes w (\varrho_q-\varrho_\ell)) + \div((u_q-u_\ell) \otimes w \varrho_q) \\
&\qquad + \pa_t(\theta(u_q-u_\ell)) + \div(u_\ell \otimes \theta (u_q-u_\ell)) + \div((u_q-u_\ell) \otimes \theta u_q)\\
&\qquad + \div(\varrho_\ell w \otimes (u_q-u_\ell)) + \div((\varrho_q-\varrho_\ell) w \otimes u_q)\\
& = -(\partial_t + {u_{q+1}} \cdot \nabla)R_{q+1}^{pre} - \div(R_{q+1}^{pre}\otimes u_{q+1}) +\div(\varrho_{q+1}S_{q+1}^{pre}) + \div \Phi_{q+1}
\end{align*}
Here $S_{q+1}^{pre}$ and $\Phi_{q+1}$ are defined as
\begin{align*}
    \div (\vr_{q+1} S_{q+1}^{\text{pre}}) &:= \div(\vr_q (w \otimes w + S_q - \de_{q+1} \text{Id}))
+ \theta (\partial_t u_\ell + (u_\ell \cdot \nabla) u_\ell) + \vr_\ell (\partial_t w + (u_\ell \cdot \nabla) w) + (\vr_\ell w \cdot \nabla) u_\ell\\
 &\qquad+   \div(\vr_\ell w\otimes (u_q - u_\ell) ) + \div( (\vr_q - \vr_\ell) w \otimes u_q ) +  \div((u_q - u_\ell) \otimes (w\varrho_q + \theta u_q )) \\
 &\qquad - \div( (u_q-u_\ell) \otimes \theta u_\ell) - \div( (\vr_q-\vr_\ell) w \otimes u_\ell)\\
 &= \div(\vr_q (w \otimes w + S_q - \de_{q+1} \text{Id}))
+ \theta (\partial_t u_\ell + (u_\ell \cdot \nabla) u_\ell) + \vr_\ell (\partial_t w + (u_\ell \cdot \nabla) w) + (\vr_\ell w \cdot \nabla) u_\ell\\
&\qquad + \div(\vr_\ell w\otimes (u_q - u_\ell) ) + \div( (\vr_q - \vr_\ell) w \otimes (u_q - u_\ell) ) +  \div((u_q - u_\ell) \otimes \theta (u_q - u_\ell ))\\
&\qquad + \div ((u_q-u_\ell)\otimes w\varrho_q)\\
&= \div(\vr_q (w \otimes w + S_q - \de_{q+1} \text{Id}))
+ \theta (\partial_t u_\ell + (u_\ell \cdot \nabla) u_\ell) + \vr_\ell (\partial_t w + (u_\ell \cdot \nabla) w) + (\vr_\ell w \cdot \nabla) u_\ell\\
&\qquad + \div(\vr_q w \otimes_s (u_q - u_\ell) )  +  \div(\theta (u_q - u_\ell) \otimes (u_q - u_\ell ))
\end{align*}
and so
\begin{align*}
    \div \Phi_{q+1} &:= (\partial_t + {u_q} \cdot \nabla) (\theta w -R_q + R_{q+1}^{pre}) +w\cdot\na R_{q+1}^{pre}  +\div ((\theta w -R_q +R_{q+1}^{pre}) \otimes {u_q}) + \div(R_{q+1}^{pre}\otimes w)
\\
&\qquad +w(\partial_t \varrho_{\ell} + \na\varrho_\ell \cdot u_\ell) + u_\ell(\partial_t\theta + u_\ell\cdot\na\theta) +u_\ell(w\cdot\na\varrho_\ell) + \div(\theta w \otimes w + \Phi_q) \\
&\qquad + (\pa_t + u_\ell\cdot\na)( w(\varrho_q-\varrho_\ell)) + (\pa_t + u_\ell\cdot\na)(\theta(u_q-u_\ell)) \\
&\qquad  + \div( (u_q-u_\ell) \otimes \theta u_\ell) + \div( (\vr_q-\vr_\ell) w \otimes u_\ell)\\
&=(\partial_t + {u_\ell} \cdot \nabla) (\theta w -R_q + R_{q+1}^{pre}) +w\cdot\na R_{q+1}^{pre}  +\div ((\theta w -R_q +R_{q+1}^{pre}) \otimes {u_\ell}) + \div(R_{q+1}^{pre}\otimes w)
\\
&\qquad +w(-\div P_{\le \ell^{-1}} R_q + Q(\vr_q,u_q)) + u_\ell(\div \as R_T) +u_\ell(\div \as R_N) + \div(\theta w \otimes w + \Phi_q) \\
&\qquad + (\pa_t + u_\ell\cdot\na)( w(\varrho_q-\varrho_\ell)) + (\pa_t + u_\ell\cdot\na)(\theta(u_q-u_\ell)) \\
&\qquad + \div( (u_q-u_\ell) \otimes \theta u_\ell) + \div( (\vr_q-\vr_\ell) w \otimes u_\ell) \\
&\qquad + \div((u_q-u_\ell)\otimes_s (\theta w -R_q +R_{q+1}^{pre}))\\
&=(\partial_t + {u_\ell} \cdot \nabla) (\theta w -R_q + R_{q+1}^{pre}) +w\cdot\na R_{q+1}^{pre} + \div(R_{q+1}^{pre}\otimes w)
\\
&\qquad +u_\ell \div (\theta w -R_q +R_{q+1}^{pre}+\as R_T+\as R_N) + (\theta w -R_q +R_{q+1}^{pre}) \cdot\na {u_\ell}\\
&\qquad +w(-\div P_{\le \ell^{-1}} R_q + Q(\vr_q,u_q)) + \div(\theta w \otimes w + \Phi_q) \\
&\qquad + (\pa_t + u_\ell\cdot\na)( w(\varrho_q-\varrho_\ell)) + (\pa_t + u_\ell\cdot\na)(\theta(u_q-u_\ell)) \\
&\qquad + \div( (u_q-u_\ell) \otimes \theta u_\ell) + \div( (\vr_q-\vr_\ell) w \otimes u_\ell) \\
&\qquad + \div((u_q-u_\ell)\otimes_s (\theta w -R_q +R_{q+1}^{pre})) \\
&=(\partial_t + {u_\ell} \cdot \nabla) (\theta w -R_q + R_{q+1}^{pre}) +w\cdot\na R_{q+1}^{pre} + \div(R_{q+1}^{pre}\otimes w)
\\
&\qquad -u_\ell \div ((u_q-u_\ell)\theta+(\varrho_q-\varrho_\ell)w) + (\theta w -R_q +R_{q+1}^{pre}) \cdot\na {u_\ell}\\
&\qquad +w(-\div P_{\le \ell^{-1}} R_q + Q(\vr_q,u_q)) + \div(\theta w \otimes w + \Phi_q) \\
&\qquad + (\pa_t + u_\ell\cdot\na)( w(\varrho_q-\varrho_\ell)) + (\pa_t + u_\ell\cdot\na)(\theta(u_q-u_\ell)) \\
&\qquad + \div( (u_q-u_\ell) \otimes \theta u_\ell) + \div( (\vr_q-\vr_\ell) w \otimes u_\ell) \\
&\qquad + \div((u_q-u_\ell)\otimes_s (\theta w -R_q +R_{q+1}^{pre}))\,.
\end{align*}
So, we group the terms as
\begin{multline}
\div (\vr_{q+1} S_{q+1}^{\text{pre}}) = \underbrace{\div(\vr_q (w \otimes w + S_\ell - \de_{q+1} \text{Id}))}_{=: \na \cdot \as S_{O}} \\
+ \underbrace{\theta (\partial_t u_\ell + (u_\ell \cdot \nabla) u_\ell)}_{=: \na \cdot \as S_{T1}+ \partial_t \Upsilon_4} + \underbrace{\vr_\ell (\partial_t w + (u_\ell \cdot \nabla) w)}_{=: \na \cdot \as S_{T2}+ \partial_t \Upsilon_5} + \underbrace{(\vr_\ell w \cdot \nabla) u_\ell}_{=: \na \cdot \as S_{N} + \partial_t \Upsilon_6}\\
 + \underbrace{ \div(\vr_q w \otimes_s (u_q - u_\ell) )  +  \div(\theta (u_q - u_\ell) \otimes (u_q - u_\ell )) + \div (\vr_q (S_q - S_\ell ))}_{=: \na \cdot \as S_{M}}
\end{multline}
and
\begin{multline}
\div \Phi_{q+1} = \underbrace{(\partial_t + {u_\ell} \cdot \nabla) (\theta w -R_q + R_{q+1}^{\text{pre}} + (u_q -u_\ell) \theta + (\vr_q -\vr_\ell) w)}_{=: \na \cdot \as \varphi_T+ \partial_t \Upsilon_1}+ \underbrace{\div(\theta w \otimes w + \Phi_\ell)}_{=: \na \cdot \as \varphi_O}  \\
 \quad + \underbrace{\div (R_{q+1}^{\text{pre}} \otimes_s w)}_{=: \na \cdot \as \varphi_R} -\underbrace{w \div P_{\le \ell^{-1}} R_q + w  Q(\vr_q,u_q)}_{=: \na \cdot \as \varphi_{H1} + \partial_t \Upsilon_2}  \\
+  \underbrace{(R_{q+1}^{\text{pre}} -R_q + \theta w + (u_q -u_\ell) \theta + (\vr_q -\vr_\ell) w) \cdot \nabla  u_\ell}_{=: \na \cdot \as \varphi_{H2}+ \partial_t \Upsilon_3} \\
+ \underbrace{\div((R_{q+1}^{\text{pre}} - R_q +\theta w) \otimes_s (u_q - u_\ell)) + \div(\Phi_q - \Phi_\ell)}_{=: \na \cdot \as \varphi_{M}} \label{phi_01}
\end{multline}

Note that, in order to apply inverse divergence operator, we need to make sure that the agrument has mean zero. For that purpose, we set
\begin{align}
\Upsilon_1 (t) &:= \int_0^t \langle\left(\partial_t + {u_\ell} \cdot \nabla\right) (\theta w -R_q + R_{q+1}^{\text{pre}} + (u_q -u_\ell) \theta + (\vr_q -\vr_\ell) w) \rangle (s)\, ds\label{e:define_up_2}\\
\Upsilon_2 (t) &:= \int_0^t \langle (-\div P_{\le \ell^{-1}}R_q + Q(\vr_q,u_q)) w \rangle (s)\, ds \label{e:define_up_3} \\
\Upsilon_3 (t) &:= \int_0^t \langle (R_{q+1}^{\text{pre}} -R_q + \theta w + (u_q -u_\ell) \theta + (\vr_q -\vr_\ell) w) \cdot \nabla  u_\ell \rangle (s)\, ds
\label{e:define_up_4} \\
\Upsilon_4 (t) &:= \int_0^t \langle \theta (\partial_t u_\ell + (u_\ell \cdot \nabla) u_\ell)\rangle (s)\, ds \label{e:define_up_6} \\
\Upsilon_5 (t) &:= \int_0^t \langle \vr_\ell (\partial_t w + (u_\ell \cdot \nabla) w)\rangle (s)\, ds \label{e:define_up_7}  \\
\Upsilon_6 (t) &:= \int_0^t \langle (\vr_\ell w \cdot \nabla) u_\ell\rangle (s)\, ds\,. \label{e:define_up_8}
\end{align}
Finally, we can write our new errors $R_{q+1}$ and $S_{q+1}$ as
\begin{align}
    R_{q+1} := R_{q+1}^{\text{pre}} + \sum_{i=1}^3 \Upsilon_i + \sum_{i=4}^6 \Upsilon_i := R_{q+1}^{\text{pre}} + \Upsilon^1 + \Upsilon^2\,,\qquad
    S_{q+1} := S_{q+1}^{\text{pre}} - \frac1{\varrho_{q+1}} \sum_{i=1}^6 \Upsilon_i \otimes u_{q+1}
\end{align}

\section{Estimates of new errors} 
\label{est}
In this section, we estimate the new error terms introduced earlier. Our main technical tools are a microlocal lemma from Appendix~\ref{sec.ml} and Proposition~\ref{prop:intermittent:inverse:div}, which provides estimates for the inverse divergence of various quantities. 

\subsection{Estimates on the Reynolds stress}

With the help of the Proposition~\ref{prop:intermittent:inverse:div}, we now obtain the relevant estimates for the new Reynolds stress $R_{q+1}$ and its new advective derivative $ \as D_t R_{q+1} = \partial_t R_{q+1} + u_{q+1}\cdot \nabla R_{q+1}$, summarized in the following proposition. Indeed, we first derive estimates on $ R_{q+1}^{\text{pre}}, \as D_t R_{q+1}^{\text{pre}} $ and, in view of last estimate in \eqref{est.asS} and \eqref{est.asph}, the estimates on $ R_{q+1}, \as D_t R_{q+1} $ follows. As before, for the remaining sections, we set $\norm{\cdot}_N = \norm{\cdot}_{C^0([0,T]+\tau_q; C^N(\T^3))}$.

\begin{prop}\label{p:Reynolds}
{There exists  $\bar{b}(\al)>1$ with the following property. For any $1<b<\bar{b}(\al)$ we can find 
$\Lambda_0 = \Lambda_0 (\al,b,M)$} such that the following estimates hold for every  $\lambda_0 \geq \Lambda_0$: 
\begin{equation}\label{est.asR}\begin{split}
\norm{R_{q+1}^{\text{pre}}}_N &\leq C_M\la_{q+1}^N \cdot
 {\la_q^\frac12}\la_{q+1}^{-\frac12} \de_q^\frac14\de_{q+1}^\frac34 
 \leq \la_{q+1}^{N-2\ga} \de_{q+2}, 
 \ \ \qquad\qquad \forall N=0,1,2 \\
\norm{\as D_t R_{q+1}^{\text{pre}} }_{N-1} &\leq C_M  \la_{q+1}^N \de_{q+1}^\frac12\cdot
 {\la_q^\frac12}\la_{q+1}^{-\frac12} \de_q^\frac14\de_{q+1}^\frac34  \leq \la_{q+1}^{N-2\ga}\de_{q+1}^\frac12 \de_{q+2}, \quad \forall N=1,2\, .
\end{split}\end{equation}
{where $C_M$ depends only upon the $M>1$.}
\end{prop}

We will estimate the terms $\as R_T$, $\as R_N$, $\as R_{O}$ and $\as R_M$ separately. For the errors $\as R_M$, and a part of $\as R_{O}$, we use a direct estimate, while the other errors, including the inverse divergence operator, are estimated by Proposition~\ref{prop:intermittent:inverse:div}. In the following subsections, we fix $n_0 = \ceil{\frac{2b(2+\al)}{(b-1)(1-\al)}}$ so that $\la_{q+1}^2({\la_{q+1}\mu_q})^{-(n_0+1)}\lec_M \de_{q+1}^\frac12$ for any $q$. Also, we remark that 
\begin{align}\label{rel.par}
\frac {1}{\la_{q+1}\tau_q} +\frac{\de_{q+1}^\frac12}{\la_{q+1}\mu_q} \lec_M  {\la_q^\frac12}{\la_{q+1}^{-\frac12}} \de_q^\frac14\de_{q+1}^\frac14.
\end{align}
For the convenience, we restrict the range of $N$ as in \eqref{est.asR} in this section, without mentioning it further.

Another important remark which will be used in this section is that
\begin{equation}\label{e:average-free}
\mathcal{R} (g (t, \cdot)+ h(t)) = \mathcal{R} (g (t, \cdot))\,,\qquad \mathcal{R}^{\mathrm{loc}} (g (t, \cdot)+ h(t)) = \mathcal{R}^{loc} (g (t, \cdot))
\end{equation}
for every smooth periodic time-dependent vector field $g$ and for every $h$ which depends only on time. 

\subsubsection{\bf Transport stress error}
Recall that, thanks to \eqref{e:average-free}
\begin{align}
\label{upsilon1}
\as R_T = \idvl{D_{t,\ell} \theta_0}
\end{align}
Since $D_{t,\ell} \xi_I =0$, thanks to \eqref{eqnew11}, we have $D_{t,\ell} \theta_0 = \sum_{I} (D_{t,\ell} G_I) \zeta_I \circ \xi_I$, where
$\zeta_I = \varphi_I (\la_{q+1} \cdot)$ and
\begin{align*}
H:= D_{t,\ell} G_I= D_{t,\ell}( b_I \theta_I(t) ) \chi_{I} (\xi_{I}(t,x)).
\end{align*}
Notice that, in view of the estimates of $b_I$ given in Lemma~\ref{lem:est.coe}, low frequency assumption \eqref{eq:inverse:div:DN:G} is satisfied for the function $H$ with $\MM{M,M_{t},\nu,\nu'} = \tau_q^{-M}$ and $\const_{G_I,p} = \tau_q^{-1} \delta_{q+1}^{1/2} $. Therefore, choosing various parameters involved in Proposition~\ref{prop:intermittent:inverse:div}, in accordance with Remark~\ref{rmk01}, and making use of the estimates  \eqref{eq:inverse:div:stress:1} and \eqref{eq:inverse:div:error:stress:bound}, we obtain the following bounds
\begin{align}\label{est.RT}
\norm{\as R_T}_N
\lec \la_{q+1}^N\frac {\de_{q+1}^\frac 12}{\la_{q+1} \tau_q} , \quad
\norm{\as D_t \as R_T}_{N-1}  \lec \la_{q+1}^N\de_{q+1}^\frac12 \frac {\de_{q+1}^\frac 12}{\la_{q+1} \tau_q}.
\end{align}
 
\subsubsection{\bf Nash stress error}
Set $\as R_N = \idvl{w \cdot \na \vr_\ell}$ and observe that $w \cdot \na \vr_\ell = w_0 \cdot \na \vr_\ell + w_c \cdot \na \vr_\ell$, where the second terms arises due to the addition of the velocity corrector. Since the bounds for the corrector are stronger than that of the principal part of the velocity perturbation, we only give estimates for the first term and briefly mention the set-up for the other term. Note that we can write
$w_0 \cdot \na \vr_\ell = \sum_{I} G_I \zeta_I \circ \xi_I$, where $\zeta_I = \psi_I (\la_{q+1} \cdot)$ and
\begin{align*}
G_I= a_I \theta_I(t) \chi_{I} (\xi_{I}(t,x)) \widetilde{h_I} \cdot \na \vr_\ell.
\end{align*}
Therefore, choosing various parameters involved in Proposition~\ref{prop:intermittent:inverse:div}, in accordance with Remark~\ref{rmk01}, and making use of the estimates \eqref{est.vp}, \eqref{est.v.dif}, \eqref{est.dtu.new}, \eqref{est.r}, \eqref{eq:inverse:div:stress:1} and \eqref{eq:inverse:div:error:stress:bound}, we obtain using $\const_{G_I,p} = \la_q \de_q^{1/2} \delta_{q+1}^{1/2} $, the following bounds
\begin{align}\label{est.RN}
\norm{\as R_N}_N 
\lec \la_{q+1}^N \frac {\de_{q+1}^\frac 12}{\la_{q+1}\tau_q} , \quad 
\norm{\as D_t\as R_N}_{N-1}  \lec \la_{q+1}^N\de_{q+1}^\frac12 \frac {\de_{q+1}^\frac 12}{\la_{q+1}\tau_q}.
\end{align}
For the second term, we make use of \cite[Proposition 4.5]{Giri03}, to rewrite the second term in the form of $G_I \zeta_I \circ \xi_I$ and obtain the result using the estimates \eqref{eq:inverse:div:stress:1} and \eqref{eq:inverse:div:error:stress:bound}.

\subsubsection{\bf Oscillation stress error}
\label{OS}
Recall that $\as R_O = \as R_{O1} + \as R_{O2}$ where
\begin{align*}
\as R_{O1} 
= \idvl{\na\cdot ( \theta_0 w_0 -  R_\ell) }, \quad 
\na\cdot \as R_{O2} 
=\na\cdot (\theta_0 w_c) + \na\cdot (\theta_c (w_0 + w_c)) := \na\cdot \as R^1_{O2} + \as R^2_{O2}.
\end{align*}
First note that since $\theta_c$ is constant in space and $\na \cdot w =0$, so we conclude $\as R^2_{O2}=0$.
Next, we write $\na\cdot ( \theta_0 w_0 - R_\ell) = \sum_{I} G_I \zeta_I \circ \xi_I$, where $\zeta_I = \varphi_I\psi_I (\la_{q+1} \cdot)$ or $(\varphi_I\psi_I -1)(\la_{q+1} \cdot)$ and
\begin{align*}
G_I = \na\cdot [a_I b_I \theta^2_I(t) \chi^2_{I} (\xi_{I}(t,x)) \widetilde{h_I} ]
\end{align*}
Therefore, choosing various parameters involved in Proposition~\ref{prop:intermittent:inverse:div}, in accordance with Remark~\ref{rmk01}, and making use of the estimates \eqref{est.v.dif}, \eqref{est.r}, \eqref{eq:inverse:div:stress:1} and \eqref{eq:inverse:div:error:stress:bound}, we obtain using $\const_{G_I,p} = \mu_q^{-1}\delta_{q+1}$, the following bounds
\begin{equation}
\begin{aligned}
\norm{\as R_{O1}}_N
&\lec \la_{q+1}^N\cdot \frac {\de_{q+1}}{\la_{q+1}\mu_q}, \quad
\norm{\as D_t \as R_{O1}}_{N-1}\lec\la_{q+1}^N\de_{q+1}^\frac12\cdot \frac {\de_{q+1}}{\la_{q+1}\mu_q}.
\end{aligned} 
\end{equation}
On the other hand, we estimate $\as R^1_{O2}$ as follows
\begin{align*}
\norm{\as R^1_{O2}}_N
&\lec  \sum_{N_1+N_2=N} \norm{\theta_0}_{N_1}\norm{w_c}_{N_2} 
\lec \la_{q+1}^N\cdot\frac{\de_{q+1}}{\la_{q+1}\mu_q} ,\\
\norm{\as D_t\as R^1_{O2}}_{N-1}
&\leq \norm{D_{t,\ell} \as R^1_{O2}}_{N-1} + \norm{(w+u_q-u_\ell)\cdot\na \as R^1_{O2}}_{N-1}\\
&\lec  \sum_{N_1+N_2=N-1} 
\norm{D_{t,\ell} \theta_0}_{N_1}\norm{w_c}_{N_2} 
+ \norm{ \theta_o}_{N_1}\norm{D_{t,\ell} w_c}_{N_2}\\
&\quad+ \sum_{N_1+N_2=N-1} 
(\norm{w}_{N_1} + \norm{u_q-u_\ell}_{N_1})\norm{\as R^1_{O2}}_{N_2+1}
\lec \la_{q+1}^{N}\de_{q+1}^\frac12 \cdot\frac{ \de_{q+1}}{\la_{q+1}\mu_q}.
\end{align*}
Therefore, we have 
\begin{equation}\label{est.RO}
\norm{\as R_O}_N \lec \la_{q+1}^N \frac{\de_{q+1}}{\la_{q+1}\mu_q}, \quad
\norm{\as D_t \as R_O}_{N-1}   
\lec \la_{q+1}\de_{q+1}^{\frac12}
\cdot \la_{q+1}^N \frac{\de_{q+1}}{\la_{q+1}\mu_q}.
\end{equation} 
\subsubsection{\bf Mediation stress error}
Recall that 
\[
\as R_M = R_\ell -R_q +(u_q-u_\ell) \theta +  (\vr_q-\vr_\ell) w.
\]
Using \eqref{est.R.dif}, \eqref{est.v.dif}, and \eqref{est.r}, we have
\begin{align*}
\norm{\as R_M}_{N} 
&\lec \norm{R_q-R_\ell}_N+\sum_{N_1+N_2=N}\norm{u_q-u_\ell}_{N_1}\norm{\theta}_{N_2}  +\sum_{N_1+N_2=N}\norm{\vr_q-\vr_\ell}_{N_1}\norm{w}_{N_2}\\
&\lec \la_{q+1}^N \cdot ( {\la_q^\frac12}\la_{q+1}^{-\frac12} \de_q^\frac14\de_{q+1}^\frac34  + (\ell\la_q)^2\de_q^\frac12\de_{q+1}^\frac12)
\lec \la_{q+1}^N \cdot  {\la_q^\frac12}\la_{q+1}^{-\frac12} \de_q^\frac14\de_{q+1}^\frac34.
\end{align*}
To estimate $\as D_t  \as R_M$, we additionally use the
decomposition $\as D_t  \as R_M  = D_{t,\ell} \as R_M + ((u_q-u_\ell)+w)\cdot\na \as R_M $ to obtain
\begin{align*}
\norm{\as D_t \as R_M}_{N-1}
&\lec \norm{\as D_t (R_q-R_\ell)}_{N-1} 
+ \sum_{N_1+N_2=N-1} \norm{\as D_t (u_q-u_\ell)}_{N_1} \norm{\theta}_{N_2} + \norm{u_q-u_\ell}_{N_1} \norm{\as D_t \theta}_{N_2} 
\\
&\qquad + \sum_{N_1+N_2=N-1} \norm{\as D_t (\vr_q-\vr_\ell)}_{N_1} \norm{w}_{N_2} + \norm{\vr_q-\vr_\ell}_{N_1} \norm{\as D_t w}_{N_2}
\nonumber\\
&\lec \la_{q+1}^N \de_{q+1}^\frac12\cdot  {\la_q^\frac12}\la_{q+1}^{-\frac12} \de_q^\frac14\de_{q+1}^\frac34.
\end{align*} 
To summarize, we obtain 
\begin{align}\label{est.RM}
\norm{\as R_M}_{N}
\lec  \la_{q+1}^N  {\la_q^\frac12}\la_{q+1}^{-\frac12} \de_q^\frac14\de_{q+1}^\frac34, \quad
\norm{\as D_t \as R_M}_{N-1}
\lec \la_{q+1}^N \de_{q+1}^\frac12  {\la_q^\frac12}\la_{q+1}^{-\frac12} \de_q^\frac14\de_{q+1}^\frac34.
\end{align}

\subsection{Estimates on the quadratic stress}

In this section, we obtain the relevant estimates for the new Reynolds stress and its new advective derivative $\as D_t S_{q+1} = \partial_t S_{q+1} + u_{q+1}\cdot \nabla S_{q+1}$, summarized in the following proposition. 

\begin{prop}\label{p:qReynolds}
	{There exists  $\bar{b}(\al)>1$ with the following property. For any $1<b<\bar{b}(\al)$ we can find 
		$\Lambda_0 = \Lambda_0 (\al,b,M)$} such that the following estimates hold for every  $\lambda_0 \geq \Lambda_0$: 
	\begin{equation}\label{est.asS}\begin{split}
	\norm{S_{q+1}^{\text{pre}}}_N &\leq C_M\la_{q+1}^N \cdot
	{\la_q^\frac12}\la_{q+1}^{-\frac12} \de_q^\frac14\de_{q+1}^\frac34 
	\leq \la_{q+1}^{N-2\ga} \de_{q+2}, 
	\ \ \qquad\qquad \forall N=0,1,2 \\
	\norm{\as D_t S_{q+1}^{\text{pre}} }_{N-1} &\leq C_M  \la_{q+1}^N \de_{q+1}^\frac12\cdot
	{\la_q^\frac12}\la_{q+1}^{-\frac12} \de_q^\frac14\de_{q+1}^\frac34  \leq \la_{q+1}^{N-2\ga}\de_{q+1}^\frac12 \de_{q+2}, \quad \forall N=1,2, \\
	\norm{\Upsilon^2}_0 + \norm{(\Upsilon^2)'}_0 &\leq \la_{q+1}^{-2\ga} \de_{q+2}.
	\end{split}\end{equation}
	{where $C_M$ depends only upon the $M>1$.}
\end{prop}

\begin{rem}
We remark that we first derive the estimate $\norm{\vr_{q+1}S_{q+1}^{\text{pre}}}_N
\leq \la_{q+1}^{N-2\ga} \de_{q+2}$, and this implies the first estimate of \eqref{est.asS}, thanks to the lower bound of $\vr_{q+1}$ and upper bound of $\norm{\vr_{q+1}}_N$. In the same manner, we can deal with material derivative estimates in \eqref{est.asS}. 
\end{rem}

\begin{rem}
To derive estimates for $S_{q+1}$, first recall that 
$$
S_{q+1} = S_{q+1}^{\text{pre}} - \frac1{\varrho_{q+1}} \sum_{i=1}^6 \Upsilon_i \otimes u_{q+1}
:= S_{q+1}^{\text{pre}} + S_{q+1}^{\text{error}}
$$
Moreover, in view of \eqref{est.asS}, \eqref{est.asph} note that
\begin{align*}
\norm{S_{q+1}^{\text{error}}}_N \leq \la_{q+1}^{N-2\ga} \de_{q+2}, \quad \text{and}\,\, 
\norm{\as D_t S_{q+1}^{\text{error}} }_{N-1} \leq \la_{q+1}^{N-2\ga}\de_{q+1}^\frac12 \de_{q+2}.
\end{align*}
Therefore, combining the above estimates with the estimates in \eqref{est.asS}, we obtain desired estimates for the quadratic term $S_{q+1}$. 
\end{rem}

\subsubsection{\bf Quadratic stress error}
Recall that $\as S_O = \as S_{O1} + \as S_{O2}$ where
\begin{align*}
\as S_{O1} 
&= \idv{\div(\vr_q (w_0 \otimes w_0 + S_\ell - \de_{q+1} \text{Id}))} := \as S_{O11} + \as S_{O12}, \\
\as S_{O2} & =\vr_q (w_0 \otimes w_c) + \vr_q  (w_c \otimes w_0 ) +  \vr_q  (w_c \otimes w_c).
\end{align*}
We first deal with the term $\as S_{O11}:= \idv{ \vr_q \div((w_0 \otimes w_0 + S_\ell - \de_{q+1} \text{Id}))}$. We compute
\begin{align*}
\na\cdot ( w_0 \otimes w_0 + S_\ell - \de_{q+1} \text{Id})
= \div \left[ \sum_{\substack{p\in \Z\\m\in \Z^3\setminus \{0\}}} \de_{q+1} c_{p,m} e^{ i\la_{q+1} m\cdot \xi_I} \right]
=\sum_{p,m} \de_{q+1} \div(c_{p,m}) e^{ i\la_{q+1} m\cdot \xi_I},
\end{align*}
because of $\dot{c}_{I,m} (h_I\cdot m) =0$. Also, since we have
\begin{align*}
D_{t,\ell} \div c_{p,m} & = \div( D_{t,\ell} c_{p,m}) - (\na u)_{ij}\pa_i (c_{p,m})_{jl}, \\
D_{t,\ell} (\vr_q \div c_{p,m}) & = \vr_q ( D_{t,\ell} \div c_{p,m}) + (\div c_{p,m}) D_{t,\ell} \vr_q
\end{align*}
it follows from \eqref{est.d} that $\norm{\vr _q \div c_{p,m}}_{\bar{N}}+(\la_{q+1}\de_{q+1}^\frac12)^{-1}\norm{\vr_q D_{t,\ell} ( \div c_{p,m})}_{\bar{N}} \lec_{\bar{N},M}  \mu_q^{-\bar{N}} \frac{|\dot{c}_{I,m}|}{\mu_q}$ for any $\bar{N}\geq 0$.  For the other term, we first observe that $\norm{D_{t,\ell} \vr_q}_{N-1} \le \la_q^N \de_q$. Hence,
$$
\norm{(\div c_{p,m}) D_{t,\ell} \vr_q}_N \le \norm{c_{p,m}}_{N+1} \norm{D_{t,\ell} \vr_q}_0 \le \mu_q^{-N-1} |\dot{c}_{I,m}| \la_q \de_q \le \mu_q^{-\bar{N}} \frac{|\dot{c}_{I,m}|}{\mu_q} \tau_q^{-1}.
$$
Therefore 
$(\la_{q+1}\de_{q+1}^\frac12)^{-1} \norm{(\div c_{p,m}) D_{t,\ell} \vr_q}_N \le \mu_q^{-\bar{N}} \frac{|\dot{c}_{I,m}|}{\mu_q}.$
Finally using $\supp(c_{p,m})\subset (t_p-\frac12\tau_q, t_p +\frac32\tau_q)\times \R^3$, we apply Corollary \ref{cor.mic2} to get
\begin{equation}\label{est.SO1}
\begin{split}
\norm{\as S_{O11}}_N
&\lec \la_{q+1}^N\cdot \frac {\de_{q+1}}{\la_{q+1}\mu_q}, \quad
\norm{\as D_t \as S_{O11}}_{N-1}\lec\la_{q+1}^N\de_{q+1}^\frac12\cdot \frac {\de_{q+1}}{\la_{q+1}\mu_q} . 
\end{split}
\end{equation}
Next, we deal with the term $\as S_{O12} := \idv{ \na \vr_q ( w_0 \otimes w_0 - S_\ell + \de_{q+1} \text{Id})}$ :
\begin{align*}
\na \vr_q ( w_0 \otimes w_0 - S_\ell + \de_{q+1} \text{Id})
= \sum_{\substack{p\in \Z\\m\in \Z^3\setminus \{0\}}} \de_{q+1} c_{p,m} e^{ i\la_{q+1} m\cdot \xi_I} \na \vr_q.
\end{align*}
Now we have the estimate $\norm{\na \vr _q  c_{p,m}}_{\bar{N}}+(\la_{q+1}\de_{q+1}^\frac12)^{-1}\norm{D_{t,\ell} (\na \vr _q c_{p,m})}_{\bar{N}} \lec_{\bar{N},M}  \mu_q^{-\bar{N}} |\dot{c}_{I,m}| \la_q \de_q^{1/2}$ for any $\bar{N}\geq 0$.
We can apply Corollary \ref{cor.mic2} to obtain
\begin{align}\label{est.SN}
\norm{\as S_{O1}}_N 
\lec \la_{q+1}^N \frac {\de_{q+1}}{\la_{q+1}\tau_q} , \quad 
\norm{\as D_t\as S_{O1}}_{N-1}  \lec \la_{q+1}^N\de_{q+1}^\frac12 \frac {\de_{q+1}}{\la_{q+1}\tau_q}.
\end{align}
On the other hand, we estimate $\as S_{O2}$ as follows,
\begin{align*}
\norm{\as S_{O2}}_N
&\lec  \sum_{N_1+N_2 + N_3=N} (\norm{w_o}_{N_1} + \norm{w_c}_{N_1}) \norm{w_c}_{N_2} \norm{\vr_q}_{N_3}
\lec \la_{q+1}^N\cdot\frac{\de_{q+1}}{\la_{q+1}\mu_q} ,\\
\norm{\as D_t\as S_{O2}}_{N-1}
&\leq \norm{D_{t,\ell} \as S_{O2}}_{N-1} + \norm{(w+u-u_\ell)\cdot\na \as S_{O2}}_{N-1}\\[2pt]
\lec  \sum_{N_1+N_2 +N_3=N-1} 
& (\norm{D_{t,\ell} w_o}_{N_1} + \norm{D_{t,\ell} w_c}_{N_1}) \norm{w_c}_{N_2} \norm{\vr_q}_{N_3}
+ ( \norm{ w_o}_{N_1} + \norm{ w_c}_{N_1} )\norm{D_{t,\ell} w_c}_{N_2} \norm{\vr_q}_{N_3} \\
& \quad + \sum_{N_1+N_2 +N_3=N-1} \norm{D_{t,\ell} \vr_q}_{N_1} (\norm{w_o}_{N_2}+ \norm{w_c}_{N_2}) \norm{w_c}_{N_3}\\
&\qquad+ \sum_{N_1+N_2=N-1} 
(\norm{w}_{N_1} + \norm{u-u_\ell}_{N_1})\norm{\as S_{O2}}_{N_2+1}
\lec \la_{q+1}^{N}\de_{q+1}^\frac12 \cdot\frac{ \de_{q+1}}{\la_{q+1}\mu_q}.
\end{align*}
Therefore, we have 
\begin{equation}\label{est.SO}
\norm{\as S_O}_N \lec \la_{q+1}^N \frac{\de_{q+1}}{\la_{q+1}\mu_q}, \quad
\norm{\as D_t \as S_O}_{N-1}   
\lec \la^N_{q+1}\de_{q+1}^{\frac12}
\cdot \frac{\de_{q+1}}{\la_{q+1}\mu_q}.
\end{equation}

\subsubsection{\bf Quadratic Mediation error} 
Recall that
\begin{align*}
\na \cdot \as S_{M}=\div(\vr_q w \otimes_s (u_q - u_\ell) )  +  \div(\theta (u_q - u_\ell) \otimes (u_q - u_\ell )) + \div (\vr_q (S_q - S_\ell )).
\end{align*}
Therefore,
\begin{align*}
\as S_{M} & = \underbrace{(\vr_q w \otimes_s (u_q - u_\ell))}_{\as S_{M11}}
+ \underbrace{ (\theta (u_q - u_\ell) \otimes (u_q - u_\ell))}_{\as S_{M21}} 
+ \underbrace{\vr_q (S_q - S_\ell )}_{\as S_{M31}} 
\end{align*}
For the first term, we estimate as follows:
\begin{align*}
\norm{\as S_{M11} }_N & \le \sum_{N_1+N_2 + N_3=N} \norm{\vr_q}_{N_1} \norm{w}_{N_2} \norm{u_q - u_\ell}_{N_3} \lesssim \la_{q+1}^N \de_{q+1}^\frac12 \ell^2 \la_q^2 \de_q^{\frac12}
\lec \la_{q+1}^N\la_q^{1/2}\la_{q+1}^{-1/2} \de_q^\frac14\de_{q+1}^\frac34,\\
\norm{\as D_t\as S_{M11}}_{N-1}
&\leq \norm{D_{t,\ell} \as S_{M11}}_{N-1} + \norm{(w+u-u_\ell)\cdot\na \as S_{M11}}_{N-1}\\[2pt]
\lec  \sum_{N_1+N_2 +N_3=N-1} 
&\norm{D_{t,\ell} w}_{N_1}\norm{u_q -u_\ell}_{N_2} \norm{\vr_q}_{N_3}
+ \norm{ w}_{N_1}\norm{D_{t,\ell} (u_q -u_\ell)}_{N_2} \norm{\vr_q}_{N_3} \\
&\quad+ \norm{D_{t,\ell} \vr_q}_{N_1}\norm{w}_{N_2} \norm{u_q-u_\ell}_{N_3}
 +\sum_{N_1+N_2=N-1} 
(\norm{w}_{N_1} + \norm{u-u_\ell}_{N_1})\norm{\as S_{M11}}_{N_2+1} \\
&\lec \la_{q+1}^N \de_{q+1}^{1/2} \la_q^{1/2}\la_{q+1}^{-1/2} \de_q^\frac14\de_{q+1}^\frac34.
\end{align*}
Similarly, we can estimate the term $S_{M21}$. Finally, in view of estimate \eqref{est.vp}, \eqref{e:ad-pressure}, \eqref{est.v.dif}, and \eqref{est.S.dif}, we bound the last term as follows:
\begin{align*}
\norm{\as S_{M31} }_N & \lesssim \sum_{N_1+N_2=N}\norm{\vr_q}_{N_1} \norm{S_q - S_\ell}_{N_2} 
\lec \la_{q+1}^N \la_q^{1/2}\la_{q+1}^{-1/2} \de_q^\frac14\de_{q+1}^\frac34, \\
\norm{\as D_t\as S_{M31}}_{N-1} & 
\leq \norm{D_{t,\ell} \as S_{M31}}_{N-1} + \norm{(w+u-u_\ell)\cdot\na \as S_{M31}}_{N-1} \\
& \lec  \sum_{N_1+N_2=N-1} 
\norm{D_{t,\ell} \vr_q}_{N_1}\norm{S_q -S_\ell}_{N_2} + \norm{\vr_q}_{N_1}\norm{D_{t,\ell}(S_q -S_\ell)}_{N_2} \\
&+\sum_{N_1+N_2=N-1} 
(\norm{w}_{N_1} + \norm{u-u_\ell}_{N_1})\norm{\as S_{M31}}_{N_2+1} 
\lec \la_{q+1}^N \de_{q+1}^{1/2} \la_q^{1/2}\la_{q+1}^{-1/2} \de_q^\frac14\de_{q+1}^\frac34.
\end{align*}
Hence, combining above estimates, we get the desired estimates for $\as S_{M}$.

\subsubsection{\bf Quadratic Nash error}
Set $\as S_N = \idv{(\vr_\ell w \cdot \na) u_\ell}$ and observe that
\begin{align*}
(\vr_\ell w \cdot \na) u_\ell
&=  \sum_{p} \sum_{m\in \Z^3\setminus \{0\}} \vr_\ell \de_{q+1}^\frac 12((b_{p,m} + (\la_{q+1}\mu_q)^{-1}g_{p,m})\cdot \na)u_\ell e^{i\la_{q+1} m \cdot \xi_I}.
\end{align*}
Since $b_{p,m}$ and $g_{p,m}$ satisfy
$\supp(b_{p,m}), \supp(g_{p,m})\subset (t_p-\frac 12\tau_q, t_p + \frac 32\tau_q) \times \R^3$ and
\begin{align*}
\norm{(\vr_\ell (b_{p,m} + (\la_{q+1}\mu_q)^{-1}g_{p,m})\cdot \na)u_\ell}_{\bar{N}}
&\lec_{\bar N} \mu_q^{-{\bar{N}}} |\dot{b}_{I,m}|\la_q\de_{q}^\frac12\\
\norm{D_{t,\ell}[(\vr_\ell(b_{p,m} + (\la_{q+1}\mu_q)^{-1}g_{p,m})\cdot \na)u_\ell]}_{\bar{N}}
&\lec_{\bar N} \la_{q+1}\de_{q+1}^\frac12 \mu_q^{-{\bar{N}}} |\dot{b}_{I,m}|\la_q\de_{q}^\frac12
\end{align*}
for any ${\bar{N}}\geq0$ by \eqref{est.vp}, \eqref{est.Dtvl}, \eqref{est.b}, and \eqref{est.g}, we apply Corollary \ref{cor.mic2} and obtain
\begin{align}\label{est.phiN}
\norm{\as S_N}_N 
\lec \la_{q+1}^N \frac {\de_{q+1}^\frac 12}{\la_{q+1}\tau_q} , \quad 
\norm{\as D_t\as S_N}_{N-1}  \lec \la_{q+1}^N\de_{q+1}^\frac12 \frac {\de_{q+1}^\frac 12}{\la_{q+1}\tau_q}.
\end{align}

%
%

\subsubsection{\bf Quadratic transport  error I}
Set $\as S_{T_1} = \idv{\theta (\partial_t u_\ell + (u_\ell \cdot \nabla) u_\ell)}$. Note that 
$$
\theta (\partial_t u_\ell + (u_\ell \cdot \nabla) u_\ell) = \theta_0 (\partial_t u_\ell + (u_\ell \cdot \nabla) u_\ell) + \theta_c (\partial_t u_\ell + (u_\ell \cdot \nabla) u_\ell) := \as S_{T^1_1} + \as S_{T^2_1}.
$$ 
To deal with the first term, first observe that
\begin{align*}
\theta (\partial_t u_\ell + (u_\ell \cdot \nabla) u_\ell)
&=  \sum_{p} \sum_{m\in \Z^3\setminus \{0\}} \de_{q+1}^\frac 12 a_{p,m} D_{t, \ell} u_\ell  e^{i\la_{q+1} m\cdot \xi_I}.
\end{align*}
Thanks to Remark~\ref{rem.new}, we estimate using \eqref{est.b}, and $\la_q \de_q^{1/2} \le \tau_q^{-1} \le \la_{q+1}\de_{q+1}^\frac12$,
\begin{align*}
\norm{a_{p,m} D_{t, \ell} u_\ell}_{\bar{N}} 
&\lec_{\bar N} \mu_q^{-{\bar{N}}} |\dot{a}_{I,m}|\la_q\de_{q}^{1/2}
\lec \mu_q^{-{\bar{N}}} |\dot{a}_{I,m}|\tau_q^{-1}\\
\norm{D_{t,\ell}[a_{p,m} D_{t, \ell} u_\ell]}_{\bar{N}}
&\lec_{\bar N} \la_{q+1}\de_{q+1}^\frac12 \mu_q^{-{\bar{N}}} |\dot{a}_{I,m}|\tau_q^{-1}
\end{align*}
Then, we can apply Corollary \ref{cor.mic2} and obtain
\begin{align}\label{est.SN1}
\norm{\as S_{T^1_1}}_N 
\lec \la_{q+1}^N \frac {\de_{q+1}^\frac 12}{\la_{q+1}\tau_q} , \quad 
\norm{\as D_t\as S_{T^1_1}}_{N-1}  \lec \la_{q+1}^N\de_{q+1}^\frac12 \frac {\de_{q+1}^\frac 12}{\la_{q+1}\tau_q}.
\end{align}
To deal with the other term $\as S_{T^2_1}$, we simply take advantage of Lemma~\ref{phase} to obtain the same result. Therefore, we get the desired result for $\as S_{T_1}$.

\subsubsection{\bf Quadratic transport  error II}
Set $\as S_{T_2} = \idv{\vr_\ell (\partial_t w + (u_\ell \cdot \nabla) w)}$ and observe that
\begin{align*}
\vr_\ell (\partial_t w + (u_\ell \cdot \nabla) w)
=\sum_{p} \sum_{m} \de_{q+1}^\frac 12  \vr_\ell D_{t,\ell} (b_{p,m} + (\la_{q+1} \mu_q)^{-1} g_{p,m}) e^{i\la_{q+1} m\cdot \xi_I}.
\end{align*}
Since $b_{p,m}$ and $g_{p,m}$ satisfy
$\supp(b_{p,m}), \supp(g_{p,m}) \subset (t_p-\frac 12\tau_q, t_p + \frac 32\tau_q) \times \R^3$
and 
\begin{align*}
&\norm{\vr_\ell D_{t,\ell} (b_{p,m} + (\la_{q+1} \mu_q)^{-1} g_{p,m}) }_{\bar{N}} \lec_{\bar{N},M} \mu_q^{-\bar{N}} \frac {|\dot{a}_{I,m}|}{\tau_q} \\
&\norm{D_{t,\ell}(\vr_\ell  D_{t,\ell} (b_{p,m} + (\la_{q+1} \mu_q)^{-1} g_{p,m}) )}_{\bar{N}}
\lec_{\bar{N},M} \la_{q+1}\de_{q+1}^\frac 12 \mu_q^{-\bar{N}} \frac {|\dot{a}_{I,m}|}{\tau_q},
\end{align*}
for any $\bar{N}\geq 0$ by \eqref{est.b}, we can apply Corollary \ref{cor.mic2} to get
\begin{align}\label{est.RT1}
\norm{\as S_{T_1}}_N
\lec \la_{q+1}^N\frac {\de_{q+1}^\frac 12}{\la_{q+1} \tau_q} , \quad
\norm{\as D_t \as S_{T_2}}_{N-1}  \lec \la_{q+1}^N\de_{q+1}^\frac12 \frac {\de_{q+1}^\frac 12}{\la_{q+1} \tau_q}.
\end{align}

\subsubsection{\bf Estimates on $\Upsilon_4, \Upsilon_5, \Upsilon_6$} We will in fact show the stronger estimates
\begin{align}
\norm{\Upsilon_i'}_0 &\leq \frac{1}{5{(T+ \tau_0)}} \la_{q+1}^{-2\ga} \de_{q+2}, \quad \text{for}\,\, i = 4,5,6.
 \label{e:est_rho_5}
\end{align}
from which the estimates follow by integration
\begin{align}
\norm{\Upsilon_i}_0 &\leq \frac{1}{5} \la_{q+1}^{-2\ga} \de_{q+2}. 
\end{align}
This can be achieved by the similar argument, as we used in the previous subsections, taking advantage of Lemma \ref{phase} and the representations \eqref{rep.W}-\eqref{e:rep4}.

%

\subsection{Estimates for the new current}\label{sec.cur} 

In this section, we obtain the last needed estimates, on the new unsolved current $\Phi_{q+1}$, which we
summarize in the following proposition.

\begin{prop}\label{p:current}
{There exists  $\bar{b}(\al)>1$ with the following property. For any $1<b<\bar{b}(\al)$ there is 
$\Lambda_0 = \Lambda_0 (\al,b,M)$ such that} the following estimates hold for $\lambda_0 \geq \Lambda_0 $:
\begin{align}\begin{split}
\label{est.asph}
\norm{\Phi_{q+1}}_N
&\leq \la_{q+1}^{N-4\ga} \de_{q+2}^\frac32 , \qquad\quad \forall N=0,1,2,\\
\norm{\as D_t \Phi_{q+1}}_{N-1} 
&\leq \la_{q+1}^{N-4\ga}\de_{q+1}^\frac12 \de_{q+2}^\frac32, \quad \forall N=1,2,\\
\norm{\Upsilon^1}_0 + \norm{(\Upsilon^1)'}_0 &\leq \la_{q+1}^{-4\ga} \de_{q+2}^\frac32.
\end{split}
\end{align}
\end{prop}

Without mentioning, we assume that $N$ is in the range above and allow the dependence on $M$ of the implicit constants in $\lec$ in this section.
For convenience, we single out the following fact, which will be repeatedly used: note that there exists $\bar{b}(\al)>1$ such that for any $1<b<\bar{b}(\al)$ and a constant $C_M$ depending only on $M$, we can find $\La_0=\La_0(\al, b,M)$ which gives
\[
C_M\left[ \frac {\de_{q+1}}{\la_{q+1}\tau_q} +\frac{\de_{q+1}^\frac32}{\la_{q+1}\mu_q} + \frac {\la_q^\frac12}{\la_{q+1}^\frac12} \de_q^\frac14\de_{q+1}^\frac54\right]
\leq  \la_{q+1}^{-4\ga} \de_{q+2}^\frac32,
\] 
for any $\lambda_0 \geq \La_0$.

\subsubsection{\bf High frequency current error - I}
We start by observing that $\as\ph_{H1}$ is
 \begin{equation}\label{e:formula_H1}
 \as\ph_{H1}= \idv{(-\div P_{\le \ell^{-1}}R_q + Q(\vr_q,u_q) ) w}
 \end{equation}
 by \eqref{e:average-free}. We thus can apply Corollary \ref{cor.mic2} to
\begin{align*}
(-\div P_{\le \ell^{-1}}&R_q + Q(\vr_q,u_q)) w\\
&= \sum_{p,m} (-\div P_{\le \ell^{-1}}R_q + Q(\vr_q,u_q)) \de_{q+1}^\frac12(b_{p,m} + (\la_{q+1}\mu_q)^{-1}g_{p,m})e^{i\la_{q+1}m\cdot \xi_I}.
\end{align*}
Indeed, using \eqref{est.R}, \eqref{est.Qvv}, \eqref{est.b}, \eqref{est.g}, we obtain
\begin{align*}
&\norm{\as\ph_{H1} }_N
\lec \la_{q+1}^{N-1} \de_{q+1}^\frac12 (\la_q^{1-2\ga}\de_{q+1} + (\ell \la_q) \la_q\de_q) 
\le \la_{q+1}^{N- 4 \gamma} \de_{q+2}^\frac32.
\end{align*}
Furthemore, \eqref{est.vp}, \eqref{est.R}, \eqref{est.v.dif}, and \eqref{est.com1} imply
\begin{align*}
\norm{D_{t,\ell} \div P_{\le \ell^{-1}} R_q}_{N-1}
&\leq \norm{\div P_{\le \ell^{-1}}  D_{t,\ell} R_q}_{N-1}
+ \norm{\div [u_\ell\cdot\na, P_{\le \ell^{-1}}] R_q}_{N-1}\\
&\quad+ \norm{(\na u_\ell)_{ki} \pa_k P_{\le l^{-1}} (R_q)_{i}}_{N-1}\\
&\lec \la_{q+1}^{N-1} (\norm{D_{t,\ell} R_q}_{1}
+ \norm{ [u_\ell\cdot\na, P_{\le \ell^{-1}}] R_q}_{1}
+\norm{\na u_q}_0 \norm{ \na R_q}_0)\\
&\lec \la_{q+1}^{N}\de_{q+1}^\frac12 \la_q^{1-2\ga} \de_{q+1}.
\end{align*}
Then, it follows that
\begin{align*}
\norm{\as D_t \ph_{H1}}_{N-1} \le 
\la_{q+1}^{N- 4 \gamma} \de_{q+1}^\frac12 \de_{q+2}^\frac32.
\end{align*}

\subsubsection{\bf Transport current error $\&$ high frequency current error -II}
\label{transport error}
Note that the transport term and the high-frequency term are similar in nature; we therefore combine them to facilitate the estimation. To proceed, first note that
\begin{align*}
\theta w -R_q + R_{q+1}^{\text{pre}} + (u_q -u_\ell) \theta + (\vr_q -\vr_\ell) w=
(\theta_0 w_0-R_\ell )-\as R_T -\as R_N-\as R_{O1},
\end{align*}
For convenience, we define the operator $\mathcal{A}: = D_{t, \ell} + (. \cdot \nabla)u_{\ell}$. Then 
using \eqref{est.phiN} and \eqref{e:average-free}, we can write
\begin{align}
\div(\as \ph_{TH}):= \div(\as\ph_{H2} + \as\ph_{T})
&= \mathcal{A} (R_{q+1}^{\text{pre}} -R_q + \theta w + (u_q -u_\ell) \theta + (\vr_q -\vr_\ell) w) 
\label{e:formula_H2}\\
&= \mathcal{A}((\theta_0 w_0-R_\ell ) - \as R_{O1}  -\as R_T -\as R_N).\nonumber
\end{align}
For the first term, we note that we can estimate $\as\ph_{TH1}:=\idvl{\mathcal{A} (\theta_0 w_0 - R_\ell)}$ using similar arguments as presented in subsection~\ref{OS}. Indeed, using the estimates \eqref{est.v.dif}, \eqref{est.r}, \eqref{eq:inverse:div:stress:1} and \eqref{eq:inverse:div:error:stress:bound}, we obtain 
\begin{align*}
\norm{\as\ph_{TH1}}_N
\lec  \la_{q+1}^N\cdot
\Big( \frac {\de_{q+1}}{\la_{q+1} \tau_q} + \la_q \de_q^\frac12 \frac {\de_{q+1}}{\la_{q+1}}\Big) ,\quad
\norm{\as D_t\as\ph_{TH1}}_{N-1} \lec \la_{q+1}^N\de_{q+1}^\frac12 \cdot
\Big(\frac {\de_{q+1}}{\la_{q+1} \tau_q} + \la_q \de_q^\frac12 \frac {\de_{q+1}}{\la_{q+1}} \Big). 
\end{align*}

Next, to deal with other terms, recall that the Reynolds stress errors $\as R_\tri$, which represents either $\as R_{O1}$, $\as R_T$, or $\as R_N$, can be written as $\as R_\tri  = \idvl {H_\tri}$. 
Furthermore, such $H_\tri$ has the form $G \zeta \circ \xi_{I}$, where $\zeta$ satisfies additional assumption as described in Appendix~\ref{apenC}. Indeed, using the mean zero property of $\zeta$ and making use of the inverse divergence formula (cf. Proposition~\ref{prop:intermittent:inverse:div}), we have
\begin{align*}
G \; \zeta\circ \xi_I  &=  \div R + \div \RR_{nonlocal}
=: \div \underbrace{\left( \divH \left( G\zeta \circ \xi_I \right) \right)}_{:=R} + \div \underbrace{\left(\cR(G \zeta \circ \xi_I) \right)}_{:=\RR_{nonlocal}}
\end{align*}
Therefore
$$
\cRl H_{\Delta} := \cRl (G \; \zeta\circ \xi_I) = R + \RR_{nonlocal} := \divH (H_{\Delta}) + \divR (H_{\Delta}).
$$
We now estimate both ($R$ and $\RR_{nonlocal}$) terms separately. Indeed, for a typical term, we have
$$
\div(\as\ph_{TH}) = \mathcal{A} (\divH (H_{\Delta}) + \divR (H_{\Delta})).
$$
To obtain $\as\ph_{TH}$, we use
$$
\as\ph_{TH} = \cRl ( \mathcal{A} (\divH (H_{\Delta}))) + \cR ( \mathcal{A} (\divR (H_{\Delta}))
$$
Firstly, we can now make use of specific structure of $R$ and estimate the terms involving $\divH(H_{\Delta})$ as follows. To deal with the Nash term, i.e., $\as R_N$, we can apply the inverse-divergence lemma (cf. Proposition~\ref{prop:intermittent:inverse:div}) to obtain
\begin{align*}
\norm{\divH \mathcal{A} (\divH (H_N))}_N & \lesssim \lambda_{q+1}^N \ta_q^{-1} \frac{\la_{q} \de_{q}^{1/2} \de_{q+1}^{1/2}}{\la_{q+1}^2} =\lambda_{q+1}^N  \cdot \frac{\de_{q+1}^{1/2}}{\la_{q+1}\ta_q } \,\,\underbrace{\frac{\la_{q} \de_{q}^{1/2}}{\la_{q+1}}}_{\lesssim \de_{q+1}^{1/2}} \lesssim \lambda_{q+1}^N \cdot \frac{\de_{q+1}}{\la_{q+1}\ta_q }, \\
\norm{D_{t,\ell} \divH [\mathcal{A} (\divH(H_N))]}_{N-1} & \lesssim 
\frac{\de_{q+1}^{1/2}}{\la_{q+1}\ta_q } \,\,\frac{\la_{q} \de_{q}^{1/2}}{\la_{q+1}}\cdot \lambda_{q+1}^{N-1} \ta_q^{-1}
\lesssim \la_{q+1}^N\de_{q+1}^\frac12 \cdot
\frac {\de_{q+1}}{\la_{q+1} \tau_q}.
\end{align*}
Therefore, we can derive estimate for 
$$
\as D_t \divH [\mathcal{A} (\divH (H_N))] = D_{t,\ell} \divH [\mathcal{A} (\divH (H_N))] + ((u_q - u_\ell) + w) \cdot \nabla [\divH \mathcal{A} (\divH (H_N))]
$$
Indeed, we already have the estimate for the first term. For the second term, we make use of $\ell^2 \la_q^2 \de_q^{1/2} \le \de_{q+1}^{1/2}$ to conclude 
$$
\norm{((u_q - u_\ell) + w) \cdot \nabla [\divH \mathcal{A} (\divH (H_N))]}_{N-1} 
\lesssim (\norm{u_q -u_\ell}_0 + \|w\|_0) \norm{\divH \mathcal{A} (\divH (H_N))}_N
\lesssim \la_{q+1}^N\de_{q+1}^\frac12 \cdot
\frac {\de_{q+1}}{\la_{q+1} \tau_q}.
$$
Similarly, to deal with the transport term, i.e., $\as R_T$, we can apply the inverse-divergence lemma (cf. Proposition~\ref{prop:intermittent:inverse:div}) to obtain
\begin{align*}
\norm{\divH \mathcal{A} (\divH (H_T))}_N & \lesssim \lambda_{q+1}^N \ta_q^{-1} \frac{\ta_q^{-1} \de_{q+1}^{1/2}}{\la_{q+1}^2} =\lambda_{q+1}^N  \cdot \frac{\de_{q+1}^{1/2}}{\la_{q+1}\ta_q } \,\,\underbrace{\frac{\ta_q^{-1}}{\la_{q+1}}}_{\lesssim \de_{q+1}^{1/2}} \lesssim \lambda_{q+1}^N \cdot \frac{\de_{q+1}}{\la_{q+1}\ta_q }, \\
\norm{\as D_t \divH [\mathcal{A} (\divH (H_T))]}_{N-1} & \lesssim 
\frac{\de_{q+1}^{1/2}}{\la_{q+1}\ta_q } \,\,\frac{\ta_q^{-1}}{\la_{q+1}}\cdot \lambda_{q+1}^{N-1} \ta_q^{-1}
\lesssim \la_{q+1}^N\de_{q+1}^\frac12 \cdot
\frac {\de_{q+1}}{\la_{q+1} \tau_q}.
\end{align*}
Finally, to deal with the oscillation term, i.e., $\as R_{O1}$, we can apply the inverse-divergence lemma (cf. Proposition~\ref{prop:intermittent:inverse:div}) to obtain
\begin{align*}
\norm{\divH \mathcal{A} (\divH (H_{O1}))}_N & \lesssim \lambda_{q+1}^N \ta_q^{-1} \frac{\mu_q^{-1} \de_{q+1}}{\la_{q+1}^2} =\lambda_{q+1}^N  \cdot \frac{\de_{q+1}}{\la_{q+1}\mu_q } \,\,\underbrace{\frac{\ta_q^{-1}}{\la_{q+1}}}_{\lesssim \de_{q+1}^{1/2}} \lesssim \lambda_{q+1}^N \cdot \frac{\de_{q+1}^{3/2}}{\la_{q+1}\mu_q }, \\
\norm{\as D_t \divH [\mathcal{A} (\divH (H_{O1}))]}_{N-1} & \lesssim 
\frac{\de_{q+1}}{\la_{q+1}\mu_q } \,\,\frac{\ta_q^{-1}}{\la_{q+1}}\cdot \lambda_{q+1}^{N-1} \ta_q^{-1}
\lesssim \la_{q+1}^N\de_{q+1}^\frac12 \cdot
\frac {\de_{q+1}^{3/2}}{\la_{q+1} \mu_q}.
\end{align*}

Again, we can use again the inverse divergence lemma (cf. Proposition~\ref{prop:intermittent:inverse:div}), and choose the parameter $\mathrm{d}$ very large (in accordance with Remark~\ref{rmk01}) to conclude that 
\begin{align*}
&\norm{\divR \mathcal{A} (\divH (H_{N}))}_N \lesssim \lambda_{q+1}^N \cdot \frac{\de_{q+1}}{\la_{q+1}\ta_q }, \quad
\norm{\as D_t [\divR \mathcal{A} (\divH (H_{N}))]}_{N-1} \lesssim \lambda_{q+1}^N \de_{q+1}^{1/2} \cdot \frac{\de_{q+1}}{\la_{q+1}\ta_q }, \\
&\norm{\divR \mathcal{A} (\divH (H_{T}))}_N \lesssim \lambda_{q+1}^N \cdot \frac{\de_{q+1}}{\la_{q+1}\ta_q }, \quad 
\norm{\as D_t [\divR \mathcal{A} (\divH (H_{T}))]}_{N-1} \lesssim \lambda_{q+1}^N \de_{q+1}^{1/2} \cdot \frac{\de_{q+1}}{\la_{q+1}\ta_q }, \\
& \norm{\divR \mathcal{A} (\divH (H_{O1}))}_N \lesssim \lambda_{q+1}^N \cdot \frac{\de_{q+1}^{3/2}}{\la_{q+1}\mu_q}, \quad 
\norm{\as D_t [\divR \mathcal{A} (\divH (H_{O1}))]}_{N-1} \lesssim \lambda_{q+1}^N \de_{q+1}^{1/2} \cdot \frac{\de_{q+1}^{3/2}}{\la_{q+1}\mu_q}.
\end{align*}

Next, to estimate various terms involving the non-local term $\cR H_{\Delta}$, we may proceed slightly differently. Indeed, we can estimate a typical term involving the non-local term as 
$$
\norm{\cR \mathcal{A} (\divR H_{\Delta})}_N = \norm{\cR \mathcal{A} (\RR_{nonlocal})}_N {\le}\norm{\mathcal{A} (\RR_{nonlocal})}_{N}
$$
We can bound the above term appropriately, by choosing the parameter $\mathrm{d}$ very large. Since the operator $\mathcal{A}$ involves a material derivative and a multiplication by $\nabla u_{\ell}$, we can use directly \eqref{eq:inverse:div:error:stress:bound} to deal with the term involving material derivative, and \eqref{eq:inverse:div:error:stress:bound} and the product rule for the other term. Indeed, by our choice of 
$\mathrm{d}$ in Remark~\ref{rmk01}, we have
\begin{align*}
&\norm{\cR \mathcal{A} (\divR H_{N})}_N \lesssim \lambda_{q+1}^N \cdot \frac{\de_{q+1}}{\la_{q+1}\ta_q }, \quad
\norm{\cR \mathcal{A} (\divR H_{T})}_N \lesssim \lambda_{q+1}^N \cdot \frac{\de_{q+1}}{\la_{q+1}\ta_q }, \\
& \norm{\cR \mathcal{A} (\divR H_{O1})}_N \lesssim \lambda_{q+1}^N \cdot \frac{\de_{q+1}^{3/2}}{\la_{q+1}\mu_q}.
\end{align*}
Next, we also need material derivative estimates of typical terms like
\begin{align}
\label{neww}
\as D_t [\cR \mathcal{A}( \RR_{nonlocal}) ] = D_{t, \ell} [\cR \mathcal{A} (\RR_{nonlocal}) ] + (u_{q+1} - u_\ell) \cdot \nabla [\cR \mathcal{A} (\RR_{nonlocal}) ],
\end{align}
where we can estimate the last term as follows:
\begin{align*}
&\| (u_{q+1} - u_\ell) \cdot \nabla [\cR \mathcal{A} (\RR_{nonlocal})] \|_{N-1} \\
&\le \quad \sum_{N_1 + N_2 = N-1} \| (u_{q+1} - u_\ell) \|_{N_1} \norm{\cR \mathcal{A} (\RR_{nonlocal})}_{N_2 +1} \\
&\le \sum_{N_1 + N_2 = N-1} \| (u_{q+1} - u_\ell) \|_{N_1}  \norm{\mathcal{A} (\RR_{nonlocal})}_{N_2 +1}.
\end{align*}
Again, by our choice of large parameter $\mathrm{d}$ in Remark~\ref{rmk01}, we have
\begin{align*}
&\norm{(u_{q+1} - u_\ell) \cdot \nabla [\cR \mathcal{A} (\divR H_{N})]}_{N-1} \lesssim \lambda_{q+1}^N \de_{q+1}^{1/2} \cdot \frac{\de_{q+1}}{\la_{q+1}\ta_q }, \\
&\norm{(u_{q+1} - u_\ell) \cdot \nabla [\cR \mathcal{A} (\divR H_{T})]}_{N-1} \lesssim \lambda_{q+1}^N \de_{q+1}^{1/2} \cdot \frac{\de_{q+1}}{\la_{q+1}\ta_q }, \\
&\norm{(u_{q+1} - u_\ell) \cdot \nabla [\cR \mathcal{A} (\divR H_{O1})]}_{N-1} \lesssim \lambda_{q+1}^N \de_{q+1}^{1/2} \cdot \frac{\de_{q+1}^{3/2}}{\la_{q+1}\mu_q}.
\end{align*}
To deal with the first term in \eqref{neww}, we write
$$
D_{t, \ell} \cR \mathcal{A} (\RR_{nonlocal}) =  \cR D_{t,\ell} \mathcal{A}( \RR_{nonlocal}) + [u_\ell \cdot \nabla, \cR] \mathcal{A} (\RR_{nonlocal}),
$$
where we can estimate the first term, by making direct use of the estimate obtained in the Proposition~\ref{prop:intermittent:inverse:div}. Indeed, notice that 
$$
\cR D_{t,\ell} \mathcal{A}( \RR_{nonlocal}) = \cR D^2_{t,\ell} ( \RR_{nonlocal}) + \cR  D_{t, \ell}\Big( (\RR_{nonlocal} \cdot \nabla) u_{\ell} \Big),
$$
where, again, we can use directly \eqref{eq:inverse:div:error:stress:bound} to deal with the first term, and \eqref{eq:inverse:div:error:stress:bound} and the product rule for the other term. 
For the commutator term, we may proceed as follows:
$$
[u_\ell \cdot \nabla, \cR] \mathcal{A} (\RR_{nonlocal}) = (u_\ell \cdot \nabla) \cR \mathcal{A} (\RR_{nonlocal}) - \cR (u_\ell \cdot \nabla) \mathcal{A} (\RR_{nonlocal}),
$$
where
\begin{align*}
\| (u_\ell \cdot \nabla) \cR \mathcal{A} (\RR_{nonlocal}) \|_{N-1} & \le \sum_{N_1 + N_2 =N-1}
\| u_\ell \|_{N_1}  \norm{\mathcal{A}(\RR_{nonlocal})}_{N_2 +1}, \\
\| \cR (u_\ell \cdot \nabla) \mathcal{A} (\RR_{nonlocal}) \|_{N-1} & \le \sum_{N_1 + N_2 =N-1}
\| u_\ell \|_{N_1}  \norm{\mathcal{A}(\RR_{nonlocal})}_{N_2 +1}.
\end{align*}
Therefore, by our choice of $\mathrm{d}$ in Remark~\ref{rmk01}, we have
\begin{align*}
&\norm{D_{t,\ell} [\cR \mathcal{A} (\divR H_{N})]}_{N-1} \lesssim \lambda_{q+1}^N \de_{q+1}^{1/2} \cdot \frac{\de_{q+1}}{\la_{q+1}\ta_q }, \quad
\norm{D_{t,\ell} [\cR \mathcal{A} (\divR H_{T})]}_{N-1} \lesssim \lambda_{q+1}^N \de_{q+1}^{1/2} \cdot \frac{\de_{q+1}}{\la_{q+1}\ta_q }, \\
&\norm{D_{t,\ell} [\cR \mathcal{A} (\divR H_{O1})]}_{N-1} \lesssim \lambda_{q+1}^N \de_{q+1}^{1/2} \cdot \frac{\de_{q+1}^{3/2}}{\la_{q+1}\mu_q}.
\end{align*}
This concludes the estimates for both the transport current error and the high-frequency current error terms.

\subsubsection{\bf Oscillation current error} 
Let us first write $\as\ph_{O} := \as\ph_{O1} + \as\ph_{O2}$, where
\begin{align*}
\as\ph_{O1} & =\idv{\div\left(\theta_0 w_0 \otimes w_0 + \Phi_\ell \right)}, \\
\as\ph_{O2} & =\theta w \otimes w - \theta_0 w_0 \otimes w_0.
\end{align*} 
To deal with the first term, notice that \eqref{e:rep4} gives
\begin{align*}
\div\left(\theta_0 w_0 \otimes w_0 + \Phi_\ell \right)
&= \div\left(\sum_{m\in \Z} \sum_{k\in \Z^3\setminus \{0\}} \de_{q+1}^\frac 32e_{m,k} e^{ i\la_{q+1} k\cdot \xi_I}\right)
= \sum_{m,k}  \de_{q+1}^\frac 32 \div(e_{m,k})e^{ i\la_{q+1} k\cdot \xi_I},
\end{align*}
because of $\dot{e}_{I,k} (h_I\cdot k) =0$. Also, we have
\begin{align*}
\norm{D_{t,\ell} \div e_{m,k}}_{\bar N}
\lec \norm{D_{t,\ell} e_{m,k}}_{\bar N+1} + \norm{(\na u_l)^{\top}:\na e_{m,k}}_{\bar N }
{\lec_{M,\bar N}\mu_{q}^{-\bar N} \frac{\la_{q+1}\de_{q+1}^\frac12}{\mu_q}}, \quad \forall \bar N\geq 0.
\end{align*}
Therefore, using $\supp(e_{m,k}) \subset (t_m -\frac 12\tau_q, t_m + \frac 32\tau_q) \times \R^3$, it follows from Corollary \ref{cor.mic2} with \eqref{est.d} that
\begin{align*}
\norm{\as\ph_{O1}}_N
\lec \la_{q+1}^N \cdot \frac {\de_{q+1}^\frac 32}{\la_{q+1}\mu_q}, \quad
\norm{\as D_t\as\ph_{O1}}_{N-1} 
\lec \la_{q+1}^N \de_{q+1}^\frac12\cdot \frac {\de_{q+1}^\frac 32}{\la_{q+1}\mu_q} .
\end{align*}
Next recall that $\as\ph_{O2} = (\theta w \otimes w - \theta_0 w_0 \otimes w_0) $. Then, \eqref{est.Wc}-\eqref{est.w.indM} imply
\begin{align*}
\norm{\as\ph_{O2}}_N
&\lec 
\norm{\theta_0 (w_c\otimes w_0)}_N 
+ \norm{\theta_0 (w_0\otimes w_c)}_N + \norm{\theta_0 (w_c\otimes w_c)}_N + \norm{\theta_c (w_0\otimes w_0)}_N\\
& \quad + \norm{\theta_c (w_c\otimes w_0)}_N 
+ \norm{\theta_c (w_0\otimes w_c)}_N + \norm{\theta_c (w_c\otimes w_c)}_N
\lec \la_{q+1}^N \cdot \frac {\de_{q+1}^\frac32}{\la_{q+1}\mu_q}\\
\norm{\as D_t\as\ph_{O2}}_{N-1} 
&\lec \norm{\as D_t (\theta (w_0\otimes w_c) + \theta (w_c\otimes w_0))}_N \\
& \qquad + \norm{\as D_t [(\theta (w_c\otimes w_c) +\theta_c (w_0\otimes w_0))]}_N
 \lec \la_{q+1}^N\de_{q+1}^\frac12 \cdot \frac {\de_{q+1}^\frac32}{\la_{q+1}\mu_q}.
\end{align*}
Therefore, combining the estimates, we get
\begin{align*}
\norm{\as\ph_{O}}_N\leq \frac 1{5} \la_{q+1}^{N-4\ga} \de_{q+2}^\frac32, \quad 
\norm{\as D_t\as\ph_{O}}_{N-1}\leq \frac 1{5} \la_{q+1}^{N-4\ga} \de_{q+1}^\frac12\de_{q+2}^\frac32,
\end{align*}
for sufficiently small $b-1>0$ and large $\la_0$. 

\subsubsection{\bf Reynolds current error} 
Recall that $\as\ph_R = R_{q+1}^{\text{pre}} \otimes_s w$. Using the estimates for the Reynolds stress, we have
\begin{align*}
\norm{\as\ph_R}_N 
&\lec \sum_{N_1+N_2=N} \norm{R_{q+1}^{\text{pre}}}_{N_1}\norm{w}_{N_2} \lec \la_{q+1}^N 
\la_q^\frac12\la_{q+1}^{-\frac12} \de_q^\frac14\de_{q+1}^\frac54
\leq \frac 15 \la_{q+1}^{N-4\ga} \de_{q+2}^\frac32,\\
\norm{\as D_t \as\ph_R}_{N-1}
&\lec \norm{D_{t,\ell} \as\ph_R}_{N-1} + \norm{(w+u_q-u_\ell)\cdot\na \as\ph_R}_{N-1}\\
&\lec  \sum_{N_1+N_2=N-1} 
\norm{D_{t,\ell} R_{q+1}^{\text{pre}}}_{N_1}\norm{w}_{N_2} 
 + \norm{ R_{q+1}^{\text{pre}}}_{N_1}\norm{D_{t,\ell} w}_{N_2} \\
& + \sum_{N_1+N_2=N-1} 
(\norm{w}_{N_1} + \norm{u_q-u_\ell}_{N_1})\norm{\as\ph_R}_{N_2+1} \\
&\lec \la_{q+1}^N\de_{q+1}^\frac12 \la_q^\frac12\la_{q+1}^{-\frac12} \de_q^\frac14\de_{q+1}^\frac54
\leq \frac 15 \la_{q+1}^{N-4\ga}\de_{q+1}^\frac12 \de_{q+2}^\frac32
\end{align*}
for sufficiently small $b-1>0$ and large $\la_0$.

\subsubsection{\bf Mediation current error} 
Recall that
\begin{align*}
\as\ph_{M}= ( (R_{q+1}^{\text{pre}} -R_q + \theta w) \otimes_s (u_q - u_\ell))+ (\Phi_q-\Phi_\ell)
\end{align*}
Then we recall that
\begin{align}\label{rep:low.freq.app1}
\theta w -R_q + R_{q+1}^{\text{pre}} + (u_q -u_\ell) \theta + (\vr_q -\vr_\ell) w=
(\theta_0 w_0-R_\ell )-\as R_T -\as R_N-\as R_{O1},
\end{align}
Also, a straightforward calculation, in view of \eqref{e:rep3}, reveals that
\begin{align*}
\norm{\theta_0 w_0-R_\ell}_N \lesssim \la_{q+1}^N \de_{q+1}, \quad \norm{\as D_t(\theta_0 w_0-R_\ell)}_{N-1} \lesssim \la_{q+1}^N \de_{q+1} \de_{q+1}^{\frac12}
\end{align*}
Therefore, using the Reynolds stress estimates, \eqref{est.r.dif}, \eqref{est.v.dif}, and \eqref{est.ph.dif} we conclude
\begin{align*}
\norm{((u_q - u_\ell) \otimes_s (\theta w -R_q + R_{q+1}^{\text{pre}})}_{N} + \norm{\Phi_q-\Phi_\ell}_N
&\leq \la_{q+1}^N\la_q^\frac12\la_{q+1}^{-\frac12} \de_q^\frac14\de_{q+1}^\frac54,\\
\norm{\as D_t[((u_q - u_\ell) \otimes_s (\theta w -R_q + R_{q+1}^{\text{pre}} )]}_{N-1}
+ \norm{\as D_t (\Phi_q-\Phi_\ell)}_{N-1}
&\leq \la_{q+1}^N\de_{q+1}^\frac12\la_q^\frac12\la_{q+1}^{-\frac12} \de_q^\frac14\de_{q+1}^\frac54.
\end{align*}
Combining all the above estimates, we conclude 
\begin{align*}
\norm{\as\ph_{M}}_N \leq \frac 15 \la_{q+1}^{N-4\ga} \de_{q+2}^\frac32, \quad 
\norm{\as D_t  \as\ph_{M}}_{N-1} \leq \frac 15 \la_{q+1}^{N-4\ga}\de_{q+1}^\frac12 \de_{q+2}^\frac32.
\end{align*}

\subsubsection{\bf Estimates on $\Upsilon_1, \Upsilon_2 $ and $\Upsilon_3$} As before, we have the stronger estimate
\begin{align}
\norm{\Upsilon_i'}_0 &\leq \frac{1}{5{(T+ \tau_0)}} \la_{q+1}^{-4\ga} \de_{q+2}^\frac32, \quad \text{for} \,\, i = 1,2,3. \label{e:est_rho_4}
\end{align}
from which the estimates follow by integration
\begin{align}
\norm{\Upsilon_i}_0 &\leq \frac{1}{5} \la_{q+1}^{-4\ga} \de_{q+2}^\frac32.
\end{align}
Again, we use usual strategies to achieve this by taking advantage of Lemma \ref{phase} and the representations \eqref{rep.W}-\eqref{e:rep4}.

\section{Proofs of the key inductive propositions}
\label{induc}

\subsection{Proof of Proposition \ref{ind.hyp}}
For a given inhomogeneous Euler-Reynolds flow $(\vr_q, u_q, p_q,R_q,\Phi_q, S_q)$ on the time interval $[0,T]+\tau_{q-1}$, we can construct the corrected density, velocity and pressure as $\vr_{q+1} = \vr_q + \theta_{q+1}$, $u_{q+1} = u_q + w_{q+1}$ and $p_{q+1} = p_q - \de_{q+1} \vr_{q}$, where $\theta_{q+1}$ is defined by \eqref{rep.theta} and $w_{q+1}$ is defined by \eqref{def.w} on $[0,T]+\tau_{q}$. Moreover,  $(\vr_{q+1},u_{q+1},p_{q+1},R_{q+1},\Phi_{q+1},S_{q+1})$ is a inhomogeneous Euler-Reynolds flow and the error $(R_{q+1}, \Phi_{q+1}, S_{q+1})$ satisfies \eqref{est.R}-\eqref{est.S} at the $(q+1)$th level as desired. Finally, let us denote the absolute implicit constant in the estimate {\eqref{est.w.indM}} for $\theta$ and $w$ by $M_0$ and define $M = 2M_0$. Then, one can easily see that
\begin{align*}
\norm{\vr_{q+1}- \vr_q}_0
+ \frac 1{\la_{q+1}} \norm{\vr_{q+1} - \vr_q}_1 
&= \norm{\theta_{q+1}}_0 
+\frac 1{\la_{q+1}} \norm{\theta_{q+1}}_1 
\leq 2M_0 \de_{q+1}^\frac12  = M \de_{q+1}^\frac12, \\
\norm{u_{q+1}- u_q}_0
+ \frac 1{\la_{q+1}} \norm{u_{q+1} - u_q}_1 
&= \norm{w_{q+1}}_0 
+\frac 1{\la_{q+1}} \norm{w_{q+1}}_1 
\leq 2M_0 \de_{q+1}^\frac12  = M \de_{q+1}^\frac12. 
\end{align*}
Moreover, since we know that $\vr_{q+1} = \vr_{{q}} + \theta$, and $\inf \theta > -M_0 \de_{q+1}^{1/2} $, we see that \eqref{cauchy_01} holds. 
Also, using \eqref{est.vp} and \eqref{est.r}, we have 
\begin{align*}
\norm{\vr_{q+1}}_0
&\leq \norm{\vr_q}_0 + \norm{\theta_{q+1}}_0
\leq 5-\de_q^\frac12 + M_0\de_{q+1}^\frac12 
\leq 5- \de_{q+1}^\frac12,\\
\norm{u_{q+1}}_0
&\leq \norm{u_q}_0 + \norm{w_{q+1}}_0
\leq 5-\de_q^\frac12 + M_0\de_{q+1}^\frac12 
\leq 5- \de_{q+1}^\frac12,\\
\norm{\vr_{q+1}}_N
&\leq \norm{\vr_q}_N + \norm{\theta_{q+1}}_N
\leq M \la_q^N\de_q^\frac12 + \frac12 M \la_{q+1}^N \de_{q+1}^\frac12
\leq M\la_{q+1}^N \de_{q+1}^\frac12,\\
\norm{u_{q+1}}_N
&\leq \norm{u_q}_N + \norm{w_{q+1}}_N
\leq M \la_q^N\de_q^\frac12 + \frac12 M \la_{q+1}^N \de_{q+1}^\frac12
\leq M\la_{q+1}^N \de_{q+1}^\frac12,\\
\norm{(\partial_t +u_{q+1}\cdot \nabla) \vr_{q+1}}_{N-1}& 
\leq \norm{(\partial_t +u_q\cdot \nabla) \vr_q}_{N-1} + \|w \cdot \nabla \vr_q\|_{N-1}\nonumber\\
& \quad + \norm{(\partial_t +u_\ell\cdot \nabla) \theta}_{N-1} + \|(u_{q+1} - u_\ell)\cdot \nabla \theta\|_{N-1}\nonumber\\
& \lesssim \tau_q^{-1} \lambda_q^{N-1}  \delta^{1/2}_q + \tau_q^{-1} \lambda_{q+1}^{N-1}  \delta^{1/2}_{q+1} + \lambda_{q+1}^{N} \delta_{q+1} \le M \tau_{q+1}^{-1} \lambda_{q+1}^{N-1}  \delta^{1/2}_{q+1}, \\
\norm{(\partial_t +u_{q+1}\cdot \nabla) u_{q+1}}_{N-1}& 
\leq \norm{(\partial_t +u_q\cdot \nabla) u_q}_{N-1} + \| w \cdot \nabla u_q\|_{N-1}\nonumber\\
& \quad + \norm{(\partial_t +u_\ell\cdot \nabla) w}_{N-1} + \|(u_{q+1} - u_\ell)\cdot \nabla w \|_{N-1}\nonumber\\
& \lesssim \tau_q^{-1} \lambda_q^{N-1}  \delta^{1/2}_q + \tau_q^{-1} \lambda_{q+1}^{N-1}  \delta^{1/2}_{q+1} + \lambda_{q+1}^{N} \delta_{q+1} \le M \tau_{q+1}^{-1} \lambda_{q+1}^{N-1}  \delta^{1/2}_{q+1}.
\end{align*}
for $N=1,2$, provided that $\la_0$ is sufficiently large. This proves that we have constructed a corrected flow $(\vr_{q+1},u_{q+1},p_{q+1},R_{q+1},\Phi_{q+1},S_{q+1})$ at $(q+1)$th level. 


\subsection{Verification of additional induction estimates}

First note that, for any two differential operators $A$ and $B$ and any $n,m > 0$, the Leibniz rule implies that
\begin{subequations}
\begin{align}
(A+B)^m
&= \sum_{k=1}^{m}\sum_{\substack{\alpha, \beta \in {\mathbb N}^k \\ |\alpha| + |\beta| =m}} \left(\prod_{i=1}^{k} A^{\alpha_i}\, B^{\beta_i}\right),
\label{eq:RNC1} \\
A^n := ((w + (u_q -u_\ell)) \cdot \nabla)^n
&= \sum_{j= 1}^{n} f_{j, n} D^{j}, \quad \text{where}\,\, f_{j,n}:= \sum_{\substack{\xi \in {\mathbb N}^n \\ |\xi| =n-j}} c_{n, j, \xi} 
\prod_{\ell=1}^{n} (D^{\xi_\ell} (w +(u_q -u_\ell))),
\label{eq:DNC1}
\end{align}
\end{subequations}
where $c_{n, j,\xi}$ are certain explicitly computable combinatorial coefficient  which depends only on the factors inside the parentheses and $D^a$ represents $\partial^{\alpha}$ for some multi-index $\alpha$ with $|\alpha| =a$. The proof is based on induction on $n$ and $m$.

We begin by estimating $f_{j,n}$. We have, by~\eqref{est.r} and ~\eqref{est.u.new}, that
\begin{align*}
    \|D_{t,\ell}^M f_{j,n}\|_N \lec \la_{q+1}^{N+n-j} \de_{q+1}^{n/2} \tau_q^{-M}.
\end{align*}
Note that we need to verify estimates \eqref{est.new} at $(q+1)$-th level, i.e., need estimates for $\| \as D_{t}^M \varrho_{q+1}\|_N$, $\| \as D_{t}^M u_{q+1}\|_N$ and $\| \partial_{t}^{M} u_{q+1}\|_N$ for all $0 \neq N \le N_{*}$ and $2 \le M\le M_{*}$. Recall that $ \as D_t u_{q+1} = \partial_t u_{q+1} + (u_{q+1}\cdot \nabla) u_{q+1}$. We verify only the velocity estimates, as the density estimates can be treated analogously. To verify the estimates, we first write 
$$
\as D_{t}^M u_{q+1}= (D_{t,\ell} + (w + (u_{q} -u_\ell)) \cdot \nabla)^M u_q + (D_{t,\ell} + (u_{q+1} -u_\ell) \cdot \nabla)^M w.
$$
Let us choose $A=((w + (u_{q} -u_\ell))\cdot \nabla)$, $B=D_{t,\ell}$, and apply the Leibniz rule \eqref{eq:RNC1}- \eqref{eq:DNC1} to calculate 

\begin{align}
&\|(D_{t,\ell} + (w + (u_{q} -u_\ell)) \cdot \nabla)^M u_q \|_N
\lesssim \sum_{k=1}^{M}\sum_{\substack{\alpha, \beta \in {\mathbb N}^k \\ |\alpha| + |\beta| =M}} \Big\|\big( \prod_{i=1}^{k} ((w + (u_{q} -u_\ell)) \cdot \nabla)^{\alpha_i}\, D_{t,\ell}^{\beta_i}\big) u_q \Big\|_N \label{ONE} \\
&\lesssim \sum_{k=1}^{M}\sum_{\substack{\alpha, \beta \in {\mathbb N}^k \\ |\alpha| + |\beta| =M}} \Big\|\Big( \prod_{i=1}^{k} \big(\sum_{j=1}^{\alpha_i} f_{j,\alpha_i} D^j \big)\, D_{t,\ell}^{\beta_i}\Big) u_q \Big\|_N \nonumber \\
 & \lesssim \sum_{\substack{\alpha, \beta \in {\mathbb N}^k \\ |\alpha| + |\beta| =M}} 
\lambda_{q+1}^{N} (\lambda_{q+1})^{\sum (\alpha_i-1)} (\delta_{q+1}^{1/2})^{\sum \alpha_i} \la_q \delta_q^{1/2} ({\tau_q^{-1}})^{\sum \beta_i} + \lambda_{q+1}^{N} (\delta_{q+1}^{1/2})^{\sum \alpha_i} (\la_{q+1})^{\sum \alpha_i} ({\tau_q^{-1}})^{\sum \beta_i} \delta_q^{1/2} \nonumber \\
& \lesssim \sum_{\substack{\alpha, \beta \in {\mathbb N}^k \\ |\alpha| + |\beta| =M}} 
\lambda_{q+1}^N \la_q\delta_q^{1/2} (\lambda_{q+1} \delta_{q+1}^{1/2})^{\sum \beta_i}  \delta_{q+1}^{1/2} (\la_{q+1} \de_{q+1}^{1/2})^{\sum \alpha_i-1}   + \lambda_{q+1}^{N} (\la_{q+1} \delta_{q+1}^{1/2})^{\sum (\alpha_i + \beta_i)} \delta_q^{1/2} \nonumber \\ 
& \le \lambda_{q+1}^N (\tau_{q+1}^{-1})^M \delta_{q+1}^{1/2}. \nonumber
\end{align}
Here we have used the fact that $\ell \la_q \de_q^{1/2} \le \de_{q+1}^{1/2}$, and $\| D^M_{t,\ell}  u_{q}\|_N \le \ell^{1-N} \lambda_q \delta_q^{1/2} (\tau_q^{-1})^M$. Indeed, this can be verified using the mollification estimates \eqref{est.final01}, and \eqref{est.new}:
\begin{align*}
&\| D^M_{t,\ell}  u_{q}\|_N = \| (D_{t} + (u_\ell -u_q) \cdot \nabla)^M u_q \|_N \\
& \le \sum_{\substack{\alpha, \beta \in {\mathbb N}^k \\ |\alpha| + |\beta| =M}} 
\ell^{-N} (\ell \la_q \de_q^{1/2})^{\sum \alpha_i}
(\ell)^{(1-\sum \alpha_i)} \la_q \delta_q^{1/2} (\tau_q^{-1})^{\sum \beta_i} + \sum_{\substack{\alpha, \beta \in {\mathbb N}^k \\ |\alpha| + |\beta| =M}} 
\ell^{-N} (\ell \la_q \de_q^{1/2})^{\sum \alpha_i}
(\la_q)^{(\sum \alpha_i)} \delta_q^{1/2} (\tau_q^{-1})^{\sum \beta_i}\\
& \le \ell^{1-N} \lambda_q \delta_q^{1/2} (\tau_q^{-1})^M.
\end{align*}
Similarly, we can estimate the other term $(D_{t,\ell} + (u_{q+1} -u_\ell) \cdot \nabla)^M w$ to conclude
$$
\|(D_{t,\ell} + (u_{q+1} -u_\ell) \cdot \nabla)^M w\|_N \le \lambda_{q+1}^N (\tau_{q+1}^{-1})^M \delta_{q+1}^{1/2}.
$$
To verify the other estimate, we write $\partial_{t}^{M} u_{q+1} =  \partial_{t}^{M} u_{q} +  \partial_{t}^{M} w$ and conclude
\begin{align*}
\|\partial_{t}^M u_{q+1} \|_N & \le \|\partial_{t}^M u_{q} \|_N + \|\partial_{t}^M w \|_N \\
&\le \|\partial_{t}^M u_{q} \|_N + \| (D_{t, \ell} -u_\ell \cdot \nabla )^M w\|_N 
\le \lambda_{q+1}^N (\tau_{q+1}^{-1})^M \delta_{q+1}^{1/2},
\end{align*}
thanks to Remark~\ref{rem.new} below.

\begin{rem}
\label{rem.new}
We can also use the Leibniz rule \eqref{eq:RNC1}- \eqref{eq:DNC1} to verify the following estimates for $N\ge1$:
$$
\| D^M_{t, \ell}  \varrho_{\ell}\|_N \le \ell^{1-N} \lambda_q \delta_q^{1/2} (\tau_q^{-1})^M, \quad 
\| D^M_{t, \ell}  u_{\ell}\|_N \le \ell^{1-N} \lambda_q \delta_q^{1/2} (\tau_q^{-1})^M.
$$
In order to do so, we first observe that
$$
D^M_{t, \ell} u_{\ell} = D^M_{t, \ell} (u_{\ell} - u_q) + D^M_{t, \ell} u_{q} = D^M_{t, \ell} (u_{\ell} - u_q) + (D_{t} + (u_\ell -u_q) \cdot \nabla)^M u_q.
$$
Recall that, thanks to mollification estimates \eqref{est.final01} in Appendix~\ref{apenC}, we have the required estimates $\|D^M_{t, \ell} (u_{\ell} - u_q)\|_N \le \ell^{1-N} \lambda_q \delta_q^{1/2} (\tau_q^{-1})^M$. For the remaining term, the desired bound is already available.  Hence, combining the two estimates yields the desired result. Moreover, since $u_\ell$ is is obtained by convolution only in the spatial variable, we have
 $\| \partial_t^M u_{\ell}\|_N \le \| \partial_t^M u_{q}\|_N \le  \lambda^N_q \delta_q^{1/2} (\tau_q^{-1})^M$. The density estimates can be treated analogously.
\end{rem}

\subsection{Proof of Proposition \ref{p:ind_technical}}
We start with a given time interval $\cal I \subset (0,T)$ with $|\cal I| \geq 3\tau_q$. Then, it is possible to find a $p_0$ such that $\supp(\th_{p_0}(\tau_q^{-1}\cdot))\subset  \cal I$. In case $I = (p_0, k, h)$ belongs to $\mathscr{I}_S$, then we simply replace $a_{I}$ in $w_{q+1}$ by $\td a_{I} =-a_{I}$ so that $\Ga_{I}$ by $\td \Ga_{I} = -\Ga_{I}$. Otherwise, we just keep the same $a_{I}$. It is easy to see that $\td\Ga_{I}$ still solves \eqref{eq.Ga1} and therefore, $\td a_{I}$ satisfies \eqref{eq.Ga}. Let us also set $\td p_{q+1} = p_{q+1}$ and observe that above replacements does not alter the estimates for $\Ga_{I}$ used in the proof of Lemma \ref{lem:est.coe}, the corresponding coefficients $\td a_{p,m}$, $\td b_{p,m}$, $\td c_{p,m}$, $\td d_{p,m}$, $\td e_{p,m}$ and $\td g_{p,m}$ satisfy \eqref{est.b}-\eqref{est.g}, and $\td w = \td w_o$, $\td w_c$, and $\td w_{q+1}$, generated by them, also satisfy \eqref{est.Wc}-\eqref{est.r}. Therefore, the inhomogeneous Euler-Reynolds flow $(\td \vr_{q+1}, \td u_{q+1}, \td p_{q+1}, \td R_{q+1}, \td \Phi_{q+1}, \td S_{q+1})$ at the $(q+1)$th level satisfies \eqref{est.vp22}-\eqref{est.new}, and \eqref{cauchy_01}-\eqref{cauchy} as desired. Notice also that, by construction, the perturbation $\td w_{q+1}$ differs from $w_{q+1}$ on the support of $\th_{p_0}(\tau_q^{-1} \cdot)$. Therefore, we conclude that 
\[
\supp_t(u_{q+1} - \td u_{q+1})
=\supp_t(w_{q+1} - \td w_{q+1})
\subset \cal I. 
\] 
Furthermore, by \eqref{eq.Ga} and \eqref{eq.Ga1}, we have 
\begin{align*}
\sum_{I\in \mathscr{I}_S} a_I^2 |(\na \xi_I)^{-1}h_I|^2
&=\tr\left(
(\na \xi_I)^{-1}  \sum_{f\in \cH_{I,R}} a_I^2 h_I\otimes h_I [(\na \xi_I)^{-1}]^\top
\right)\\
&= \tr (\de_{q+1}\I - S_\ell - \widetilde{M})
\end{align*} 
where 
\begin{align*}
\widetilde{M} & = \sum_{p',k'}\th_{I'}^2\chi_{I'}^2(\xi_{I'})\sum_{h'\in \mathscr{I}_R} a_{I'}^2\left(\dint_{\T^3}\psi_{I'}^2dx \right) (\na \xi_{I'})^{-1} h' \otimes  (\na \xi_{I'})^{-1} h' \\
& \quad +  \sum_{p',k'}\th_{I'}^2\chi_{I'}^2(\xi_{I'})\sum_{h'\in \mathscr{I}_\Phi} a_{I'}^2\left(\dint_{\T^3}\psi_{I'}^2dx \right) (\na \xi_{I'})^{-1} h' \otimes  (\na \xi_{I'})^{-1} h' 
\end{align*}
Therefore, we simply repeat the proof in section \ref{subsec:R} to conclude
\[
\norm{\widetilde{M}}_0 \lec \la_q^{-2\ga} \de_{q+1}.
\] 
Then, it follows that
\begin{align*}
|w_0 - \td w_0 |^2
&=  \sum_{I\in \mathscr{I}_S: p_I = p_0}
4\th_I^2(t) \chi_I^2(\xi_I) a_I^2 |(\na \xi_I)^{-1} h_I|^2 (1+(\psi_I^2(\la_{q+1} \xi_I)-1))\\
&= \sum_{I\in \mathscr{I}_S: p_I = p_0}
4\th_I^2\chi_I^2(\xi_I)  (3\de_{q+1} - \tr(S_\ell) - \tr(\widetilde{M}))\\&\quad+
\sum_{m\in \Z^3\setminus\{0\}}\sum_{I\in \mathscr{I}_S: p_I = p_0}
4\th_I^2\chi_I^2(\xi_I) a_I^2 |(\na \xi_I)^{-1} h_I|^2
\dot{c}_{I,k} e^{i\la_{q+1} m\cdot \xi_I}\\
&= 4\th_{p_0}^6(\tau_q^{-1}t) (3\de_{q+1}- \tr S_\ell - \tr {\widetilde{M} }) + \sum_{m\in \Z^3\setminus\{0\}} 4\de_{q+1}\tr(\td c_{p_0,m}^S)e^{i\la_{q+1}m\cdot \xi_I},
\end{align*}
where 
\[
\tr(\td c_{p_0,m}^S) = \sum_{I\in \mathscr{I}_S: p_I = p_0}
\th_I^2(t) \chi_I^2(\xi_I) \de_{q+1}^{-1} a_I^2\dot{c}_{I,m} |(\na \xi_I)^{-1} h_I|^2.
\]
Following the steps leading to the estimate \eqref{est.d} for $c_{p_0,m}$, we can conclude the estimate $\norm{\tr(\td c_{p_0,m}^S)}_N \lec \mu_q^{-N} |\dot{c}_{I,m}|$ holds for $N=0,1,2$. Then
\begin{align*}
\norm{w_0 - \td  w_0}_{C^0([0,T]; L^2(\T^3))}^2
&\geq 4(2\pi)^3(3\de_{q+1} - \norm{S_\ell}_0 - \norm{\tr(\widetilde{M})}_0) \\
&\quad- \sum_{m\in \Z^3}4 
\left| 
\de_{q+1} \int\tr(\td c_{p_0,m}^S)e^{i\la_{q+1}m\cdot \xi_I} dx
\right|\\
&\geq 12\de_{q+1} - C\de_{q+1}(\la_q^{-2\ga} + \la_q^{-2\ga} + (\la_{q+1} \mu_q)^{-2})\\
&\geq 4\de_{q+1}
\end{align*}
for sufficiently large $\la_0$. Indeed, we can use Lemma \ref{phase} to get
\begin{align*}
\sum_m\left| \int \tr(\td c_{p_0,m}^S) e^{\la_{q+1} m\cdot \xi_I} dx\right|
&\lec \sum_m\frac{\norm{\tr(\td c_{p_0,m}^S)}_2 + \norm{\tr(\td c_{p_0,m}^S)}_0\norm{{\na} \xi_I}_{C^0([t_{p_0}-\frac12\tau_q, t_{p_0}+\frac32\tau_q];C^2( \T^3))}}{\la_{q+1}^2 |m|^2}
\\
&\lec (\la_{q+1}\mu_q)^{-2} \sum_m\frac{|\dot{c}_{I,m}|}{|m|^2}
\lec (\la_{q+1}\mu_q)^{-2} 
\left(\sum_m |\dot{c}_{I,m}|^2\right)^\frac12
\left(\sum_m \frac{1}{|m|^4}\right)^\frac12.
\end{align*}
Therefore, we obtain
\begin{align*}
\norm{u_{q+1} - \td u_{q+1} }_{C^0([0,T]; L^2(\T^3))}
&=\norm{w_{q+1}- \td w_{q+1}}_{C^0([0,T]; L^2(\T^3))}\\
&\geq \norm{w_o - \td  w_o}_{C^0([0,T]; L^2(\T^3))}
-(2\pi)^\frac32(\norm{w_c}_0 + \norm{\td w_c}_0)\\
&\geq  2\de_{q+1}^\frac12 -\frac{(2\pi)^\frac32 2M_0}{\la_{q+1}\mu_q} \de_{q+1}^\frac12 \geq \de_{q+1}^\frac12
\end{align*}
for sufficiently large $\la_0$. 

Finally, let us assume that a inhomogeneous Euler-Reynolds flow $(\td \vr_q, \td u_q, \td p_q, \td R_q, \td \Phi_q, \td S_q)$ satisfies \eqref{est.vp22}-\eqref{est.new} and
\[
\supp_t(\vr_q- \td \vr_q, u_q-\td u_q, p_q-\td p_q, R_q-\td R_q, \Phi_q-\td \Phi_q, S_q-\td S_q)\subset \cal{J} 
\]
for some time interval $\cal{J}$. We may now construct the regularized flow, $\td R_\ell$, $\td \Phi_\ell$, and $\td S_\ell$ as we constructed for $R_\ell$, $\Phi_\ell$, and $S_\ell$ and notice that they differ only in $\cal J + l_t\subset\cal J + (\la_q\de_q^\frac12)^{-1}$. Consequently, $w_{q+1}$ and $\td w_{q+1}$ differ only in $\cal J + (\la_q\de_q^\frac12)^{-1}$ and hence the inhomogeneous Euler-Reynolds flows $(\vr_{q+1}, {u}_{q+1},  p_{q+1},  R_{q+1}, \Phi_{q+1}, S_{q+1})$ and $(\td \vr_{q+1}, \td u_{q+1},  \td p_{q+1}, \td R_{q+1}, \td \Phi_{q+1}, \td S_{q+1})$ satisfy 
\[
\supp_t(\vr_{q+1}-\td \vr_{q+1}, u_{q+1}-\td u_{q+1}, p_{q+1}-\td p_{q+1}, R_{q+1}-\td R_{q+1}, \Phi_{q+1}-\td \Phi_{q+1}, S_{q+1}-\td S _{q+1})\subset \cal{J} + (\la_q\de_q^\frac12)^{-1}.
\] 

 \appendix
 \section{Some technical lemmas}
This appendix contains a collection of auxiliary lemmas which are used throughout the paper. The proof of the following lemma can be found in \cite[Appendix]{BDLSV2020}.
 \begin{lem}[H\"{o}lder norm of compositions] \label{lem:est.com} Let $H:\Omega \to \R$ and $\Psi: \R^n \to \Omega$ be two smooth functions for some $\Omega\subset \R^d$. Then, for each $N\in \N$, we have
\begin{align}
&\norm{\na^N (H\circ \Psi)}_0
\lec \norm{\na H}_0 \norm{\na\Psi}_{N-1} + \norm{\na H}_{N-1} \norm{\Psi}_0^{N-1}\norm{\Psi}_N \nonumber\\
&\norm{\na^N (H\circ \Psi)}_0
\lec \norm{\na H}_0 \norm{\na\Psi}_{N-1} + \norm{\na H}_{N-1} \norm{\na\Psi}_0^{N} \label{chain},
\end{align}
where the implicit constant in the inequalities depends only on $n$, $d$, and $N$.
 \end{lem}
To define the new triple $(R_q, \Phi_q, S_q)$ we need to ``invert the divergence'' of vector fields and tensors. For this purpose, we recall the inverse divergence operator introduced in \cite{DLSz2013}.
\begin{defn}[Inverse divergence operator]\label{idv.defn}
For any $g\in C^\infty(\T^3;\R^3)$, the inverse divergence operator is defined by
\begin{align}
\label{oneone}
	(\cal{R}g)_{ij}= \cal{R}_{ijk} g_k 
	= -\frac 12 \De^{-2} \pa_{ijk} g_k + \frac 12\De^{-1} \pa_k g_k \de_{ij} - \De^{-1}\pa_i g_j - \De^{-1}\pa_jg_i. 
\end{align}
\end{defn}
\begin{rem} The image of the divergence free operator $\cal{R}g(x)$ is designed to be a trace-free symmetric matrix at each point $x$ and to solve 
	\[
	\div(\cal{R}g) = g - \langle g \rangle\, .
	\]
\end{rem}
In our analysis we also need an inverse divergence operator which maps a mean-zero scalar function to a mean-zero vector-valued one. To this end, we abuse the notation and define 
\[
(\cR f)_i = \De^{-1} \pa_i f. 
\]
Indeed, $\div \cR f =f - \langle f \rangle$.

Moreover, we state the following version of the stationary phase lemma, who proof, using standard Schauder estimates, can be found in \cite{DaSz2016}.
 \begin{lem}\label{phase}
 Let $\alpha \in (0,1)$, and $N\geq 1$. Suppose that $b\in C^\infty(\T^3)$ and $\xi\in C^\infty(\T^3;\R^3)$ satisfies
 \[
 \frac 1{C} \leq |\na \xi| \leq C,
 \]
 for some constant $C>1$. Then
 \[
 \left| \int_{\T^3} b(x) e^{ik\cdot \xi} dx \right|
 \lec \frac{\norm{b}_N + \norm{b}_0\norm{{\na} \xi}_{N}}{|k|^N} ,
 \]
and for the operator $\mathcal{R}$ defined in \eqref{oneone}, we have
\[
\| \mathcal{R} \left(b(x) e^{ik\cdot \xi} \right) \|_{\alpha}
\lec \frac{\|b\|_{0}}{|k|^{1-\alpha}} + \frac{\norm{b}_{N+\alpha} + \norm{b}_0\norm{{\na} \xi}_{N +\alpha}}{|k|^{N-\alpha}} ,
\]
where the implicit constants depend on $C$, $\alpha$ and $N$, but not on $k$. 
 \end{lem}
 
\begin{lem}[Commutator estimate] Let $g$ and $h$ be in $C^\infty([0,T]\times \T^3)$ and set $g_\ell = P_{\leq \ell^{-1} } g$, $h_\ell = P_{\leq \ell^{-1}} h$ and $(gh)_\ell =
 P_{\leq \ell^{-1}} (gh)$. Then, for each $N \geq 0$, the following holds,
\begin{align}
\norm{g_\ell h_\ell  - (gh)_\ell }_N \lec_N  \ell^{2-N} \norm{g}_1\norm{h}_1. \label{est.com} 
\end{align}
 \end{lem}

\begin{lem}[Commutator estimate]
\label{lem:com2} 
For any $N\geq 0$, we have
\begin{align}
&\norm{ [u_\ell\cdot \na, {P}_{\le \ell^{-1}}]H}_N \lec \ell^{1-N}\norm{\na u}_0\norm{\na H}_0\label{est.com1}\\
&\norm{[u_\ell\cdot\na, {P}_{> \ell^{-1}}] H}_{N}\lec \ell^{1-N}\norm{\na u}_0 \norm{\na H}_0.  \label{est.com2}
\end{align} 
\end{lem}

\section{A microlocal lemma}\label{sec.ml}
In order to estimate different types of new errors, we shall make use of the following microlocal lemma. For a proof of the lemma, we refer to \cite[Lemma 8.1]{Kwon1}.

\begin{lem}[Microlocal Lemma]
\label{mic} 
Let $Z$ be a Fourier multiplier defined on $C^\infty(\T^3)$ by
\[
\crF[Zh](\zeta) = \mathfrak{m}(\zeta)\crF[h](\zeta), \quad \forall \zeta\in \Z^3
\]
for some $\mathfrak{m}$ which has an extension in $\cal{S}(\R^3)$ (which we continue to denote by $\mathfrak{m}$). Then, for any $n_0\in \N$, $\la>0$, and any scalar functions $a$ and $\xi$ in $C^\infty(\T^3)$, 
$Z(a e^{i\la \xi})$ can be decomposed as
\begin{equation*}
\begin{split}
Z(a e^{i\la \xi}) 
=&\   \left[a \mathfrak{m}(\la\na \xi)
+\sum_{k=1}^{2{n_0}} C_{k}^\la(\xi,a): (\na^k \mathfrak{m})(\la \na\xi)
+\e_{n_0}(\xi,a)\right]e^{i\la \xi}
\end{split}
\end{equation*}
for some tensor-valued coefficient $C_{k}^\la(\xi,a)$ and a remainder $\e_{n_0}(\xi, a)$ which is specified in the following formula: 
\begin{equation}\label{def.eN}\begin{split}
\e_{n_0}(\xi, a)(x)
&= 
\sum_{\substack{n_1+n_2\\=n_0+1}} \frac{(-1)^{n_1}c_{n_1,n_2}}{n_0!} \\
&\cdot\int_0^1\int_{\R^3} \widecheck{\mathfrak{m}}(y) e^{-i\la \na\xi(x)\cdot y} ((y\cdot\na)^{n_1}a)(x-ry) e^{i\la Z[\xi](r)}\be_{n_2} [\xi](r) (1-r)^{n_0} dydr,
\end{split}
\end{equation}
where $c_{n_1,n_2}$ is a constant depending only on $n_1$ and $n_2$, and the function $\be_n[\xi]$ is
\begin{align*}
\be_{n} [\xi](r)
&=B_{n} (i\la Z'(r), i\la Z''(r), \cdots, i\la Z^{(n)}(r)),\\
Z{[\xi]}(r) &=Z{[\xi]}_{x,y}(r) = r\int_0^1(1-s) (y\cdot\na )^2 \xi(x-rsy) ds,
\end{align*}
with $B_n$ denoting the $n$th complete exponential Bell polynomial.
\end{lem}
In fact, this lemma has an important consequence on the anti-divergence operator $\mathcal{R}$.
\begin{cor}\cite[Corollary 8.2]{Kwon1}\label{cor.mic2} Let $N=0,1,2$ and $F = \sum_{m\in \Z^3\setminus\{0\}}\sum_{p\in \Z} a_{p,m}e^{i\la_{q+1}m\cdot\xi_m}$. Assume that a function $a_{p,m}$ fullfills the following requirements. 
\begin{enumerate}[(i)]
\item The support of $a_{p,m}$ satisfies $\supp(a_{p,m}) \subset (t_p - \frac 12\tau_q, t_p + \frac 32\tau_q)\times \R^3$. In particular, 
for $p$ and $p'$ neither same nor adjacent, we have
\begin{equation}\label{dis.amk}
\supp(a_{p,m}) \cap \supp(a_{p', m'}) = \emptyset, \quad \forall m, m' \in \Z^3\setminus\{0\}.
\end{equation}
\item For any $j\geq 0$ and $(p,m)\in \Z\times \Z^3$,  
\[
\norm{a_{p,m}}_j+ (\la_{q+1}\de_{q+1}^\frac 12)^{-1}\norm{D_{t,\ell} a_{p,m}}_j\lec_j  \mu_q^{-j} |\dot{a}_m|, \quad
\sum_m |m|^{n_0+2} |\dot{a}_m| \leq  a_F, 
\]
for some sequence $\{\dot{a}_m\}$ and some $a_F>0$, where $ n_0 = {\ceil{\frac{2b(2+\al)}{(b-1)(1-\al)}}}$ and {$\norm{\cdot}_{j} = \norm{\cdot}_{C(\mathcal{I}; C^j(\T^3))}$ on some time interval $\mathcal{I}\subset \R$. }
\end{enumerate}
Then, {for any $b>1$, we can find $\La_0(b)$ such that for any $\la_0\geq \La_0(b)$,} $\cR F$ satisfies the following inequalities: 
\begin{align*}
\norm{\cR F}_{N} \lec \la_{q+1}^{N-1} a_F, \quad
\norm{\as D_t \cR F}_{N-1} \lec \la_{q+1}^{N-1}\de_{q+1}^\frac12 a_F\,
\end{align*}
upon setting $\as D_t = (\pa_t + u_{q+1} \cdot \na )$.
\end{cor}

\section{Material derivatives estimates}
\label{apenC}
In this section, we briefly recall several lemmas which in turn help to verify material derivative estimates for the velocity field. For the details, we refer to \cite{Giri01, Giri02, Giri03}.

\subsection{Mollification estimates}
We first recall the following algebraic identity.
Let $v$ be a sufficiently smooth divergence-free vector field and let $D_t = \pa_t + v\cdot \na$ be the material derivative operator associated to $v$. For any sufficiently smooth function $F = F(x,t)$ and any $n,m\geq 0$, the Leibniz rule implies that
\begin{subequations}
\begin{align}
D^n D_t^m F
&=D^n (\partial_t + v \cdot \nabla_x)^m F
= \sum_{\substack{m'\leq m \\ n'+m' \leq n+m}} d_{n,m,n',m'}(v)(x,t) D^{n'} \partial_t^{m'} F \, ,
\label{eq:RNC} \\
d_{n,m,n',m'}(v)
&= \sum_{k = 0}^{m-m'} \sum_{\substack{\{ \gamma \in {\mathbb N}^k \colon |\gamma| = n-n'+k,\\ \beta \in {\mathbb N}^k \colon |\beta| = m-m'-k\}}} c(m,n, k, \gamma, \beta)  \prod_{\ell=1}^{k} \left( D^{\gamma_\ell} \partial_t^{\beta_\ell}  v(x,t) \right) \, ,
\label{eq:DNC}
\end{align}
\end{subequations}
where $c(m,n, k, \gamma, \beta)$ denotes an explicitly computable combinatorial coefficient  which depends only on the factors inside the parentheses. Identities \eqref{eq:RNC}--\eqref{eq:DNC} hold because $D$ and $\partial_t$ commute; the proof is based on induction on $n$ and $m$.

Let us now assume that a function $f$ satisfies the estimates \eqref{est.new}. Recall that $f_\ell (\theta):= P_{\le \ell^{-1}} f(\theta) = P_{\le 2^J} f(\theta) = \int_{\R^3} f(\theta + 2^{-J} \kappa) \,\widecheck{m}(\kappa)d\kappa$, where $J \in \N$ is the maximum number satisfying $2^J \le \ell^{-1}$, and $\int_{\R^3} \widecheck{m}(\kappa)d\kappa = m(0)=1$. Assume that $v$ also satisfies the estimates \eqref{est.new}. 

Expanding $f$ in a Taylor series in space around $\theta$ yields the formula
\begin{align*}
    f(\theta + 2^{-J} \kappa) = f(\theta) + \sum_{|\alpha|=1} \frac{1}{\alpha!} (2^{-J} \kappa)^{(\alpha)} \int_0^1  D^\alpha f (\theta+\eta 2^{-J} \kappa) \,d\eta \, .
\end{align*}
Therefore,
\begin{align*}
f_\ell(\theta) - f(\theta)
&= \sum_{|\alpha|= 1} \frac{1}{\alpha!}  \int_{\R^3} \widecheck{m}(\kappa) (2^{-J}\kappa)^{(\alpha)} \int_0^1  D^\alpha  f (\theta+\eta 2^{-J} \kappa) \,d\eta \, d\kappa \, .
\end{align*}
Now we appeal to the identity \eqref{eq:RNC} with $F = f_\ell - f$ to obtain
\begin{align*}
\norm{D^n D_t^m (f_\ell - f)}_{L^\infty(\T^3)}                                  \les \sum_{\substack{m'\leq m \\ n'+m' \leq n+m}} \norm{d_{n,m,n',m'}(v)}_{L^\infty(\T^3)} \norm{ D^{n'} \partial_t^{m'} (f_\ell - f)}_{L^\infty(\T^3)} \, . 
\end{align*}
From given assumptions on $f$, $v$ and the formula \eqref{eq:DNC}, we have that (with $\lambda=\la_q, \Tau^{-1} = \tau_q^{-1}$)
\begin{align*}
\norm{d_{n,m,n',m'}(v)}_{L^\infty(\T^3)} \lesssim \sum_{k=0}^{m-m'}  \lambda^{ n-n' +k} \de_q^{1/2}(\Tau^{-1})^{m-m' -k} \lec \lambda^{ n-n'} (\Tau^{-1})^{m-m'}.
\end{align*} 
Combining this estimate, we deduce that
\begin{align}
&\left\|D^n D_{t}^m (f_\ell - f)\right\|_{L^\infty(\T^3)} \notag\\
&\qquad\les \sum_{\substack{m'\leq m \\ n'+ m' \leq n+m}} \lambda^{(n-n')} (\Tau^{-1} )^{(m-m')} \de_q^{1/2}\norm{ D^{n'} \partial_t^{m'} (f_\ell - f)}_{L^\infty((a,b) \times\T^3)} 
\notag\\
&\qquad\les \sum_{\substack{m'\leq m \\ n'+m' \leq n+m}} \sum_{|\alpha| = 1} \la^{(n-n')} (\Tau^{-1})^{(m-m')}  \de_q^{1/2} \times \lambda_q^{n' + |\alpha| } \de_q^{1/2} (\Tau^{-1})^{m'}                                
\underbrace{\int_{\R^3} \abs{(2^{-J} \kappa)^{(\alpha)}} |\widecheck{m}(\kappa)      | d\kappa}_{\approx \ell^{|\alpha|}}
\notag\\
&\qquad\les  \sum_{|\alpha| = 1} \la^{n}  \lambda_q^{ |\alpha|} \de_q (\Tau^{-1})^{m}  \ell^{|\alpha|} 
\lesssim \la_q^n \de^{1/2}_q (\ell \la_q \de_q^{1/2}) (\tau_q^{-1})^{m} \lesssim \ell^{1-n} \la_q \de_q^{1/2} (\tau_q^{-1})^m. \label{est.final01}
\end{align}

\subsection{Commutator with material derivatives}
\label{comm01}
Here we recall a useful ingredient for bounding commutators of Eulerian and material derivatives in the following lemma.

\begin{lem}
\label{lem:Komatsu}
Let $m, n \geq 0$. Then we have that the commutator of $D_t^m$ and $D^n$ is given by
\begin{align}
 \left[ D_t^m, D^n \right] = \sum_{\{ \alpha \in \N^n \colon 1 \leq |\alpha|\leq m \} } \frac{m!}{\alpha! (m-|\alpha|)!} \left( \prod_{\ell = 1}^{n} (\ad D_t)^{\alpha_\ell}(D) \right) D_t^{m-|\alpha|}.
 \label{eq:Komatsu}
\end{align}
By the product in \eqref{eq:Komatsu} we mean the product/composition of operators 
\[
\prod_{\ell = 1}^{n} (\ad D_t)^{\alpha_\ell}(D) = (\ad D_t)^{\alpha_n}(D) (\ad D_t)^{\alpha_{n-1}}(D) \ldots (\ad D_t)^{\alpha_1}(D)
\,,
\]
so that on the right side of \eqref{eq:Komatsu} we have a sum of differential operators of order at most $n$.
\end{lem}

For the above lemma to be useful, we need to be able to characterize the operator $(\ad D_t)^a (D)$. 

\begin{lem}
\label{lem:ad:Dt:a:D}
Let $a \in \N$. Then the order $1$ differential operator $(\ad D_t)^a (D)$ may be expressed as
\begin{align}
(\ad D_t)^a (D) =  \sum_{k=1}^{a}  \sum_{\{ \beta \in \N^k \colon |\beta| = a-k \} } c_{a,k,\beta}   \prod_{j=1}^{k} (D_t^{\beta_j} D v) \cdot \nabla
\label{eq:ad:Dt:a:D}
\end{align}
where the $\prod$ in \eqref{eq:ad:Dt:a:D} denotes the product of matrices, $c_{a,k,\beta}$ are coefficients which depend only on $a,k$, $\beta$.
\end{lem}


\subsection{Application of Fa\`a di Bruno formula}
\label{Faadi}

Let us recall the following version of the multivariable Fa\`a di Bruno formula \cite[Theorem 2.1]{ConstantineSavits96}. Let $g= g(x_1,\ldots,x_d) = f ( h(x_1,\ldots,x_d))$, where $f \colon \R^m \to \R$, and $h \colon \R^d \to \R^m$ are $C^n$ smooth functions of their respective variables. Let $\alpha \in {\mathbb N}_0^d$ be s.t. $|\alpha|=n$, and let $\beta \in \N_0^m$ be such that $1\leq |\beta|\leq n$. We then define  
\begin{align} 
p(\alpha,\beta) 
&= \Bigg\{ (k_1,\ldots,k_n; \ell_1,\ldots,\ell_n) \in (\N_0^m)^n \times (\N_0^d)^n \colon \exists s \mbox{ with }1\leq s \leq n \mbox{ s.t. } \notag \\
&\qquad   \qquad  |k_j|, |\ell_j| > 0 \Leftrightarrow 1 \leq j \leq s, \, 0 \prec \ell_{1} \prec \ldots \prec \ell_s,  \sum_{j=1}^s k_j =  \beta, \sum_{j=1}^s |k_j| \ell_j = \alpha \Bigg\}.
\label{eq:Faa:di:Bruno:1}
\end{align} 
Then the multivariable Fa\`a di Bruno formula states that we have the equality
\begin{align}
\partial^\alpha g (x) = \alpha! \sum_{|\beta|=1}^n (\partial^\beta f)(h(x)) \sum_{p(\alpha,\beta)} \prod_{j=1}^n \frac{(\partial^{\ell_j} h(x))^{k_j}}{k_j! (\ell_j!)^{k_j}}.
\label{eq:Faa:di:Bruno:2}
\end{align}

\medskip

Let $X(a,t)$ be the flow induced by the vector field $v$, with initial condition $X(a,t) = x$. Denote by $a = X^{-1}(x,t)$ the inverse of the map $X$. We then note that 
\begin{align*}
D_t^M g(x,t) = \left(\partial_t^M ( (g \circ X)(a,t)) \right)\vert_{a = X^{-1}(x,t)}.
\end{align*}
We wish to apply the above with the function $g(x,t) = \psi(\Gamma^{-2}h(x,t))$. 
We apply the Fa\`a di Bruno formula \eqref{eq:Faa:di:Bruno:1}--\eqref{eq:Faa:di:Bruno:2} with the one dimensional differential operator $\partial_t^M$ to the composition $g\circ X$, note that $\partial_t^{\beta_i} (h(X(a,t),t)) = (D_t^{\beta_i} h)(X(a,t),t)$, and then evaluate the resulting expression at $a = X^{-1}(x,t)$, to obtain
\begin{align}
\label{Faadi_01}
D_t^M g(x,t) = M! \sum_{B=1}^{M} \Gamma^{-2B} \psi^{(B)}(\Gamma^{-2} h(x,t)) \sum_{\substack{\{\kappa , \beta \in \N^M \colon \\ |\kappa|= B, \kappa \cdot \beta = M \} }} \prod_{i=1}^{M} \frac{ \left( (D_t^{\beta_i} h)(x,t) \right)^{\kappa_i}}{\kappa_i! (\beta_i !)^{\kappa_i}}.
\end{align}

\section{A localized inverse divergence}
\label{apenD}
This appendix contains a special inverse divergence operator, which is used to deal with the transport current error term in subsection~\ref{transport error}. We simply state the main result and a complete proof can be found in \cite[appendix]{Giri03}.

\begin{proposition}[\bf Inverse divergence iteration step]
	Let $n\geq 2$ be given. Fix a zero-mean $\T^n$-periodic function $\zeta$ and a zero-mean $\T^n$-periodic symmetric tensor field $\vartheta^{(i,j)}$ which are related by {$\zeta  =  \partial_{ij} \vartheta^{(i,j)}$}. Let $\xi_I$ be a volume preserving diffeomorphism of $\T^n$. Define the matrix $A = (\nabla \xi_I)^{-1}$. Given a vector field $G^k$, we have
	\begin{align}
	G^k (\zeta \circ \xi_I)
	= \partial_\ell R^{k \ell} + E^k
	\end{align}
	where the symmetric stress $R^{k\ell}$ is given by
	\begin{align}
	R^{k\ell}
	&= 
	G^k A_i^\ell(\partial_j\vartheta^{(i,j)} \circ \xi_I)
	+G^\ell A_i^k (\partial_j\vartheta^{(i,j)} \circ \xi_I)
	- G^n\pa_n\Phi^m A_i^k   A_j^\ell(\pa_m\vartheta^{(i,j)} \circ \xi_I)
	\, ,
	\end{align}
	and the error term $E^k$ is given by 
	\begin{align}
	E^k
	&= -\pa_\ell (G^\ell A_i^k) (\partial_j\vartheta^{(i,j)} \circ \xi_I)
	- (\pa_\ell G^k)A_i^\ell(\partial_j\vartheta^{(i,j)} \circ \xi_I)
	+ \pa_n (G^\ell A_i^k \pa_\ell\Phi^m) A_j^n(\pa_m\vartheta^{(i,j)} \circ\xi_I)
	\, .
	\end{align}
\end{proposition}

\begin{proposition}[\bf Main inverse divergence operator]
	\label{prop:intermittent:inverse:div}
	Let dimension $n\geq 2$ and Lebesgue exponent $p\in[1,\infty]$ be free parameters. The remainder of the proposition is composed first of \emph{low and high-frequency assumptions}, which then produce a \emph{localized output} satisfying a number of properties.  Finally, the proposition concludes with \emph{nonlocal assumptions and output}. 
	\smallskip
	
	\noindent\textbf{Part 1: Low-frequency assumptions}
	\begin{enumerate}[(i)]
		\item\label{item:cond.G.inverse.div} Let $G$ be a vector field and assume there exist a constant $\const_{G,p} > 0$ and parameters 
		\begin{equation}\label{eq:inv:div:NM}
		N_*\geq M_*\geq 1 \, ,
		\end{equation}
		$M_t$, and $\lambda, \nu,\nu' \geq 1$ such that 
		\begin{align}
		\norm{D^N D_{t}^M G}_{L^p}&\lesssim \const_{G,p} \lambda^N\MM{M,M_{t},\nu,\nu'}
		\label{eq:inverse:div:DN:G}
		\end{align}
		for all $N \leq N_*$ and $M \leq M_*$.
		\item Fix an incompressible vector field $u(t,x):\R\times\T^n\rightarrow \R^n$ and denote its material derivative by $D_t = \partial_t + u\cdot\nabla$. Let $\xi_I$ be a volume preserving diffeomorphism of $\T^n$ such that 
		\begin{align}
		D_t \xi_I= 0 \,
		\qquad \mbox{and} \qquad
		\norm{\nabla \xi_I - \I}_{L^\infty(\supp G)} \leq \frac 12 \,. \label{eq:DDpsi2}
		\end{align}
		Denote by $\xi_I^{-1}$ the inverse of the flow $\xi_I$,  which is the identity at a time slice which intersects the support of $G$.
		Assume that  the velocity field $v$ and the flow functions $\xi_I$ and $\xi_I^{-1}$ satisfy the bounds 
		\begin{subequations}
			\begin{align}
			\norm{D^{N+1}   \xi_I}_{L^{\infty}(\supp G)} + \norm{D^{N+1}   \xi_I^{-1}}_{L^{\infty}(\supp G)} 
			&\les \lambda'^{N}
			\label{eq:DDpsi}\\
			\norm{D^ND_t^M D u}_{L^{\infty}(\supp G)}
			&\les \nu \lambda'^{N}\MM{M,M_{t},\nu,\nu'}
			\label{eq:DDv}
			\,,
			\end{align}
		\end{subequations}
		for all $N \leq N_*$, $M\leq M_*$, and some $\lambda'>0$. 
	\end{enumerate}
	\smallskip
	
	\noindent\textbf{Part 2: High-frequency assumptions}
	\begin{enumerate}[(i)]
		\item\label{item:inverse:i} Let $\zeta \colon \T^n \to \R$ be a zero mean scalar function such that there exists a large positive even integer $\dpot \gg 1$ and a smooth, mean-zero, adjacent-pairwise symmetric tensor potential\footnote{We use $i_j$ for $1\leq j \leq \dpot$ to denote any number in the set $\{1,\dots,n\}$.} $\vartheta^{(i_1,\dots, i_\dpot)}:\T^n \rightarrow \R^{\left(n^\dpot\right)}$ such that $\zeta(x) = \partial_{i_1}\dots\partial_{i_\dpot} \vartheta^{(i_1 \dots i_\dpot)}(x)$.
		\item \label{item:inverse:ii} There exists a parameter $\mu\geq 1$ such that $\zeta$ and $\vartheta$ are $(\frac{\T}{\mu})^n$-periodic.
		\item \label{item:inverse:iii} There exist parameters $1 \ll \Upsilon \leq \Upsilon' \leq \Lambda$, $\const_{*,p}>0$ such that for all $0\leq N \leq {N_*}$ and all $0\leq k \leq \dpot$,
		\begin{align}
		\norm{D^N \partial_{i_1}\dots \partial_{i_k} \vartheta^{(i_1,\dots, i_\dpot)}}_{L^p} \les \const_{*,p} \Upsilon^{k-\dpot} \MM{N, \dpot - k , \Upsilon', \Lambda} \, .
		\label{eq:DN:Mikado:density}
		\end{align} 
		\item\label{item:inverse:iv} There exists $\Ndec$ such that the above parameters satisfy
		\begin{align}
		\lambda', \lambda \ll \mu \leq \Upsilon \leq \Upsilon' \leq \Lambda  \,, \qquad \max(\lambda,\lambda') \Upsilon^{-2} \Upsilon' \leq 1 \, , \qquad N_*-\dpot \geq 2\Ndec + n+1 \, , 
		\label{eq:inverse:div:parameters:0}
		\end{align}
		where by in the first inequality in \eqref{eq:inverse:div:parameters:0} we mean that 
		\begin{align}
		\Lambda^{n+1} \left(\frac{\mu}{2\pi \sqrt{3} \max(\lambda,\lambda')}\right)^{-\Ndec} \leq 1
		\,.
		\label{eq:inverse:div:parameters:1}
		\end{align}
	\end{enumerate}
	\smallskip
	
	\noindent\textbf{Part 3: Localized output}
	\begin{enumerate}[(i)]
		\item\label{item:div:local:0} There exists a symmetric tensor $R$ and a vector field $E$ such that\index{$\divH$}
		\begin{align}
		G \; \zeta\circ \xi_I  &=  \div R + E  
		=: \div\left( \divH \left( G\zeta \circ \xi_I \right) \right) + E \, . \label{eq:inverse:div}
		\end{align}
		We use the notation $R=\divH( G \zeta \circ \xi_I)$ for the symmetric stress.
		\item\label{item:div:local:i} The support of $R$ is a subset of $\supp G \cap \supp \vartheta$.
		\item\label{item:div:local:ii} There exists an explicitly computable positive integer $\const_\divH$, an explicitly computable function $r(j):\{0,1,\dots,\const_\divH\}\rightarrow \mathbb{N}$ and explicitly computable tensors
		\begin{align*}
		&\rho^{\beta(j)} \, , \qquad \beta(j)=(\beta_1,\beta_2,\dots,\beta_{r(j)})\in  \{1,\dots,n\}^{r(j)} \, , \\
		&H^{\alpha(j)} \, , \qquad \alpha(j)=(\alpha_1,\alpha_2,\dots,\alpha_{r(j)},k,\ell) \in \{1,\dots,n\}^{r(j)+2} \, 
		\end{align*}
		of rank $r(j)$ and $r(j)+2$, respectively, all of which depend only on $G$, $\zeta$, $\xi_I$, $n$, $\dpot$, such that the following holds.  The symmetric, localized stress $R$ can be decomposed into a sum of symmetric, localized stresses as\footnote{The contraction is on the first $r(j)$ indices, and the resulting rank two tensor is symmetric.}
		\begin{align}\label{eq:divH:formula}
		\divH^{k\ell} (G \zeta \circ \xi_I) = R^{k\ell} = \sum_{j=0}^{\const_\divH}  H^{\alpha(j)} \rho^{\beta(j)} \circ \xi_I \, .
		\end{align}
		Furthermore, we have that
		\begin{align}
		\supp H^{\alpha(j)} \subseteq \supp G  \, , \qquad \supp \rho^{\beta(j)} \subseteq \supp \vartheta \,  . \label{eq:inverse:div:linear}
		\end{align}
		\item\label{item:div:local:iii} For all $N \leq N_* - \frac \dpot 2$, $M\leq M_*$, and $j\leq \const_\divH$, we have the subsidiary estimates\footnote{In fact it is clear from the algorithm that as $j$ increases, the estimates become much stronger.  For simplicity, we simply record identical estimates for each term which are sufficient for our aims.}
		\begin{subequations}\label{eq:inverse:div:sub:main}
			\begin{align}
			\left\| D^N \rho^{\beta{(j)}} \right\|_{L^p} &\les \const_{*,p} \Upsilon^{-2} \Upsilon' \MM{N,1,\Upsilon',\Lambda} \label{eq:inverse:div:sub:1} \\
			\left\|D^N D_{t}^M H^{\alpha{(j)}}\right\|_{L^p} &\lesssim \const_{G,p} \left(\max(\lambda,\lambda')\right)^N \MM{M,M_{t},\nu,\nu'} \, . \label{eq:inverse:div:sub:2} 
			\end{align}
		\end{subequations}
		\item\label{item:div:local:iv} For all $N \leq N_* - \frac \dpot 2$ and $M\leq M_*$, we have the main estimate
		\begin{align}
		\norm{D^N D_{t}^M R}_{L^p}
		&\les  \const_{G,p} \const_{*,p}  \Upsilon' \Upsilon^{-2} \MM{N,1,\Upsilon',\Lambda} \MM{M,M_{t},\nu,\nu'} 
		\label{eq:inverse:div:stress:1}
		\end{align}
		\item\label{item:div:nonlocal} For $N \leq N_* - \frac \dpot 2 $ and $M\leq M_*$ the error term $E$  in \eqref{eq:inverse:div} satisfies
		\begin{align}
		\norm{D^N D_{t}^M E}_{L^p}  
		\les \const_{G,p} \const_{*,p}   \max(\lambda,\lambda')^{\frac \dpot 2} \left( \Upsilon' \Upsilon^{-2}\right)^{\frac \dpot 2} \Lambda^{N} \MM{M,M_{t},\nu,\nu'} 
		\,.
		\label{eq:inverse:div:error:1}
		\end{align}
	\end{enumerate}
	\smallskip
	
	\noindent\textbf{Part 4: Nonlocal assumptions and output}
	\begin{enumerate}[(i)]
		\item 
		\label{item:nonlocal:v}
		Let $N_\circ, M_\circ$ be integers such that 
		\begin{equation}\label{eq:inv:div:wut}
		1 \leq M_\circ \leq N_\circ \leq \frac{M_*}{2} \, ,
		\end{equation}
		and let $K_\circ$ be a positive integer.\footnote{{$K_\circ$ serves as an extra amplitude gain which will be used later to eat some material derivative losses.}}  Assume that in addition to the bound \eqref{eq:DDv} we have the following global lossy estimates
		\begin{align}
		\norm{D^N \partial_t^M u}_{L^\infty}\les  \const_u \lambda'^N \nu'^M
		\label{eq:inverse:div:v:global}
		\end{align}
		for all  $M \leq M_\circ$ and $N+M \leq N_\circ + M_\circ$, where 
		\begin{align}
		\const_u \lambda' \les \nu' 
		\,.
		\label{eq:inverse:div:v:global:parameters}
		\end{align}
		\item Assume that $\dpot $ is large enough so that\index{$\dpot$}\index{$\divR$}
		\begin{align}
		\const_{G,p} \const_{*,p} \max(\lambda,\lambda')^{\frac \dpot 4} (\Upsilon' \Upsilon^{-2})^{\frac \dpot 4}  \Lambda^{{n}+2+K_\circ} \left(1 + \frac{\max\{ \nu', \const_u \Lambda \}}{\nu 
		}\right)^{M_\circ}
		\leq 1
		\, .
		\label{eq:riots:4}
		\end{align}
	\end{enumerate}
	
	Then we may write  
	\begin{align}
	E = \div \RR_{nonlocal} + \int_{\T^3} G \zeta \circ \xi_I\, dx =: \div \left(\divR(G \zeta \circ \xi_I)\right) + \int_{\T^3} G \zeta \circ\xi_I  \,  dx \, ,
	\label{eq:inverse:div:error:stress}
	\end{align}
	where $\RR_{nonlocal} = \divR(G \zeta \circ \xi_I)$ is a traceless symmetric stress which satisfies
	\begin{align}
	\norm{D^N D_{t}^M \RR_{ nonlocal} }_{L^\infty}  
	\leq  \frac{1}{\Lambda^{K_\circ}} \max(\lambda,\lambda')^{\frac \dpot 4} (\Upsilon' \Upsilon^{-2})^{\frac \dpot 4} \Lambda^N \nu^M
	\label{eq:inverse:div:error:stress:bound}
	\end{align}
	for  $N \leq N_\circ$ and $M\leq M_\circ$.
\end{proposition}

\begin{rem}
\label{rmk01}
Note that we need to specify various parameters involved in the above Proposition~\ref{prop:intermittent:inverse:div} for our setup. Indeed, in order to successfully apply the Proposition~\ref{prop:intermittent:inverse:div}, we choose
$\lambda' = \ell^{-1}, \lambda = \mu_q^{-1}, \nu= \tau_q^{-1}, \nu'= \mu = \Upsilon = \Upsilon' = \Lambda = \lambda_{q+1}$, $\Ndec = 2 n_0 +2$, $M_\circ = 2$, $N_\circ =3$, $M_* = 6$, $\dpot= \ceil{\frac{8(8b+1)}{(b-1)(1-\beta)}}$, $N_* = \dpot + 4 n_0 +8$, and $K_\circ =1$. Moreover, in our setup, we choose $M_t =M_*$, $\const_{*,p}=\const_u= 1$, $\const_{G,p} = \lambda_q \delta_q^{1/2} \delta_{q+1}^{1/2}$ for the Nash term $\as R_{N}$, $\const_{G,p} = \tau_q^{-1} \delta_{q+1}^{1/2} $ for the transport term $\as R_T$, and $\const_{G,p} = \mu_q^{-1}\delta_{q+1}$ for the oscillation term $\as R_{O}$.
\end{rem}

\subsection*{Acknowledgments}
We express our gratitude to the anonymous referee for the careful reading
and valuable suggestions and comments that have improved this manuscript significantly. The second author acknowledges the support of the Department of Atomic Energy,  Government of India, under project no.$12$-R$\&$D-TFR-$5.01$-$0520$, and DST-SERB Swarnajayanti Fellowships grant DST/SJF/MS/$2021$/$44$.

\bibliographystyle{abbrv}

\end{document}